\documentclass[a4paper, 11pt]{amsart}%{article}
\usepackage[margin=1.8cm]{geometry}
\usepackage{anysize}
\usepackage[hang]{subfigure}

\marginsize{3,5cm}{2,5cm}{2,5cm}{2,5cm}%tutaj ustawiasz marginesy
\usepackage{amssymb}
\usepackage{amsmath,amsbsy,amssymb,amscd}
\usepackage{t1enc}
\usepackage[cp1250]{inputenc}
\usepackage[british]{babel}
\usepackage[all]{xy}
\usepackage{color}
\usepackage{amsfonts}
\usepackage{latexsym}
\usepackage{comment}
\usepackage{amsthm}
\usepackage{mathrsfs}
\usepackage{hyperref}
\usepackage{graphicx}

\usepackage{color}
\usepackage{caption}
\usepackage{enumerate}
\usepackage{bbm}
\definecolor{orange}{rgb}{1,0.5,0}

\def\Blue{}%\color{blue}}
\def\OldBlue{}
\def\inx{i_{n,x}}
\usepackage{marginnote}
\edef\marginnotetextwidth{\the\textwidth}
\def\Red{}
\usepackage{mathrsfs} %Script dla zapisu \mathscr{}

\DeclareMathAlphabet{\mathpzc}{OT1}{pzc}{L}{it} %stylizowany
\def\cal#1{{\mathcal #1}} %skrÄ‚Ĺ‚t do calligraphic
\def\a{\alpha}
\def\vep{\varepsilon}

\newtheorem{definition}{Definition}[section]

\newtheorem{proposition}{Proposition}[section]
\newtheorem{theorem}{Theorem}[section]

\newtheorem{corollary}{Corollary}[section]
\newtheorem{remark}{Remark}[section]
\newtheorem{lemma}{Lemma}[section]

\numberwithin{equation}{section}

\def\i{\mathrm{i}}

\def\H{\check{H}}
\def\C{\mathbb{C}}
\def\geq{\geqslant}
\def\leq{\leqslant}
\def\R{\mathbb{R}}
\def\T{\mathbb{T}}
\newcommand{\raz}{\mathbbm{1}}
\def\eps{\varepsilon}

\def\Z{\mathbb{Z}}
\def\N{\mathbb{N}}

\def\Q{\mathbb{Q}}
\def\cB{\mathcal{B}}

\def\cC{\mathcal{C}}

\newcommand{\ov}{\overline}

\newcommand{\bea}{\begin{eqnarray}}
  \newcommand{\eea}{\end{eqnarray}}
  \newcommand{\beab}{\begin{eqnarray*}}
  \newcommand{\eeab}{\end{eqnarray*}}
\renewcommand{\a}{\alpha}
  \newcommand{\be}{\begin{equation}}
  \newcommand{\ee}{\end{equation}}
  
 \newcommand{\zdk}{(Z,\mathcal{D},\kappa)}
 \newcommand{\xbm}{(X,\mathcal{B},\mu)}
 \newcommand{\ycn}{(Y,\mathcal{C},\nu)}

\newcommand{\cD}{\mathcal D}

\newcommand{\bfu}{\boldsymbol{u}}

\newcommand{\mob}{\boldsymbol{\mu}}

\title{On disjointness properties of some parabolic flows}
\author{Adam Kanigowski}
\address{Department of Mathematics, University of Maryland at College
Park, College Park, MD 20740, USA}
\email{adkanigowski@gmail.com}

\author{Mariusz Lema\'nczyk}
\address{Faculty of Mathematics and Computer Science, Nicolaus Copernic
University
Ul. Chopina 12/18
87-100 Torun, Poland}
\email{mlem@mat.umk.pl}

\author{Corinna Ulcigrai}
\address{Institut f{\"u}r Mathematik, Universit{\"a}t  Z{\"u}rich Winterthurerstrasse 190,
CH-8057 Z¨urich} \address{School of Mathematics,  University of Bristol, Howard House, Queens Ave BS8 1SD Bristol, UK} 
\email{corinna.ulcigrai@gmail.com}

\begin{document}
\baselineskip=14pt \maketitle

\begin{abstract}{The Ratner property, a quantitative form of divergence of nearby trajectories, is a central feature in the study of parabolic homogeneous flows. Discovered by Marina Ratner and used in her 1980th seminal works on horocycle flows, it pushed forward the disjointness theory of such systems}. In this paper, exploiting  a recent variation of the Ratner property,  we prove new disjointness phenomena for smooth parabolic flows beyond the homogeneous {world}. In particular, we establish a general disjointness criterion based on the \emph{switchable} Ratner property.  We then apply this new criterion
to study disjointness properties  of  smooth time changes of horocycle flows and smooth Arnol'd flows on $\T^2$, focusing in particular on disjointness of  {distinct} flow \emph{rescalings}. {As a consequence, we answer a question by Marina Ratner on the M\"obius orthogonality of time-changes of horocycle flows. {\OldBlue In fact,  we prove M{\"o}bius orthogonality for all smooth time-changes of horocycle flows and uniquely ergodic realizations of Arnol'd flows considered.}}
\end{abstract}

\begin{flushright}
{\it Dedicated to the memory of Marina Ratner}
\end{flushright}
\tableofcontents

\section{Introduction}
In this paper we study ergodic properties of {\it parabolic} dynamical systems. Classical examples of parabolic systems are horocycle flows acting on unit tangent bundles of compact surfaces with constant negative curvature. In the 1980's, Marina Ratner established many spectacular rigidity phenomena for horocycle flows. In particular, Ratner classified invariant measures and  {joinings (for basic ergodic theory notions, see Section~\ref{s:Sec2}): both invariant measures and joinings have to be {\it algebraic}}. The key phenomenon used in M.\ Ratner's theorems is  the {\it Ratner property} {(originally, in \cite{Rat}, it was called H-property, renamed as the R-property in \cite{Thou})} which describes a special (polynomial) way of divergence of nearby trajectories. In \cite{Rat2}, M.\ Ratner showed that the  R-property persists for  smooth {\it time changes} of horocycle  flows. {This, in \cite{Rat2}, {\OldBlue allowed  her to} show that (smooth) time changes of horocycle flows  share with horocycle flows similar rigidity phenomena}.

There is no formal definition for a system to be parabolic. In view of the above, it seems that the Ratner property is one of  characteristics making a system parabolic. For a long time there were no known examples of systems with {the Ratner property} beyond horocycle flows and their (smooth) time changes. The situation has changed within the last {12} years. K.\ Fr\c{a}czek and the second author showed in \cite{Fr-Le} that there exists a class of flows on $\T^2$ (smooth flows with one singular point) for which a variant of {the R-property} holds.  However, the flows in \cite{Fr-Le} are not (globally) smooth.
The first  {class} of {\OldBlue globally smooth} flows with {{\OldBlue a}  variation of the Ratner property} beyond the horocyclic world {was} given by B.\ Fayad and the first author in \cite{Fa-Ka}. {The class considered in \cite{Fa-Ka} consists of (some)} smooth mixing flows on the torus $\T^2$ with one fixed point (these are sometimes known as Arnol'd and Kochergin flows).
This triggered further investigations and {the first and third authors} and J.\ Ku\l aga-Przymus in \cite{Ka-Ku-Ul} proved {\OldBlue that the same variant of the} Ratner property holds for smooth Arnol'd flows on every surface of genus $g>1$. {\OldBlue The variation of the Ratner property considered} in \cite{Fa-Ka} and \cite{Ka-Ku-Ul} {\OldBlue is the so called {\it switchable Ratner property} (or \emph{SR-property} for short). {\OldBlue For this variant (whose definition is recalled in Section~\ref{Rproperties}), one only requires that the controlled form of divergence for nearby trajectories in Ratner's property  holds \emph{either} in the future \emph{or} in the past, depending on the initial points.} A recent result of G.\ Forni and the first author \cite{ForK} shows that {the SR-property} holds {\OldBlue also} in the class of smooth (non-trivial) time changes of constant type {\it Heisenberg nilflows}.

{\OldBlue All the above mentioned variations  of  the R-property  (which are detailed in Section~\ref{Rproperties}) were defined in order to have the same strong dynamical consequences of the original R-property itself, hence we will sometimes call them simply \emph{Ratner properties}.  In particular, all Ratner properties, as the original R-property does, restrict the type of
self-joinings that the flow can have (see Section 2.5) and allow to enhance mixing properties.}  }

{The Ratner property (or its variations) of a flow imposes some restrictions not only on the set of self-joinings but also on the set of its joinings with another (ergodic) flow.} In particular, we can ask whether for two systems {sharing the same  Ratner property is it} possible to classify joinings between them. The first result in this direction can be found {\OldBlue in Ratner's work~\cite{Rat2}, where she shows} that {two flows $(\bar{h}_t)$ and $(\tilde{h}_t)$ given by two smooth time changes of a horocycle flow $(h_t)$  are disjoint, i.e.\ the only joining between them is the  product measure, whenever the cocycles corresponding to the time changes are not {\it jointly cohomologous} (see Section~\ref{s:tch} for basic definitions)}.

In the present paper, we establish a general disjointness criterion based on the {SR-property} (Theorem~\ref{disjoint.flows}),  which we then use to prove  disjointness properties  of horocycle {flows, their smooth time changes and some Arnol'd flows on $\T^2$, see Theorems~\ref{main:prop},~\ref{main:th2} and~\ref{main:th}}.  The  statement (as well as the proof) of the criterion, which is rather technical, is given  in {Section}~\ref{sec:Sec3}. {We now pass to a description of our main new disjointness results}.

\subsection*{Disjointness for time-changes of horocycle flows}
To state the first of our results, we {need to} recall some basic definitions.
Let $G=SL(2,\R)$ and $\Gamma$ be a lattice in $G$. Let
$$h_t:=\begin{pmatrix} 1&t\\0& 1\end{pmatrix}, {\;t\in\R,}$$
{and, by an abuse of notation, let also $(h_t)$ denote}
the \emph{horocycle flow} acting on $M:=G/ \Gamma$ {(by left multiplication by $h_t$) considered with Haar measure  $\mu$}. For $\tau\in C^1(M)$, $\tau>0$, let $(\widetilde{h}^\tau_t)$ denote the time change of $(h_t)$ given by $\tau$, i.e. for $x\in M$,
$$\widetilde{h}^\tau_t(x)=h_{{u}(x,t)}x,\qquad  \text{where}\ \int_0^{{u}(x,t)}\tau(h_sx)\,ds=t.$$
Let us remark that while {a time change seems to be the easiest form of perturbation of a flow (indeed, it  preserves orbits),  a time change of a horocycle flow} typically breaks the homogeneity of the original flow and  therefore can introduce new spectral and disjointness phenomena (see, for example, the results below). Thus, (smooth) time changes of {horocycle flows} constitute a fundamental class of {non-homogeneous} parabolic flows.

{We will be particularly interested in \emph{rescalings} of a flow $(R_t)$ acting on a probability standard space $\zdk$.
Given a real number $p>0$, by the $p$-\emph{rescaling} of  $(R_t)$ we simply mean the flow $(R_{pt})$ (in which the time was rescaled by the factor $p$). Notice that when $p$ is an integer, the time-one map of $p$-rescaling coincides with the $p$-\emph{power} of the time-one map $R_1$ of the flow, so considering rescalings is an analogous operation to considering the powers of a given transformation. In the case of a horocycle flow (or, more generally, a unipotent flow on a homogenous space), rescalings are also called \emph{algebraic reparametrizations} since the time-changed flow is still of algebraic nature. We are interested in the situation when two rescalings $(R_{pt})$ and $(R_{qt})$, where $p,q>0$ and $p\neq q$, are disjoint, i.e.\ the only joining between them is product measure. Let us recall that the notion of disjointness was introduced by Furstenberg in \cite{Fu} to study how \emph{different} two systems can be. In particular, if two flows are isomorphic, the isomorphism yields a non-trivial joining, so that the two flows cannot be disjoint (in fact, disjoint flows do not have a non-trivial common factor).}

Notice  that {two different} rescalings of {a horocycle flow $(h_t)$ are \emph{never} disjoint}. {Indeed,
if $(g_s)$ denotes the \emph{geodesic flow} given by left multiplication by}
$${g_s}:=\begin{pmatrix} e^s&0\\0& e^{-s} \end{pmatrix}$$
on $SL(2,\R)/\Gamma$,
the renormalization equation
\be\label{renorm} h_t g_s=g_sh_{e^{-2s}t},\qquad \forall \ t,s \in \mathbb{R}, \ee
yields that, for any positive $p\neq q$, {the flows $(h_{pt})$ and $(h_{qt})$} are conjugated by $g_s$ with $s=-\frac{\log(q/p)}{2}$ (and hence are not disjoint).

However, for a large class of smooth time changes of {a} horocycle flow, we have the following result on {disjointness of rescalings}. {In Theorem~\ref{main:th2},
we consider the class} $B^+(M)$ of $W^6$-smooth \emph{positive} functions (time changes)
which have non-trivial support outside of the discrete series (see Section~\ref{sec:deviations} for {the relevant definitions which require some basic notions} from the  representation theory of $SL(2,\R)$).

\begin{theorem}\label{main:prop} Assume that $\Gamma$ is cocompact. {Assume moreover that the cocycle determined by $\tau\in B^+(M)$ is not  a quasi-coboundary}. Then for any {positive} $p,q\in \R\setminus\{0\}$, $p\neq q$, the flow rescalings $(\widetilde{h}^\tau_{pt})$ and $(\widetilde{h}^\tau_{qt})$ are disjoint.
\end{theorem}

Notice that the above theorem does not hold if the cocycle given by $\tau$ is a quasi-coboundary, since in this case, $(\widetilde{h}^\tau_{t})$ is isomorphic to $(h_t)$ and we {have} already remarked that the theorem fails for {any} horocycle flow because of the renormalization equation \eqref{renorm}. Thus, {non-trivial and sufficiently smooth} time changes of $(h_t)$ {have \emph{always} better} disjointness properties than the horocycle flow itself.

Theorem~\ref{main:prop} {is  proved}  using the new disjointness criterion given by Theorem~\ref{disjoint.flows} below, {together} with quantitative estimates on deviations of ergodic averages  {\OldBlue which follows from the works of Flaminio and Forni \cite{Fl-Fo} and Bufetov and Forni \cite{BF} (it can be deduced from the results in \cite{BF}, see Appendix~\ref{app:BF}).}
%(which {\OldBlue can be deduced from} the  main result of Bufetov and Forni from \cite{BF}, see {Section} \ref{sec:deviations}).

Let us remark that the question on possible disjointness
of different rescalings of time changes of a horocycle flow is implicit in \cite{Rat2}.  Indeed, from the results in~\cite{Rat2}, one can deduce (see \cite{Fl-Fo-note} for details)
that, for a smooth time change $(\widetilde{h}^{v}_t)$ of the horocycle flow, the rescalings $(\widetilde{h}^{v}_{pt})$ and $(\widetilde{h}^{v}_{qt})$ are disjoint if and only if the cocycles determined by $v$ and $v\circ g_{-r}$, where $r:=\frac12\log(q/p)$, are not \emph{jointly cohomologous} (see \cite{Rat2} for the definition).

Shortly after proving our result, and motivated by it, L.\ Flaminio and G.\ Forni   informed us that an alternative proof of Theorem~\ref{main:prop}  can be obtained from the above mentioned disjointness result  by Ratner \cite{Rat2},  together with some considerations on the cohomological equation derived from  their work in \cite{Fl-Fo}, and that, it also holds for an arbitrary $v \in W^6(M)$ such that the cocycle generated by $v$ is not quasi-coboundary, see \cite{Fl-Fo-note}. The main result proved  in \cite{Fl-Fo-note} is indeed that this absence of  joint cohomology between the cocycles  determined by $v$ and $v\circ g_{-r}$ holds for $r\neq 0$   whenever the cocycle determined by $v$ is not a quasi-coboundary. Note that the latter result is also \emph{a posteriori} implied by our Theorem~\ref{main:prop} when $v\in B^+(M)$.

We  stress  that our proof of Theorem~\ref{main:prop} does \emph{not} rely on  Ratner's result from~\cite{Rat2} and hence is perhaps more direct. It has furthermore the advantage of showing how deviations of ergodic averages and, in particular, quantitative \emph{shearing phenomena} (through our disjointness criterion) play directly a role in disjointness of rescalings.
While it might be possible to adapt Flaminio and Forni's approach, which  seem to rely essentially on representation theory only, to generalizations to other unipotent flows,  our criterion of disjointness (Theorem~\ref{disjoint.flows}), which requires as an input only dynamical properties of the flow, can  be applied to study  many other parabolic flows with similar quantitative shearing properties \emph{outside} of the homogeneous world  (see the further results in this paper on which we will detail shortly, as well as  the conclusive remark at the end of Introduction mentioning other recent works based on our criterion, see \cite{Be-Ka,Do-Ka, ForK}), thus providing a unifying approach to disjointness questions.

\subsection*{Disjointness properties of Arnol'd flows}
 Another important class of \emph{parabolic flows} is given by \emph{area-preserving flows on surfaces}. Let us define  {the} class of \emph{Arnol'd flows}, which, as we explain below, are flows which arise naturally {when} studying area preserving flows on a surface of genus one. A flow from this class has a special representation (see Section~\ref{s:spflow} for definitions) over an irrational rotation $R_\alpha (x)=x+\alpha$  on $\mathbb{T}$ and under a roof function  $f: \mathbb{T}\to \mathbb{R}_{>0}$ which (identifying $\mathbb{T}$ with $[0,1)$) has the form $$f(x)=-A_-\log x-A_+\log(1-x)+h(x),$$ where $h\in C^5(\T)$ and $A_-\neq A_+$ (both $A_-,A_+$ are positive numbers), see Figure~\ref{Arnoldspecialflow}. We will denote such  flows by
 $((R_\alpha)^f_{{t}})$.

\begin{figure}[h!]
{    \includegraphics[width=0.5\textwidth]{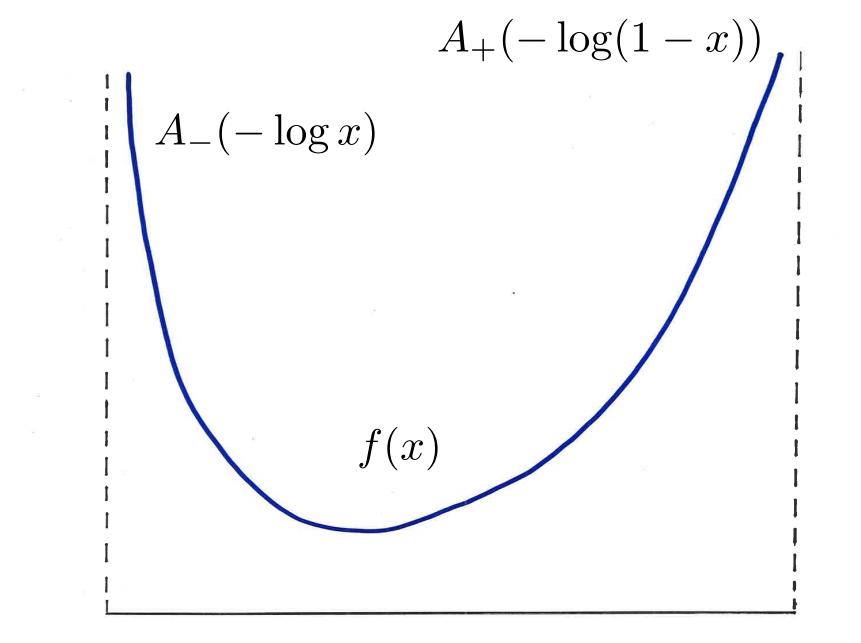}}
\centering{\caption{
Special representation of an {Arnol'd} flow. \label{Arnoldspecialflow}}}

\end{figure}

% \begin{figure}[h!]
  %\end{figure}

{A} motivation to study Arnol'd flows comes from considering  smooth area-preserving flows on the torus $\T^2$, with the simplest possible critical points, namely a {center} and a simple saddle (as in Figure~\ref{Arnoldflow}).   These flows are  also known as {\it multi-valued} or {\it locally} Hamiltonian flows (see in particular the works by Novikov and its school \cite{No}) since in suitable coordinates they are locally given by $\stackrel{\cdot}x=\frac{\partial H}{\partial y}$, $\stackrel{\cdot}y=-\frac{\partial H}{\partial x}$, where $H$ is a local Hamiltonian. The minimal components of a \emph{typical} such flow   {consist} of invariant circles (which foliates a neighbourhood of the center, see Figure~\ref{Arnoldflow}) and a non-trivial \emph{minimal component} (supported on a torus minus a disk). {The restriction $(\phi_t)$  of the flow to the minimal component  admits} a special representation of the form described above. {All this} was first remarked by Arnol'd in the paper \cite{Ar} (from which the name \emph{Arnol'd flows} comes), where mixing of  {$(\phi_t)$  was} conjectured.

 \begin{figure}[h!]
 \subfigure[on the torus \label{Arnoldflowtorus}]{
  \includegraphics[width=0.39\textwidth]{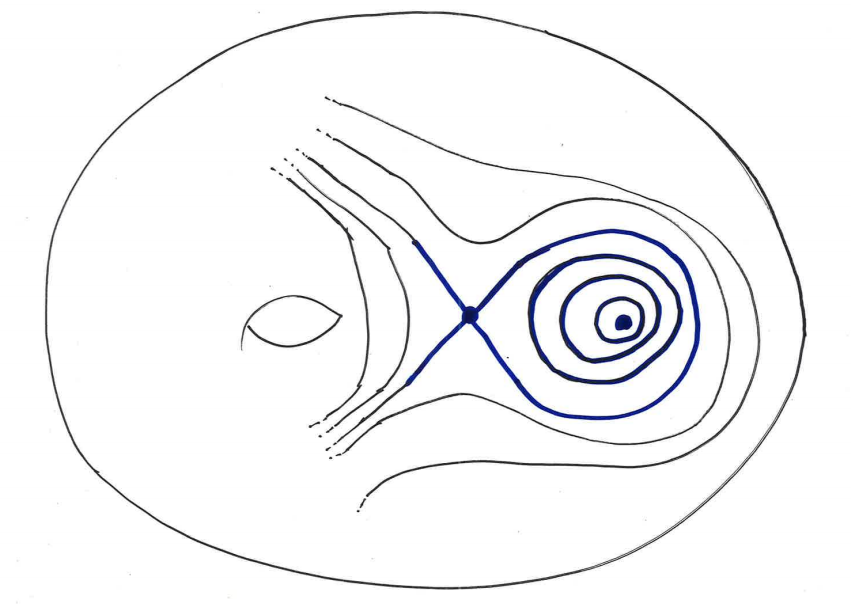} 	} \hspace{9mm}
\subfigure[on $\mathbb{R}^2/\mathbb{Z}^2$ \label{Arnoldflowsquare}]{\includegraphics[width=0.33\textwidth]{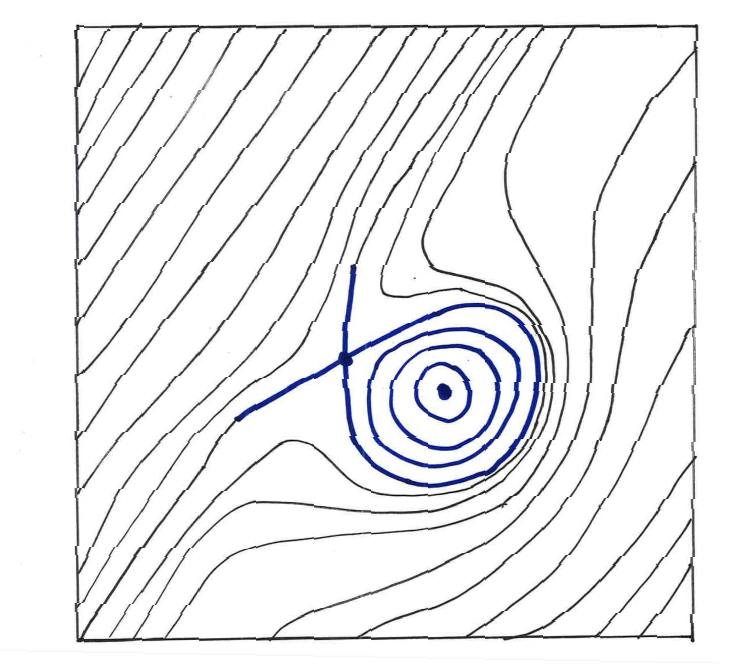}   }	
 \caption{A locally Hamiltonian flow on $\mathbb{T}^2$ with a periodic orbits island and a minimal component (Arnol'd flow). \label{Arnoldflow}}
\end{figure}
% \smallskip

Our next result {shows disjointness} of rescalings of Arnol'd flows:
\begin{theorem}\label{main:th2} There exists a full Lebesgue measure set  $\mathcal{D}\subset \T$ such that for every $\alpha\in \mathcal{D}$ and every $p,q>0$, $\frac{p}{q}\notin \left\{1,\frac{A_-}{A_+},\frac{A_+}{A_-}\right\}$, the flows  $((R_\alpha)_{{pt}}^f)$ and $((R_\alpha)_{{qt}}^f)$ are disjoint.
 \end{theorem}
\noindent The proof of Theorem~\ref{main:th2} {which is done in Section~\ref{s:Sec7} (and uses the results of Sections~\ref{s:Sec6})}   is also based on the disjointness criterion given by Theorem~\ref{disjoint.flows}.

Theorems~\ref{main:th2} and~\ref{main:prop} show an interesting phenomenon: {the Ratner} property in the non-homogeneous setting yields \emph{stronger} dynamical consequences than in the homogeneous one. The reason (which will be explained in more details in Section~\ref{s:Sec5}) is that {the Ratner property} for {a} horocycle flow (hence in the homogeneous setting) depends only on the distance between the initial points and not on their position {in the space}. In all other examples of {a Ratner} property (in particular, for {smooth} non-trivial time changes of horocycle flows)  divergence depends both on their distance and, what is more important, on their position in the space. This allows one to get stronger rigidity conclusions for the set of joinings.

\smallskip

Finally, our third main result, which is once again based on the disjointness criterion given by Theorem~\ref{disjoint.flows},
  shows that the two classes of parabolic flows {considered} so far, namely smooth time changes of horocycle flows and Arnol'd flows, are typically disjoint:
\begin{theorem}\label{main:th}There exists a full Lebesgue measure set of rotation numbers $\mathcal{D}\subset \T$ such that for every $\alpha\in \mathcal{D}$, {every horocycle flow $(h_t)$ (with $\Gamma$ cocompact)} and every  $\tau\in C^1(M)$, the Arnol'd flow $((R_\alpha)_{{t}}^f)$ and the time-change flow $(\widetilde{h}^\tau_t)$ are disjoint.
\end{theorem}
The proof of {Theorem~\ref{main:th}} is at the heart of Section~\ref{s:Sec9}.  Notice that
since the result holds for \emph{all} smooth time changes, it implies in particular  disjointness between (homogenous) horocycles flows  and area-preserving flows in the class of Arnol'd flows considered in Theorem~\ref{main:th2} (sets $\mathcal{D}$ in both theorems are the same).

\subsection*{Consequences for  M\"obius disjointness.}
We now briefly discuss a relation between Theorems~\ref{main:prop} and~\ref{main:th2}  with the \emph{M\"obius disjointness problem} \cite{Sa} which is lately  under intensive study, see e.g.\ survey \cite{Fe-Ku-Le}. If $(T_t)$ is a continuous flow on a compact metric space $X$ then it is called {\em M\"obius disjoint} if
\be\label{sar1}
\lim_{N\to\infty}\frac1N\sum_{n\leq N}F(T_{nt}x)\mob(n)=0\ee
for all $F\in C(X)$, $t\in\R$ and $x\in X$, {where} $\mob$ stands for the classical M\"obius function (whose definition is recalled in {Section}~\ref{s:Sarnak}).
Sarnak conjectured in \cite{Sa}  that all zero entropy homeomorphisms are M\"obius disjoint. As  recent development shows, Sarnak's conjecture is very likely equivalent to the famous Chowla conjecture on correlations of the M\"obius function, see \cite{Fe-Ku-Le} for details.

M\"obius disjointness for horocycle flows has been proved by Bourgain, Sarnak and Ziegler in \cite{Bo-Sa-Zi}. A few years ago M. Ratner asked about M\"obius disjointness of (smooth) time changes of horocycle flows (see Question~7 in \cite{Fe-Ku-Le}).
Since the flows appearing in Theorems~\ref{main:prop} and~\ref{main:th2}  are totally ergodic (in fact, they are known to be mixing,  see respectively \cite{SK} for Arnol'd flows and \cite{Ma} for time changes of horocycle flows),  in view of the so called K\'atai-Bourgain-Sarnak-Ziegler criterion for orthogonality in \cite{Bo-Sa-Zi} (recalled in
 {Section}~\ref{s:Sarnak}),
our results in particular imply  that M\"obius disjointness holds for all smooth time-changes in $B^+(M)$ of horocycle flows and for all uniquely ergodic models of  Arnol'd flows (see {Section}~\ref{s:Sarnak} for more details).  In particular, Theorem~\ref{main:prop} answers positively Ratner's question on M\"obius disjointness (Question~7 in \cite{Fe-Ku-Le}).

A question which has seen a recent surge of interest is \emph{uniform convergence} in M\"obius disjointness, namely whether the convergence  in~\eqref{sar1} is \emph{uniform} in $x \in {X}$. It is for example shown in \cite{Ab-Ku-Le-Ru} that Sarnak's conjecture is equivalent to its uniform convergence form.
While on the level of M\"obius disjointness we cannot see any difference between horocycle flows and their (non-trivial) time changes, this situation changes drastically when we ask about {uniformity}.
Indeed, Theorem~\ref{main:prop} not only answers positively Ratner's question on M\"obius disjointness, but it also yields {uniform convergence} (in $x\in M$) in~\eqref{sar1}. On the other hand, the question of whether  uniform convergence holds  for horocycle flows themselves remains largely open, see Question~5 in \cite{Fe-Ku-Le}. Thus, this is another instance, where in the non-homogeneous set up one can prove stronger properties than in the homogeneous world.

In the case of  Arnol'd flows and the corresponding locally Hamiltonian flows, however, while we  show M\"obius disjointness of any uniquely ergodic realization, we have been  unable until now to show uniformity of convergence in~\eqref{sar1}.
Uniform convergence for M\"obius {disjointness}, especially in its stronger form of so called convergence on (typical) short intervals, will be discussed in some detail in Section~\ref{s:Sarnak}.

\medskip
We conclude by saying that the (parabolic) disjointness criterion (Theorem~\ref{disjoint.flows}) seems to be a general tool in studying disjointness of systems with {a Ratner} property. A variant of this criterion is used in a recent work of G.\ Forni and the first author \cite{ForK}, where disjointness properties of time changes of Heisenberg nilflows are studied and also in  recent works of P.\ Berk and the first author \cite{Be-Ka} and C.\ Dong and the first author \cite{Do-Ka} for getting disjointness of some classes of {interval exchange transformations (IET's)} and translation flows.

\smallskip

\subsection*{Outline of the paper} The structure of the paper is as follows. In Section~\ref{s:Sec2}, we recall some basic definitions and properties, in particular about flows and special flows, joinings and irrational rotations on the circle.   We also recall, for {the} convenience of the reader, the  definition of the SR-property (a recent variation of the classical Ratner property), which is not directly used, but {implicitly} drives many of the exploited phenomena. In Section~\ref{sec:Sec3}, we state and prove the new disjointness criterion (Theorem~\ref{disjoint.flows}). Section~\ref{s:Sec4} is devoted to the proof {of} Theorem~\ref{main:prop} on disjointness for time changes of horocycle flows which gives the first illustration of the use of our criterion. In Section~\ref{s:Sec5}, we state a disjointness criterion in the language of special flows and we then show how it can be deduced from  {our} main criterion {on} disjointness. The proof of the  disjointness result for Arnol'd flows, i.e.\ Theorem~\ref{main:th2}, takes Sections~\ref{s:Sec6} and~\ref{s:Sec7}: in Sections~\ref{s:Sec6} we {first} prove some estimates on the growth of Birkhoff sums of derivatives of the roof function, then in Sections~\ref{s:Sec7} we use them in order to quantify shearing and derive the splitting and slow drift of nearby orbits, properties which are needed to apply the disjointness criterion.  In Section~\ref{s:Sec9}, we show that Arnol'd flows from the considered class are disjoint with smooth time changes of horocycle flows (Theorem~\ref{main:th}). Finally, in Section~\ref{s:Sarnak} we discuss consequences of our main results to the problem of M\"obius disjointness. Three appendices follow. In the first two, we give the proofs of two technical results which are used in the proof of Theorem~\ref{main:prop}: in the first one, Appendix~\ref{app:horocycleR}, we prove some quantitative shearing properties for time-changes of horocycle flows; in  Appendix~\ref{app:BF} we state  a result on ergodic averages of a special family of smooth functions which can be deduced from \cite{BF} (we thank G.\ Forni for the proof). Finally, {in  Appendix~{\ref{s:Sec10}},} we present a  missing link in the literature, namely, the equivalence  between two different {kinds} of convergence  on short intervals.

\section{Preliminaries: definitions, notation and some basic facts}\label{s:Sec2}
\subsection{Flows, joinings and disjointness}\label{sec:back1}
Assume that $\zdk$ is a probability standard Borel space. If $Z$ is additionally a metric space (with a metric $d$), then we will also write $(Z,\mathcal{D},\kappa,d)$ to emphasize the role of $d$. By ${\rm Aut}\zdk$, we denote the set of all automorphisms of $\zdk$, i.e.\ $R\in{\rm Aut}\zdk$ if $R:Z\to Z$ is a bi-measurable, measure-preserving $\kappa$-a.e.\ bijection. Each $R\in{\rm Aut}\zdk$ determines a unitary operator, also denoted by $R$, on $L^2\zdk$ given by $Rf:=f\circ R$ for $f\in L^2\zdk$. Endowed with the weak operator topology of unitary operators, ${\rm Aut}\zdk$ becomes a Polish group.

We will be mainly deal with {\em flows}, i.e.\ with measurable, measure-preserving $\R$-actions $(R_t)_{t\in\R}\subset{\rm Aut}\zdk$. Measurability means that the map $(z,t)\to R_tz$ is measurable (in fact, it is equivalent to saying that the unitary representation $t\mapsto R_t$ is continuous). If no confusion arises, we will  abbreviate notation $(R_t)_{t\in\R}$ to $R_t$. By the {\em centralizer} $C(R_t)$ of the flow $R_t$, we mean the set of all $W\in{\rm Aut}\zdk$ that commute with the flow: $WR_t=R_tW$ (for each $t\in\R$) (in fact, $C(R_t)$ is a closed subgroup of ${\rm Aut}\zdk)$.  It is well-known (see e.g.\ \cite{Co-Fo-Si}) that if a flow $R_t$ is ergodic (i.e.\ its only invariant sets are either of zero or full measure) then for Lebesgue a.e.\ $s\in\R$, the time automorphism $R_s$ is ergodic.  Given $0\neq r\in \R$, by $R_{rt}$ we denote the flow $t\mapsto R_{rt}$.

Assume that $T_t,S_t$ are flows on probability standard Borel spaces $\xbm$ and $\ycn$, respectively. A \emph{joining} between $T_t$ and $S_t$ is a $T_t\times S_t$-invariant (for each $t\in\R$) probability measure on $(X\times Y,\mathcal{B}\otimes\mathcal{C})$ with the projections $\mu$ and $\nu$, respectively.
By $J(T_t,S_t)$ we denote the set of {\rm joinings} between the flows $T_t$ and $S_t$.
Following Furstenberg \cite{Fu}, we say that $T_t$ and $S_t$ are {\em disjoint}, and we write $T_t\perp S_t$, if product measure $\mu\otimes\nu$ is the only member of $J(T_t,S_t)$. Note that if $T_t\perp S_t$ then at least one of the two flows must be ergodic. If both flows are ergodic, then by $J^e(T_t,S_t)$ we denote the subset of {\em ergodic joinings},  that is, of those $\rho\in J(T_t,S_t)$ which make the flow $(T_t\times S_t)_{t\in\R}\subset{\rm Aut}(X\times Y,\mathcal{B}\otimes\mathcal{C},\rho)$ ergodic (this set is always nonempty, as the ergodic decomposition of any joining yields a.e.\ ergodic component also a joining). Note also that if $W$ yields an isomorphism of $T_t$ and $S_t$ then the formula
$$
\rho(B\times C):=\mu(B\cap W^{-1}C),\;\;B\in\mathcal{B},C\in\mathcal{C},$$
determines a member of $J(T_t,S_t)$ which is ergodic if $T_t$ (hence $S_t$) is ergodic.  Note also that, for each $r\neq0$, we have $J(T_t,S_t)=J(T_{rt},S_{rt})$ with the equality $J^e(T_t,S_t)=J^e(T_{rt},S_{rt})$ whenever the flows $T_t$ and $S_t$ are ergodic.

Unless it is stated otherwise, flows are assumed to be {\em free} $\R$-actions, i.e.\ $t\mapsto T_tx$ is 1-1 for $\mu$-a.e.\ $x\in X$. If $T_t$ is ergodic and  $S_t=Id$ (for each $t\in\R$, so $S_t$ is not a free $\R$-action) then $T_t\perp S_t$. In fact, whenever an automorphism $R\in {\rm Aut}\zdk$ is ergodic then $R$ is disjoint from the identity (defined on an arbitrary probability, standard Borel space). We recall here that joinings for automorphisms are defined similarly as for flows -- we replace $\R$-actions by $\Z$-actions, in fact, joinings can be defined for actions of more general groups, see e.g.\ \cite{Gl} for more details.

\begin{remark}\label{r:disTA}\em
Although, in general, ergodic properties of a time automorphism and the whole flow are different (e.g.\ a time automorphism can have more factors, or can have a larger centralizer), there is no difference if we consider disjointness of time automorphisms. More precisely, given $r\neq0$, the flows $T_{rt}$ and $S_{rt}$ are disjoint if and only if the time automorphisms $T_r$ and $S_r$ are disjoint. Indeed, the $\Z$-subaction $(T_{rn}\times S_{rn})$ is cocompact, so if $\rho\in J(T_r,S_r)$ then $\frac1r\int_0^r \rho\circ(T_s\times S_s)\,ds\in J(T_{rt},S_{rt})$.\end{remark}

\subsection{Special flows}\label{s:spflow} Assume that $R\in{\rm Aut}\zdk$ is an ergodic automorphism and let $f:Z\to \R$ belongs to $L^1\zdk$, $f>0$ ($\kappa$-a.e). For $n>0$, set $f^{(n)}(z)=\sum_{j=0}^{n-1}f(R^jz)$. Complemented by $f^{(0)}(z)=0$ and $f^{(-n)}(z):=-f^{(n)}(T^{-n}z)$, we obtain the cocycle identity
\be\label{cocind}
f^{(m+n)}(z)=f^{(m)}(z)+f^{(n)}(R^mz)\ee
true for $\kappa$-a.e.\ $z\in Z$ {and all $m,n\in\Z$}. By a {\em special flow} $R^f=(R^f)_{t\in\R}$ with the {\em base} $R$ and {the} {\em roof function} $f$, we mean the {$\R$}-action given by
\be\label{defspfl}
R^f_t(z,r)=(R^nz, r+t-f^{(n)}(z)),\ee
where $n\in\Z$ is the only number satisfying $f^{(n)}(z)\leq r+t<f^{(n+1)}(z)$. Then $T^f_t$ is a flow defined on the probability standard Borel space $(Z^f,\mathcal{D}^f,\kappa^f)$, where the space $$Z^f=\{(z,r)\in Z\times\R:\:z\in Z,0\leq r<f(z)\},$$
consists of the area in the plane under the graph of the function $f$, $\mathcal{D}^f$ is the restriction of the product $\sigma$-algebra $\mathcal{D}\otimes\mathcal{B}(\R)$ to $Z^f$, and $\kappa^f$ is the restriction of product measure
$\kappa\otimes\lambda_{\R}$ to $\mathcal{D}^f$, normalized by $\int_{Z}f\,d\kappa$. It is not hard to see that in view of the ergodicity of $R$, the special flow $R^f_t$ is also ergodic.

\begin{remark}\label{spflor}{\em
One can observe that if $r\neq0$ then the two special flows
$$
\mbox{$(R^f_{r^{-1}t})_{t\in\R}$ and $(R^{rf}_t)_{t\in\R}$ are isomorphic}.$$
Indeed, the map $Z^{f}\ni(x,s)\mapsto (x,rs)\in Z^{rf}$ is measure-preserving and equivariant.
}\end{remark}
\subsection{Time changes of flows}\label{s:tch} Assume that $R_t$ is a flow on $\zdk$ and let $v\in L^1\zdk$ be a positive function. The function $v$ determines a cocycle over $R_t$ (which, by abusing the notation, we will denote again by $v$), given by the formula
\be\label{cocy1}
v(t,x):= \int_0^t v(R_sx) d s.
\ee
Recall that the cocycle property {means} that $v(t_1+t_2,x)=v(t_1,x)+v(t_2,R_{t_1}x)$ for $\kappa$-a.e.\ $x\in {Z}$ and all $t_1,t_2\in\R$.
One can then show  (see for example~\cite{Co-Fo-Si}) that for a.e.\ $x\in {Z}$ and all $t\in\R$, there exists a unique {$u=u(t,x)$}  such that
\be\label{zamcz}
\int_0^uv(R_sx)\, ds=t\;\;\ee
(Note that if $t<0$, then $u<0$.) Now, we can define $\widetilde{R}_t(x):=R_{u(t,x)}(x)$ and this is indeed an $\R$-action as
the function $u=u(t,x)$ satisfies the cocycle identity: $u(t_1+t_2,x)=u(t_1,x)+u(t_2,\widetilde{R}_{t_1}x)$. The new flow $\widetilde{R}_t$ has the same orbits as the original flow (hence it is ergodic if $R_t$ was), and it preserves the measure $\widetilde{\kappa}\ll\kappa$, where $\frac{d\widetilde{\kappa}}{d\kappa}=
{v/\int_Zv\,d\kappa}$. We will also use notation $\widetilde{R}^v_t$ for the time change if it is not clear which function $v$ is used.

We say that $v$ is a \emph{quasi-coboundary} if  $v(t,x)=\xi(x)-\xi(R_tx)+t$
for a measurable $\xi:Z\to\R$. If $v$ is a quasi-coboundary, then the flows $R_t$ and $\widetilde{R}_t$ are isomorphic. More generally, if $w:{Z}\to\R^{+}$ is another time change of $R_t$ and if
$$
v(t,x)-w(t,x)=\xi(x)-\xi(R_tx)$$
(for $\kappa$-a.e.\ $x\in {Z}$ and all $t\in\R$) then the two time changes $\widetilde{R}^v_t$ and $\widetilde{R}^w_t$ are isomorphic, where the isomorphism is given by the map $x\mapsto \widetilde{R}^w_{\xi(x)}x$ (in the special case of the trivial time change of $R_t$ given by $w=1$, for which the associated cocycle is $w(x,t)\equiv t$, this reduces to the previous definition of quasi-coboundary).
Note that on $Z$ we consider, in general, two different absolutely continuous measures.

\begin{remark}\em One can show that any special flow $R_t^f$, where $R\in{\rm Aut}\zdk$, can be obtained from {a time change of the \emph{suspension} of $R$, i.e.\ from a time change of} the special flow over $R$ built under the constant function~1.\end{remark}

%\smallskip
\begin{remark}\em
%\paragraph{\bSmooth time changes of smooth flows}
 If $R_t$ is a smooth flow on a manifold $Z$ and $v$ is also smooth then the time change flow $\widetilde{R}^v_t$ is a smooth flow. In fact, when $A:Z\to TZ$ is a vector field and
$$
\frac d{dt}R_tx=A(R_tx)$$
and if $v:Z\to\R^{+}$ is smooth, then consider the vector field $\frac1v A$. There is a unique (smooth) flow $\widetilde{R}_t$ given by
$$
\frac d{dt}\widetilde{R}_tx=\frac1{v(\widetilde{R}_tx)}A(\widetilde{R}_tx)$$
and one can show that the flow $\widetilde{R}_t$ thus defined (which is clearly smooth by construction) coincides with the flow $\widetilde{R}^v_t$ defined at the beginning of the section.
\end{remark}
%Let us show that the flow $\widetilde{R}_t$ thus defined (which is clearly smooth by construction) coincides with the flow $\widetilde{R}^v_t$ defined at the beginning of the section.
%We have $\widetilde{R}_tx=R_{u(t,x)}x$, so
%$$
%\frac d{dt}\widetilde{R}_tx=\frac d{dt}u(t,x)\frac d{dt} R_{u(t,x)}x= \frac d{dt}u(t,x) A(R_{u(t,x)}x).$$
%But
%$$
%\frac d{dt}\widetilde{R}_tx=v^{-1}(\widetilde{R}_tx)A(\widetilde{R}_tx),$$
%so looking at the RHS of the two equations above, we obtain
%$v^{-1}(\widetilde{R}_tx)=\frac d{dt} u(t,x)$, so
%$$
%v(R_{u(t,x)}x)\frac d{dt}u(t,x)=1.  $$
%(Notice that since $v$ is bounded avay from zero, $\frac{d}{dt}u(t,x)$ is bounded.) Integrating form $0$ to $T$, we obtain
%$$
%\int_0^T v(R_{u(t,x)}x)\frac d{dt} u(t,x)\, dt=T.$$
%Now, substituting $s=u(t,x)$, we finally obtain
%$$
%\int_0^{u(T,x)}v(R_sx)ds=T,$$
%%Note also that $v(x,u(t,x))=t$.
%which shows that the two flows are the same.

%\smallskip

Let us conclude this section with two very simple lemmas on time changes which will be helpful later.

\begin{lemma}\label{r:ue}
%\begin{remark}\em
Assume that  $Z$ is a compact metric space and  $R_t: Z \to Z$ is a uniquely ergodic flow. {Then, for} any   {$v\in C(Z)$} a positive function with $\int_Zv\,d\kappa=1$, {we have}
$$
\lim_{t\to\infty}\frac t{u(t,z)}=1$$
uniformly in $z\in Z$.
\end{lemma}
\begin{proof}
Remark that  for each (and, in fact, uniformly {in}) $z\in Z$, we have
$$
\frac{t}{u(t,z)}=\frac{1}{u(t,z)}\int_0^{u(t,z)}v(R_sz)\,ds \xrightarrow[t\to \infty]{} \int_Zv\,d\kappa=1.
$$
The conclusion  follows immediately.
\end{proof}

The following shows that the study of  rescalings of  a flow  can be reduced to the study of time changes.
% easy lemma shows that if we want to  understand algebraic reparametrization of $\widetilde{R}_t$, namely study flows of the form, we can reduce the problem the
\begin{lemma}\label{ptc}
Consider {a} flow $R_t: Z \to Z$ on  $\zdk$, $v:Z\to\R^{+}$ {integrable} and let {$\widetilde{R}^{v}_t$} be the corresponding time change.  Fix $0< p\in\R$ and consider {$\widetilde{R}^{pv}_t$}. {Then, for} $\kappa$-a.e.\ $x\in Z$ and all $t\in\R$, {we have}
$$
\widetilde{R}^{pv}_t(x)=\widetilde{R}^v_{(1/p)t}(x).
$$
\end{lemma}
\begin{proof}
For $\kappa$-a.e.\ $x\in Z$ and all $t\in\R$ we have $\int_0^{u(x,t)}(pv)(R_sx)\,ds=pt$. This can be interpreted by saying that
$\widetilde{R}^{pv}_{pt}(x)=R_{u(x,t)}(x)$. Hence, since $R_{u(x,t)}(x)=\widetilde{R}_t(x)$, it gives the conclusion.
\end{proof}

\subsection{Irrational rotations, Denjoy-Koksma inequality}\label{sec:DK}
For each $x\in\R$, we set $$\|x\|:=\min(\{x\},1-\{x\}),\qquad  \text{where}\  \{x\}=x-[x]$$ stands for the fractional part of $x$. By $\T=[0,1)$ we denote the additive circle. Then, clearly, $\|\cdot\|$ determines a translation invariant metric on $\T$. By $\lambda_{\T}$ (or $\lambda$) we will denote Lebesgue measure on $\T$.

Let $$
[0,1)\ni\alpha=[0;a_1,a_2,\ldots]=\frac{1}{a_1+\frac{1}{a_2+
\ldots}}$$ be irrational.  Then the positive integers $a_i$ are called {\em partial quotients} of $\alpha$. By setting
$$
p_0=0,p_1=1,p_{n+1}=a_{n+1}p_n+p_{n-1},\;
q_0=1,q_1=a_1,q_{n+1}=a_{n+1}q_{n+1}+q_{n-1}$$
for $n\geq2$, we obtain the sequences $(p_n),(q_n)$ of {\em numerators} and {\em denominators} of $\alpha$, respectively. Then (see e.g.\ \cite{Kh})
$$
\frac{1}{2q_nq_{n+1}}<\left|\alpha-\frac{p_n}{q_n}\right|<
\frac1{q_nq_{n+1}};$$
equivalently,
\begin{equation}\label{CFbasicestimate}
\frac1{2q_{n+1}}<\|q_n\alpha\|<\frac1{q_{n+1}}.
\end{equation}

From this inequality, we can deduce the following elementary \emph{spacing} properties of orbits which we will use several times later.
\begin{lemma}[Spacing of orbits up to return times]\label{spacing_orbit}
For any $x \in \mathbb{T}$ and $n \in \mathbb{N}$, the orbit $\{ x+i\alpha, \ 0\leq i < q_n\}$ is such that
\be\label{ee2}
\|x+i\alpha-(x+j\alpha)\|\geq \frac{1}{2q_n}, \qquad \text{for\ all} \ 0\leq i\neq j < q_n.\ee
%Furthermore, there is at most one point of the orbit inside the interval $[\frac{1}{q_{n+1}},\frac{1}{q_{n+1}} ]$
On the other hand, in each interval of length $2/q_n$ there must be a point of the form $x+j\alpha$ for some $0\leq j<q_n$.
\end{lemma}
\begin{proof}
For the first part, note that for $0\leq i,j<q_n$, $\|x+i\alpha-(x+j\alpha)\|=\|(i-j)\alpha\|$ with $|i-j|<q_n$. Since $\frac1{2q_n}<\|q_{n-1}\alpha\|$ and $\|r\alpha\|\geq\|q_{n-1}\alpha\|$ whenever $|r|<q_n$, we have \eqref{ee2}.  For the second,  assuming that $\alpha>\frac{p_n}{q_n}$, we have $k\alpha-\frac{kp_n}{q_n}<\frac1{q_{n+1}}$ for $0\leq k<q_n$, so that in particular there {exist} $0\leq i \neq j < q_n$ such that
$$
\|x+i\alpha-(x+j\alpha)\|< \frac{1}{q_n} + \frac{1}{q_{n+1}} < \frac{2}{q_n}.
$$
Treating the case $\alpha<\frac{p_n}{q_n}$ similarly, the second part follows.
% it follows that in each interval of length $2/q_n$ there must be a point of the form $x+j\alpha$ for some $0\leq j<q_n$. Hence, by considering the interval $(-1/q_n,1/q_n)$, we obtain the desired conclusion.
\end{proof}

\smallskip

If $f:\T\to\R$ has bounded variation then, for each $n\geq0$ and $x\in\T$, we have the following \emph{Denjoy-Koksma inequality}:
\be\label{eq:doko}
\left|f^{(q_n)}(x)-q_n\int_{\T}f\,d\lambda\right|\leq 2{\rm Var}(f), \ee see e.g.\ \cite{Ku-Ni}.

\smallskip
We will also make use of so called \emph{Ostrowski expansion} of a natural number $N$. Namely, we can represent $N$ as
$$
N=\sum_{j=1}^Kb_jq_j,$$
where $0\leq b_j\leq a_j$ (and $b_K\neq 0$) and if, for some $j$, we have $b_j=a_j$ then necessarily $b_{j+1}=0$. Such a decomposition is unique, see e.g.\ \cite{Ku-Ni}.

%In what follows we will consider Arnol'd flows $(T_t)$ on $\T^2$: we will assume there is one (asymmetric) singularity at $0$ and that the base rotation belongs to a full measure set which will be specified later (it will be a subset of the set considered in [FayadKanigowski]). We will show that for any $p\neq q\in \R\setminus \{0\}$ $T_{pt}$ is disjoint from $T_{qt}$. This in particular will imply Sarnak conjecture for the flow $(T_t)$.\footnote{In fact, whenever $(T_t)$ is uniquely ergodic, all continuous observables $(f(T_{nt}x))_{n\geq 1}$ for the flow will be orthogonal to all arithmetic multiplicative functions.}

\subsection{Ratner properties}\label{Rproperties}
Our main disjointness criterion (Theorem~\ref{disjoint.flows}) is based on quantitative divergence properties and is {inspired} by the Ratner property (originally defined in \cite{Rat}) and in particular its recent variation known as {the} \emph{switchable Ratner property} (SR-property), first introduced in \cite{Fa-Ka}.  While the SR-property is not {explicitly} used in the paper,  it is implicitly present, both in formulation of the disjointness criterion and the phenomena driving it. We hence  include in this section its formal definition followed {by} some explanations, as we think that the reader who is not familiar with it, might benefit by reading it before the following section.

\smallskip
%the version of Ratner's property that will suit our needs.
Let $(X,d)$ be a $\sigma$-compact metric space, $\mathcal{B}$ the $\sigma$-algebra of Borel subsets of $X$ {and} $\mu$ a Borel probability measure on {$X$}.  Let $(T_t)_{t\in\mathbb{R}}$ be an ergodic flow acting on $(X,\mathcal{B},\mu)$.
The Ratner property encodes a quantitative property of controlled divergence of nearby trajectories in the flow direction. Heuristically, we want that for \emph{most} pairs of nearby points $x,x'$, the orbits of $x,x'$ \emph{split} in the flow direction (say at time $t=M$) by a definite amount (the \emph{shift} $p$, which belongs to a fixed compact set $P$) and then \emph{realign}, so that now $T_{t}(x)$ and the time-shifted orbit $T_{t+p}(x')$ are close, and stay close for a fixed proportion $\kappa M$ of the time it took to see the shift, namely for \emph{most} times $t \in [M,M+L]$ where $L>\kappa M$ (this is made precise in Definition~\ref{def:SR}). One can see that this type of phenomenon is possible only for parabolic systems, in which orbits of nearby points diverge with polynomial or subpolynomial speed.

In the \emph{switchable} variation, this phenomenon might not happen for positive times $t$ (i.e. in the \emph{future}), but only in the \emph{past}, i.e.\ for $t<0$ (and one can \emph{switch} between exploiting the past or the future according to the pair of points $x,x'$). This is encoded by the following formal definition.  (We comment below on the relation with the definitions in the literature, see Remark~\ref{rk:Ratner_hist}).

\begin{definition}[SR-property]\label{def:SR}\em
 % and $t_0>0$.
 We say that the flow $(T_t)$ has the {\em  SR-property} (with \emph{shifts set} $P$)  if there exists a compact set $P\subset \mathbb{R}\setminus\{0\}$ (the \emph{shifts set}) such that:
%\begin{equation}\label{j22}
%\parbox{0.9\textwidth}{
%\smallskip

\noindent  for every $\vep>0$ and $N\in\N$ there exist $\kappa=\kappa(\vep)$, $\delta=\delta(\vep,N)$ and a set $Z=Z(\vep,N)$ with $\mu(Z)>1-\vep$, such that:

\smallskip
\noindent for every $x,y\in Z$ with $d(x,y)<\delta$ and $x$ not in the orbit of $y$, there exist $p=p(x,y)\in P$ and $M=M(x,y),$ $L=L(x,y)$ such that
$M \geq N$ and ${L} \geq \kappa M$, and
%\end{equation}
at least one of the following holds:
\begin{enumerate}[(i)]
\item
$ d(T_{t}(x),T_{t+p}(y))<\vep \quad \text{for} \quad t\in U \subset [M, M+L]$, or
\item
$d(T_{t}(x),T_{t+p}(y))<\vep \quad \text{for} \quad -t\in U \subset [M, M+L]$,
\end{enumerate}
where $U\subset [M, M+L] $ is such that its Lebesgue measure $|U|$ satisfies $|U|>(1-\vep)L$.
%We say that $(T_t)_{t\in\mathbb{R}}$ has the \emph{switchable weak Ratner property}, or, for short, the {\em SWR-property} (with the set $P$) if
%$\{t_0>0: (T_t)_{t\in\mathbb{R}} \text{ has }sR(t_0,P)\text{-property}\}$ is uncountable.
\end{definition}
%\begin{definition}[SWR-Property, see \cite{Fa-Ka}]\label{def:SR}
% We say that the flow $(T_t)_{t\in\mathbb{R}}$ has \emph{switchable Ratner property}, or, for short, the {\em SR-property}, if
%$(T_t)_{t\in\mathbb{R}}$ has  the { SWR-property} with the set $P=\{1,-1\}$.
%\end{definition}
%We will assume moreover that the flows under consideration are {\em almost continuous}. Recall that $\mathcal{T}=(T_t)_{t\in\R}$ is said to be almost continuous if
%$$
%\parbox{0.9\textwidth}{
%for every $\vep>0$ there exists $X_\vep\subset X$ with $\mu(X_\vep)>1-\vep$ such that for every $\vep'>0$ there exists $\delta'>0$ such that for every $x\in X_\vep$ we have $d(T_t(x),T_{t'}(x))<\vep'$ whenever $t,t'\in [-\delta',\delta']$.
%}
%$$
%\todoC[inline]{I tried to add an intutive explanation of the Ratner property... is there a good reference?}
%\todoA[inline]{I don't think there is a reference, i like your explanation...}

%\begin{remark}\em \label{rk:Ratner_expl}
Referring to the heuristic explanation before the definition, \emph{most} initial points is formalized by the set $Z$ of arbitrarily large measure. The splitting and realignement with shift of nearby orbits should happen for arbitrarily large times (i.e. for every $N\geq 0$ there must be an $M\geq N$).  Finally,  $(i)$ describes controlled splitting in the \emph{future}, while  $(ii)$ describes controlled splitting in the \emph{past}.

%Intuitively, the SR-property (or the SWR-property) mean that, for a large set of choices of nearby initial points (i.e. pairs of points in the set $Z$ which are $\delta$ close), the orbits of the  two points \emph{either in the past, or in the future} (according to wheather (i) or (ii) hold), diverge and then, after some arbitrarily large time ($M t_0$ or $-M t_0$) \emph{realign}, so that $T_{nt_0}(x)$ is close to a a \emph{shifted} point $T_{nt_0+p}(y)$ of the orbit of $y$ (where $p\in P$ denotes the temporal \emph{shift}), and the two orbits then \emph{stay close} for a \emph{fixed proportion} $\kappa$ of the time $M$. One can see that this type of phenomenon is possible only for parabolic systems, in which orbits of nearby points diverge with polynomial or subpolynomial speed.
%\end{remark}

\begin{remark}\label{rk:Ratner_hist}\em
In the classical Ratner property,  $P=\{1,-1\}$ and for all $x,y\in Z$ (i) has to be satisfied. The first extensions of the Ratner property were introduced by K. Fr\k{a}czek and M. Lema{\'{n}}czyk in \cite{Fr-Le} and \cite{FL2} (as mentioned in the introduction) and amounted to allow $P$ to be any finite set (the \emph{finite Ratner property})  or any compact set $P\subset \mathbb{R}\setminus\{0\}$ (the \emph{weak Ratner property})\footnote{The definition given in \cite{Fr-Le} and \cite{FL2} is slightly different, in that the property is stated in terms of the time $t$ maps of the flow, and the flow is said to have the corresponding Ratner properties if the set of times $t$ such that $T_t$ has the corresponding property is uncountable. This is clearly implied by the definition given here.}.
  The \emph{switchable} variation (where the possibility $(ii)$ of controlled splitting in the past is also allowed) was introduced by B.~Fayad and the first author in \cite{Fa-Ka}.

  All the variants of the Ratner properties are defined so that, if $(T_t)$  has the SR-property, then it has the \emph{finite extension of joinings property} (shortened as FEJ property, see \cite{Fa-Ka}), which is a rigidity property that restricts the type of self-joinings that $(T_t)$  can have (see \cite{Ryz-Tho,Fr-Le}).  %We refer the reader to \cite{Gl, Ryz-Tho} for the definition of joinings and FEJ.
  \end{remark}
% Thus the weak Ratner property \cite{FL2} is analogous to Definition \ref{def:SWR} but with the restriction that for all $x,y\in Z$ (i) has to be satisfied.

%Let us also stress that in the Ratner property (some times also called \emph{two-point Ratner property})  The generalizations given mentioned in the introduction, i.e.

% In the original definition of the Ratner property, \cite{Ra:ho}, $P=\{-1,1\}$ and  Therefore SR-property differs from the original one just by the possibility of chosing (i) or (ii) (depending on points).

%One can show that the SR-property, as well as other Ratner properties (with the set $P$ being finite), are an isomorphism invariant (see for example the Appendix of \cite{KKU}).  Let us remark that it is not known whether the weak Ratner property is an isomorphism invariant.

%It is well known that if $(T_t)_{t\in\mathbb{R}}$  is mixing \emph{and} has the FEJ property, then it is automatically mixing of all orders, \cite{Ryz-Tho}, so this property was used for example \cite{FK} and in \cite{KKF} to \emph{enhance} mixing properties (showing for example that flows which are mixing and have a variant of the Ratner properties are indeed mixing of all orders).

\section{A criterion for disjointness}\label{sec:Sec3}
The main result of this section is the new disjointness {criterion  which} the main results of the rest of the paper are based {on}. Since this criterion and the explanations in this section are inspired by the SR-property (even if the property is not used directly), the reader who is not familiar with it might benefit from reading {Section}~\ref{Rproperties} first.
%While this property is not explicitely used in the paper, we  we recommend the reader
 %included its definition and some explanations in \S~\ref{Rproperties}, since the reader who is not familiar with it might want to read that section first.

The statement of the criterion is given in {Section}~\ref{s:Sec3}.
Since it   is rather technical, let us first explain the guiding ideas behind it. %s formulation.
%. Before, we need to introduce a definition, of \emph{almost linear reparametrizations} (in \S~\ref{AL} below).
%The reader might want to keep in mind that
The criterion was devised and formulated so that it can be applied  to prove disjointness of two flows which both have the SR-property (see Definition~\ref{def:SR}), so that in both flows one can observe a controlled form of divergence of nearby trajectories (for example polynomial divergence), but  the  speed of divergence for the two flows is different
%. More precisely, we have two flows for which we can observe divergence of nearby trajectories in the flow direction in a controlled way (say with polynomial speed) but the speeds are different
(for example for one flow is it  linear, in the other quadratic).

The key \emph{shearing phenomenon} exploited in the criterion is that, for pairs of two nearby points in the first system  and two nearby points in  the second, after some time (depending on both pairs of points) we will see a \emph{relative divergence}, i.~e.~in one pair we will see a realignement with some shift $C_1>0$ in the flow direction, while in the other we will see a realignement with a shift $C_2>0$ with $C_1\neq C_2$. To see that this is the case when both flows have the Ratner property (\emph{not} switchable for now) but with different speeds, one can reason as follows. Using the Ratner or SR-property, one can find times $M_1>0$ and $M_2>0$ such that the first pair  splits exactly by $1$ at time $M_1$ and the second splits by exactly $1$ at time $M_2$. If  the second pair has not yet not split at time $M_1$ (or symmetrically, first pair has not yet split at time $M_2$), then we are done with $C_1=1$ for one flow and $C_2=0$ for the other. The interesting case is when $M_1$ and $M_2$ are very close (assume for simplicity that $M_1=M_2$). In this case one can use the fact that the speed of divergence is different to conclude  that after some time, say $10M_1$, {the} two pairs will have split by a different amount (in the linear and quadratic case, it is a simple consequence of the fact that a line and parabola can have at most two points of intersection -at time $M_1$- and after $M_1$ they start slowly diverging). This explains how the relative shearing appears in this situation.

Let us remark that if both flows have only the \emph{switchable} Ratner property, one might have pairs of points for which one see the Ratner form of shearing only in the past for one flow, and only in the future for the other. Therefore, there is no hope to implement the heuristic {described} above on the \emph{whole} (or large parts) of the space.
An essential feature of the criterion is that for one of the two flows one can consider only pairs of points  in \emph{small} parts of space (whose  measure is  bounded {from} below, but not necessarily close to one), on which one sees controlled shearing \emph{both} forward \emph{and} backward. Thus, one can couple these pairs with pairs in the other flow which have the Ratner property in the future \emph{or} the past to get the \emph{relative divergence} explained above. We will come back to this point after the statement of the criterion in  {Section}~\ref{s:Sec3}.

In order to give the precise formulation of the criterion, we now first need to give the definition of \emph{almost linear} reparametrization. % (which will be used in the criterion to describe the realignement of the trajectories with the smaller.

\subsection{Almost linear reparametrizations}\label{AL}

 In the disjointness criterion stated in the next {Section}~\ref{s:Sec3}, we will consider a class of time reparametrizations given by functions $t \mapsto a(t)$ which are \emph{almost linear} (in the sense of Definition \ref{agood}). Heuristically, we want the function $a(\cdot)$ to be "close" to a linear function, either being close to a piecewise linear function, which on each interval of continuity has the form $a(t)=t+R_i$ (we call this case (PAL) for \emph{Piecewise} Almost Linear), or having derivative $a'(t)$  close to $1$ (a case dabbed (SAL), for \emph{Smoothly} Almost Linear).

% Now we explain where the technicalities come from: in Section 3.1. we start with definition of almost linear reparametrizations.

For a (measurable) set $I\subset \R $, $|I|$ denotes its Lebesgue measure.

\begin{definition}[almost linear functions and good triples]\label{agood}\em Assume that $I=[u_1,u_2]$ is an interval, $a:I\to \R$, $U\subset I$ and $0<d,\xi<1$. We say that $a$ is \emph{almost linear} and that the triple $(a,U,d)$ is $\xi$-{\em good} if at least one of the following holds:
\begin{enumerate}
\item[(PAL)]\ \ ({\it Piecewise Almost Linear function}) we can write $U$ as $U=\bigcup_{i=1}^v(c_i,d_i)$, where $v\leq |I|/d$ and $|U|>(1-\xi)|I|$, and for each  $i=1,\ldots,v$, for  $x\in(c_i,d_i)$
 we have that $a(t)=t+R_i$ for  some $R_i\in \R$ such that
\begin{equation*}%\label{vari}
R_1\leq \xi |I|, \qquad  |R_{i+1}-R_i|<\xi \text{ for every } i=1,\ldots,v-1.
\end{equation*}
In this case we say that $a$ is a \emph{piecewise almost linear} function.
\item[(SAL)]\;\;({\it Smooth {Almost Linear} function})  the function $a\in C^1(I)$, $U=I=[u_1,u_2]$ and  we have that $|a(u_i)-u_i|\leq \xi|I|$ for $i=1,2$ and $$1-\xi<a'(t)\leq 1+\xi, \qquad \text{for\ every}\ t\in I.$$
In this case we say that $a$ is a \emph{smooth almost linear} function.
\end{enumerate}
\end{definition}

\begin{remark}\em The notion of (PAL)  for {$a(\cdot)$} is introduced in order to deal with special flows: in this case, we need to take care of the fact that special flows are discontinuous close to the roof. This phenomenon produces the decomposition $(c_i,d_i)$, $i=1,\ldots, v$ in (PAL).
\end{remark}

This definition of almost linear functions and good triples is given so that the following result (which is essentially based on a  change of variables argument) holds. The lemma guarantees that if $a(\cdot)$ is (PAL) or (SAL) then the orbital integrals $\int_{M}^{M+L}f(T_tx)dt$ and $\int_{M}^{M+L}f(T_{a(t)}x)dt$ are close.
\begin{lemma}[Closeness of almost linear good reparametrizations]\label{rem:good}\em Let $0<d,\xi<1$, $a:[M,M+L]\to \R$ and let $U\subset [M,M+L]$ be such that $a$ is almost linear and $(a,U,d)$ is $\xi$-good. Assume that $f:\R\to (0,1]$ is measurable. Then
$$
\frac{1}{L}\int_M^{M+L}f(a(t))\,dt<\frac{1}{L}\int_{M}^{M+L}f(s)\,ds+
9\xi/d.
$$
\end{lemma}
The rest of this subsection will be taken by the proof of this {lemma}.
\begin{proof}
Assume first that $a$ is smoothly almost {linear, that is,  (SAL)} is satisfied.  Then
$$
\left|\frac{1}{L}\int_M^{M+L}f(a(t))\,dt-
\frac{1}{L}\int_M^{M+L}f(a(t))a'(t)\,dt\right|\leq \xi.
$$
On the other hand, substituting $s=a(t)$, by the assumptions on the endpoints given by $(SAL)$, we obtain that
$$
\frac{1}{L}\int_M^{M+L}f(a(t))a'(t)\,dt=
\frac{1}{L}\int_{a(M)}^{a(M+L)}f(s)\,ds=
\frac{1}{L}\int_M^{M+L}f(s)\,ds+
2\xi.
$$%\marginpar{was 4$\xi$, isn't 2$\xi$ enough?}
Combining the estimates, this concludes the proof in this case (since $d<1$ and hence $3\xi<9\xi /d$).

\smallskip
Assume now that $a$ is piecewise almost {linear, so (PAL)} is satisfied. Let $A_s:=\{1\leq i\leq v\;:\; d_i-c_i\leq \xi\}$ (where $s$ stays for \emph{short} intervals). Let  $U_s:=\bigcup_{i\in A_s}[c_i,d_i]$ and $U_l:=U\setminus U_s$ (where $l$ stays for \emph{large}). We then have
\begin{align*}
\left|\frac{1}{L}\int_M^{M+L}f(a(t))\,dt-\frac{1}{L}\int_{U_l} f(a(t))\,dt\right|& =\frac{1}{L}\int_{(I\backslash U)\bigcup U_s}f(a(t))\,dt \\ \leq \frac{|I\backslash U|}{L}+\frac{\xi v}L\leq \frac{\xi L}L+\frac \xi L\frac Ld<\frac{2\xi}{d}
\end{align*}
(where in the last inequality we used that $d<1$). Moreover, substituting ${r}=a(t)$ on each $(c_i,d_i)$ for  $i\notin A_s$ and denoting by $a(x^\pm)$ respectively the right and left limit of the function $a(t)$ as $t \to x^\pm$, we get
\begin{align}
\frac{1}{L}\int_{U_l} f(a(t))\,dt & = \frac{1}{L} \displaystyle\sum_{\begin{subarray}{c}i=1,\dots,v\\ i\notin A_s\end{subarray}}\int_{a(c_i^+)}^{a(d_i^-)}
f({r})\,d{r}  \nonumber
\\ & =  \frac{1}{L}\displaystyle\sum_{\begin{subarray}{c} i=1,\dots,  v \\ i\notin A_s\end{subarray}}\int_{a(c_i^+)}^{a(d_i^-)-\xi}f({r})
\,d{r} +
\frac{1}{L}\displaystyle\sum_{\begin{subarray}{c} i=1,\dots,  v \\ i\notin A_s\end{subarray}} \int_{a(d_i^-)-\xi}^{a(d_i^-)}f({r})\,d{r}. \label{sec.term}
\end{align}
Since $|A_s|\leq v\leq L/d$ ({where $|A_s|$} denotes the cardinality of $A_s$), we have
$$
\frac{1}{L}\displaystyle\sum_{\begin{subarray}{c} i=1,\dots,\\  v i\notin A_s\end{subarray}} \int_{a(d_i^-)-\xi}^{a(d_i^-)}f({r})\,
d{r}\leq \frac{\xi |A_s|}{L}\leq \frac{\xi}{d}.
$$
Now, we estimate the first term in RHS of~\eqref{sec.term}. We claim that the intervals $[a(c_i^+),a(d_i^-)-\xi]$ for $i\notin A_s$ are non-empty (since $i\notin A_s$ implies that $d_i-c_i>\xi$ and hence $a(d_i^-)-a(c_i^+)>\xi$) and \emph{pairwise disjoint}. To see this,
 notice that by the definition (PAL) of almost linear function, for $i\notin A_s$, we have  $$a(c_i^+)=c_i+R_i\leq d_i+R_i-\xi= a(d_i^-)-\xi=d_i+R_i-\xi\leq c_{i+1}+R_{i+1}=a(c_{i+1}^+),$$
 which proves that the intervals $[a(c_i^+),a(d_i^-)-\xi]$ are pairwise disjoint (and in increasing order). This now gives that
 $$
\frac{1}{L}\sum_{i=1,i\notin A_s}^v\int_{a(c_i^+)}^{a(d_i^-)-\xi}f({r})\,
d{r}
\leq \frac{1}{L}\int_{a(c_1^+)}^{a(d_v^-)}f({r})
\,d{r}.$$
Remark now that, by  the assumptions in (PAL) and recalling that $d<1$, we have that $$|R_v|=\left|\sum_{i=1}^{v-1}(R_{i+1}-R_i)+R_1\right|\leq \xi v+\xi L\leq 2\xi \frac{L}{d}.$$
Hence, using this bound on $R_v$ and the assumptions in (PAL) (and $d<1$), we have that
\begin{align*}
& \left| \frac{1}{L}\int_{a(c_1^+)}^{a(d_v^-)}f({r})\,d{r}  - \frac{1}{L}\int_M^{M+L}f({r})\,d{r} \right|  \leq  \frac{|a(c_1^+)-M|+|a(d_v^-)-M-L|}{L} \\ & \qquad \leq \frac{|c_1-M|+|d_v-M-L|+|R_1|+|R_v|}{L}
\leq 3\xi+2\frac{\xi}{d}.
\end{align*}
Combining all the estimates together, we get the desired conclusion also in this case. This finishes the proof.
\end{proof}

\subsection{The disjointness criterion}\label{s:Sec3}
Recall that we are considering measurable, measure-preserving $\R$-actions (i.e.\ flows) on probability standard Borel spaces. In fact, we will be constantly assuming that our configuration spaces are $\sigma$-compact metric spaces (considered with the Borel $\sigma$-algebras and probability Borel measures).

\smallskip
Let $(T_t)$ and $(S_t)$ be two weakly mixing flows acting  on $(X,\cB,\mu,d_1)$ and $(Y,\cC,\nu,d_2)$, respectively. Given $A\subset X$, we set
$V^1_{\epsilon}(A):=\{x\in X: d_1(x,A)<\epsilon\}$, and a similar notation $V^2_{\epsilon}(A')$ is used for a subset $A'$ of $Y$. Let $P=\{p,-p\}$, $p\neq 0$.

Let us first state the criterion, then make some comments that connect it to the heuristics presented at the beginning of the section.

\begin{theorem}[Disjointness criterion]\label{disjoint.flows} Let $0<c<1$. Assume that we have a sequence of sets $$(X_k)\subset\mathcal{B},\qquad \mu(X_k)\to \mu(X),$$ together with a sequence of automorphisms
$$(A_k)\subset Aut(X_k,\cB|_{X_k},\mu|_{X_k}),  \qquad \text{such that}  \ A_k\to Id\  \text{uniformly}\footnote{This means that for each $\delta'>0$,  $d_1(A_kx,x)<\delta'$ for all $k$ large enough and $x\in X_k$.}.$$
Assume moreover that for every $0<\epsilon<1$ and $N\in \N$
there exist a sequence
$$(E_k)=(E_k(\epsilon))\subset \cB, \qquad \mu(E_k)\geq c\mu(X),$$
and $0<\kappa=\kappa(\epsilon)<\epsilon$, $\delta=\delta(\epsilon,N)>0$ and a set
$$Z=Z(\epsilon,N)\subset Y, \qquad \nu(Z)\geq (1-\epsilon)\nu(Y)$$
 such that for all $y,y'\in Z$ satisfying $d_2(y,y')<\delta$, every $k$ such that $d_1(A_k,Id)<\delta$ and every $x\in E_k\cap X_k$, $x':=A_{k}x$ there are
 $$M\geq N,\qquad  L\geq 1, \qquad  \frac{L}{M}\geq \kappa, \quad \text{and}\quad  p\in P$$
 for which at least one of the following holds:
\begin{equation}\label{forw}
\max\left(d_1(T_tx,T_{a(t)+p}x'),\;d_2(S_ty,S_{a(t)}y')\right)
<\epsilon\text{ for } t\in U\subset [M,M+L]
\end{equation}
or
\begin{equation}\label{backw}
\max\left(d_1(T_{t}x,T_{a(t)+p}x'),d_2(S_{t}y,S_{a(t)}y')\right)
<\epsilon\text{ for } -t\in U\subset [M,M+L],
\end{equation}
where $a:=a_{x,x',y,y'}:[M,M+L]\to \R$ {is} extended by $a(-t)=a(t)$,  and $(a,U,c)$ is $\epsilon$-good.

\smallskip
Then, the flows $(T_t)$ and $(S_t)$ are disjoint.
\end{theorem}

As we explained at the beginning of this section, the criterion is meant to be applied to two flows having the SR-property (or the Ratner property) but with different speed.
The parameters $(X_k, A_k, E_k)$ allows us to relax the SR-property for one of the flows: we only need to control a positive proportion of space (the sets $E_k$) and only in a favorite direction (that we control well, for instance, the geodesic direction $A_k=g_{1/k}$ for the horocycle flow), we do not need to control \emph{all} nearby points. The set $Z$ is just {the set, where the} SR-property holds for the other flow.
The almost linear function $a(t)$ describes the relative drift between points $x,x'$ (or $y,y'$) (for example, it would be $t+td(x,x')$ if the divergence is linear).  Formulas \eqref{forw} and \eqref{backw}  then describe the \emph{relative shearing}, \emph{either} in the future \emph{or} in the past.

\smallskip
%The reader might recognize, both for pairs $x,x'$  of points in $X$ and pairs $y,y'$ of points in $Y$, a form of the \emph{switchable Ratner property} (see \S~\ref{Rproperties}) in the control among distances of orbits: namely, the closeness of points in the orbit $T_{t}x$ of $x$ to \emph{shifted} points of the form $ T_{a(t)+p}x'$ in the orbit of $x'$, or of a point $S_t(y),S_{a(t)}y'$ (for most $t$ in) a positive proportion $L>\kappa M$ of the time $M$ where the control starts).   The definition of almost linear function $a(t)$ allows, for corresponding times $t$, some fluctuation in the realignement of the orbits of $y,y'$ in the other flow (i.e it allows for an shift of size $\epsilon$).
%(a feature which is only possible in parabolic flows where shearing is of polynomial or subpolynomial nature) and, as in the \emph{switchable} version of this property, it is sufficient to verify it \emph{either} in the future (equation \eqref{forw}) \emph{or} in the past  (equation \eqref{backw}).

The rest of this section is devoted to the proof of {the} disjointness criterion.
\begin{proof}%[Proof of Lemma \ref{disjoint.flows}]
Let $\rho\in J((T_t)_{t\in\R},(S_t)_{t\in \R})$ be an ergodic joining, $\rho\neq \mu\times \nu$. Recall that by the weak mixing of the flow, all non-zero time automorphisms are weakly mixing, hence ergodic and therefore disjoint from the identity. Thus,
since $T_w$ is disjoint from $Id$  for $w\in P=\{p,-p\}$ {and} $\rho$ is not product measure, there exist $B_w\in \cB,C_w\in\cC$ such that
\begin{equation}\label{iddisT}|\rho(T_{-w}(B_w)\times C_w)-\rho(B_w\times C_w)|>\eta
\end{equation}
for some $0<\eta<1$. There exists $0<\epsilon<\frac{c\eta}{1000}$ such that $$\max\left(\left|\mu(V^1_{\epsilon}(B_w))-\mu(B_w)\right|,
\left|\nu(V^2_{\epsilon}(C_w))-\nu(C_w)\right|\right) <\eta/32.$$ Since $\rho$ is a joining, by the triangle inequality, for each $t\in\R$, we have
\begin{equation}\label{pertu}
|\rho(T_{-t}V^1_{\epsilon}(B_w)\times V^2_\epsilon(C_w))-\rho(T_{-t}B_w\times C_w)|<\frac{\eta}{16}.
\end{equation}

By applying the pointwise ergodic theorem to the joining flow $(T_t\times S_t, \rho)$ and the sets %$B_w\times C_w$,
$T_{-w}B_w\times C_w$ and $T_{-w}V^1_{\epsilon}(B_w)\times V^2_{\epsilon}(C_w)$
for $w\in P$, there exist $N_0\in \N$, $\kappa>0$ (which we can assume additionally to be of the form $\kappa=\kappa(\vep)$ as in the assumptions of our theorem) and a set $U_1\in \cB\otimes \cC$ with $\rho(U_1)>(1-\frac{c}{100})\rho(X\times Y)$ (recall that $c$ comes from our assumption) such that for every $L,M\geq N_0$ with $\frac{L}{M}\geq \kappa$ and $w\in P\cup\{0\}$, we have
\begin{equation}\label{eq1}\left|\frac{1}{L}\int_{M}^{M+L}\raz_{T_{-w}B_w
\times C_w}(T_tx,S_ty)\,dt-\rho(T_{-w}B_w\times C_w)\right|<\frac{\eta}{16},
\end{equation}
\begin{equation}\label{eq3}\left|\frac{1}{L}\int_{M}^{M+L}
\raz_{T_{-w}V^1_{\epsilon}(B_w)
\times V^2_{\epsilon}(C_w)}(T_tx,S_ty)\,dt-\rho(T_{-w}V^1_{\epsilon}(B_w)\times V^2_\epsilon(C_w))\right|<\frac{\eta}{16}
\end{equation}
and
\begin{equation}\label{eq2}\left|\frac{1}{L}\int_{M}^{M+L}\raz_{T_{-w}B_w
\times C_w}(T_{-t}x,S_{-t}y)\,dt-\rho(T_{-w}B_w\times C_w)\right|<\frac{\eta}{16},
\end{equation}
\begin{equation}\label{eq4}\left|\frac{1}{L}\int_{M}^{M+L}
\raz_{T_{-w}V^1_{\epsilon}(B_w)
\times V^2_{\epsilon}(C_w)}(T_{-t}x,S_{-t}y)\,dt-\rho(T_{-w}V^1_{\epsilon}(B_w)\times V^2_\epsilon(C_w))\right|<\frac{\eta}{16}
\end{equation}
whenever $(x,y)\in U_1$.
Let $U_2:=U_1\cap (X\times Z)$, where $Z=Z(\epsilon,N_0)$ comes from our assumptions. Then $\rho(U_2)>(1-c/50)\rho(X\times Y)$.
 %for $\epsilon=\epsilon'$ and $N_0$.
 Note also that since $X\times Y$ is $\sigma$-compact, measure $\rho$ is regular and hence, we can additionally assume that $U_2$ is compact.
Define ${\rm proj}:X\times Y\to X$, ${\rm proj}(x,y)=x$. Then the fibers of ${\rm proj}$ are $\sigma$-compact, and since $U_2$ is compact, the fibers of the map ${\rm proj}|_{U_2}:U_2\to {\rm proj}(U_2)\subset X$ are also $\sigma$-compact and ${\rm proj}(U_2)$ is also compact. Thus, by Kunugui's selection  theorem (see e.g.\ \cite{Ga-Le-Sch}, Thm. 4.1), it follows that there exists  a measurable (selection) $y:{\rm proj}(U_2)\to X\times Y$ such that $(x,y(x))\in U_2$. Note that $\mu({\rm proj}(U_2))\geq \rho(U_2)>(1-c/50)\mu(X)$. By Luzin's theorem there exists $X_{\rm cont}\subset {\rm proj}(U_2)$, $\mu(X_{\rm cont})\geq (1-c/50)\mu(X)$ such that $y$ is uniformly continuous on $X_{\rm cont}$. Finally, set
$$
\widetilde{U}:=U_2\cap (X_{\rm cont}\times Y).
$$
We have $\rho(\widetilde{U})>(1-c/10)\rho(X\times Y)$. Moreover, if $U_X:={\rm proj}(\widetilde{U})$ then also
$\mu(U_X)\geq \rho(\widetilde{U})>(1-c/10)\rho(X\times Y)$.
Hence, by the definitions of sequences $(A_k)$ and $(E_k)=(E_k(\epsilon))$, it follows that there exists $k_0=k_0(\epsilon)$ such that for $k\geq k_0$, we have
\begin{equation}\label{kurr}
\mu(A_{-k}(U_X\cap X_k)\cap (U_X\cap X_k)\cap E_k)>0.
\end{equation}
Let $\delta=\delta(\epsilon,N_0)$ come from the assumptions of our theorem. By the uniform continuity of $y:X_{\rm cont}\to Y$ it follows that there exists $0<\delta'<\delta$ such that $d_1(x_1,x_2)<\delta'$ implies $d_2(y(x_1),y(x_2))<\delta$ for each $x_1,x_2\in X_{\rm cont}$.
Since $A_k\to Id$ uniformly
and $\widetilde{U}\subset X_{\rm cont}\times Y$, there exists $k_1=k_1(\epsilon)$ such that for $k\geq k_1$, $d_2(y(x),y(A_kx))<\delta$  for $x\in X_k\cap X_{\rm cont}$. Fix $k\geq \max(k_0,k_1+1)$ (so that $d_1(A_k,Id)<\delta'$). Let $x\in A_{-k}(U_X\cap X_k)\cap (U_X\cap X_k)\cap E_k$. Such a point does exist in view of~\eqref{kurr}. Set $x'=A_kx$, $y=y(x)$, $y'=y(x')$. By definition, $(x,y),(x',y')\in \widetilde{U}$ and $d_2(y,y')<\delta$ and all other  assumptions of our theorem are satisfied for $(x,y)$, $(x',y')$ (so that we obtain $M,L,p$ depending on $(x,y)$ and $(x',y')$ satisfying~\eqref{forw} or~\eqref{backw}).

We will assume that~\eqref{forw} holds and will get a contradiction using~\eqref{eq1} and~\eqref{eq3}. If~\eqref{backw} holds, we argue analogously using~\eqref{eq2} and~\eqref{eq4}.
We claim that
\begin{equation}\label{disjon}
\rho(T_{-w}(B_w)
\times C_w)>\rho(B_w\times C_w)-\frac{\eta}{2}.
\end{equation}
Indeed, in view of~\eqref{pertu},~\eqref{disjon} follows by showing that
$$
\rho(T_{-w}(V^1_{\epsilon}(B_w))
\times V^2_{\epsilon}(C_w))>\rho(B_w\times C_w)-\frac{\eta}{4}.
$$

Using~\eqref{forw} (for $w=p$),~\eqref{eq1} (for $w=0$) and $(x,y)\in \widetilde{U}\subset U_1$ (and remembering that for $U$ in~\eqref{forw}, we have $|U|>(1-\epsilon)L$ as $(a,U,c)$ is $\epsilon$-good), we obtain  that
\begin{multline}
\frac{1}{L}\int_{M}^{M+L}\raz_{V^1_{\epsilon}(B_w)
\times V^2_{\epsilon}(C_w)}(T_{a(t)+w}x',S_{a(t)}y')\,dt> \frac{1}{L}\int_{M}^{M+L}\raz_{B_w
\times C_w}(T_tx,S_ty)\,dt>\\
\rho(B_w\times C_w)-\epsilon-\frac{\eta}{16}.
\end{multline}
Hence to complete the proof of claim~\eqref{disjon}, it is enough to show that
\begin{multline}\label{eq:suf}
\frac{1}{L}\int_{M}^{M+L}\raz_{V^1_{\epsilon}(B_w)
\times
V^2_\epsilon(C_w)}(T_{a(t)+w}x,S_{a(t)}y)\,dt<
\rho(T_{-w}V^1_{\epsilon}(B_w)
\times V^2_\epsilon(C_w))+\frac{\eta}{8}.
\end{multline}
Since $(a,U,c)$ is $\epsilon$-good, we get by Remark~\ref{rem:good} (for $f(t)=\raz_{V^1_{\epsilon}(B_w)
\times V^2_\epsilon(C_w)}(T_{t+w}x,S_{t}y)$)
\begin{multline*}
\frac{1}{L}\int_{M}^{M+L}\raz_{V^1_{\epsilon}(B_w)
\times V^2_\epsilon(C_w)}(T_{a(t)+w}x,S_{a(t)}y)\,dt\leq \\
\frac{1}{L}\int_{M}^{M+L}\raz_{V^1_{\epsilon}(B_w)
\times V^2_\epsilon(C_w)}(T_{s+w}x,S_{s}y)\,ds+ 9\epsilon/c\leq  \\
\rho(T_{-w}V^1_{\epsilon}(B_w)\times V^2_{\epsilon}(C_w))+\eta/16+9\epsilon/c,
\end{multline*}
where the latter inequality follows from~\eqref{eq3} for $w\in P$ remembering that $(x,y)\in \widetilde{U}\subset U_1$.  This completes the proof of~\eqref{eq:suf} and hence also of~\eqref{disjon} since $9\epsilon/c<\eta/16$.

Now, reasoning in a similar way, we get
\begin{equation}\label{rhot}\rho(T_{-w}(B_w)\times C_w)<\rho(B_w\times C_w)+\frac{\eta}{2},
\end{equation}
so putting together~\eqref{disjon} and\eqref{rhot} yields  $|\rho(T_{-w}(B_w)\times C_w)-\rho(B_w\times C_w)|<\frac{\eta}{2}$. This however contradicts~\eqref{iddisT}.
\end{proof}

\section{Disjointness for time changes of horocycle flows (proof of Theorem~\ref{main:prop})}\label{s:Sec4}
In this section, as a first application  of the disjointness criterion given by Theorem~\ref{disjoint.flows}, we show how it can be used to give a proof of Theorem~\ref{main:prop} on disjointness of rescalings of smooth time changes of horocycle flows. We first state two results which provide the main ingredients for the proof, namely Proposition~\ref{thm:hor} in {Section}~\ref{sec:prelimh},  {which  follows from the form of shearing (in the Ratner property) of} time-changes of horocycle flow (and {whose proof is postponed to} Appendix~\ref{app:horocycleR}), and Lemma~\ref{lem:buff} in {Section}~\ref{sec:deviations}, which follows from the work of Bufetov and Forni {\cite{BF}} (see Appendix~\ref{app:BF}). {Theorem~\ref{main:prop}} is then proved in {Section}~\ref{sec:proofmainprop},  exploiting these two as ingredients and the disjointness criterion.

\subsection{Preliminaries and shearing properties of time changes of horocycle flows}\label{sec:prelimh}
The Lie algebra $sl(2,\R)$ is generated by $U,V,X$, where
$$U:=\begin{pmatrix} 0& 1\\
0& 0\end{pmatrix}, \qquad   V:=\begin{pmatrix} 0& 0\\
1& 0\end{pmatrix}, \qquad X:=\begin{pmatrix} 1& 0\\
0& -1\end{pmatrix}$$
and $X$ generates the geodesic flow $(g_t)$, $U$ is the generator of the horocycle flow $(h_t)$ and $V$ generates the opposite horocycle flow $(u_t)$.

For two points $x,y\in M$ which are sufficiently {close,} $d_M(x,y)\leq \epsilon_0$ (for some $\epsilon_0$ depending on $M$), let $d_X$, $d_V$ and $d_U$ denote distances along, respectively, the geodesic, {the} opposite horocycle and the horocycle (the distances are measured locally in the Lie algebra). For a function $\xi\in C^1(M)$ and an element $L\in sl(2,\R)$, we denote by $L\xi$ the derivative of $\xi$ in direction $L$.

In this section $\tau\in C^1(M)$ is fixed and we consider $(\widetilde h^\tau_t)$. For simplicity,  we will drop $\tau$ from the notation and denote the time change simply by $(\tilde{h}_t)$. Recall that $u=u(t,x)$ is given by \eqref{zamcz} (for $v=\tau$).

Take $x,y\in M$ with $d_M(x,y)<\epsilon_0$. Using local coordinates, it follows that
$$
x=\exp(d_U(x,y) U)\exp(d_X(x,y) X)\exp(d_V(x,y) V)y,
$$
where ${d_W(x,y)}<2\epsilon_0$ for $W\in\{U,X,V\}$.
We have the following observation, which is a straightforward consequence of the Taylor formula:

\begin{lemma}\label{taylor:form}For every $\epsilon>0$ there exists $\delta>0$ such that for every $x,y\in M$, $d_M(x,y)<\delta$, we have
$$
|\tau(x)-\tau(y)-\sum_{W\in\{U,X,V\}}d_W(x,y)(W\tau)(x)|<
\epsilon\sum_{W\in\{U,X,V\}}d_W(x,y).
$$
\end{lemma}

We will also use the following matrix presentation: if $d_M(x,y)<\epsilon_0$, then (by the right invariance of $d_G$) there are (unique) small $s=s(x,y), r=r(x,y),\ov{v} =\ov{v}(x,y)$ satisfying
\be\label{marep}
xy^{-1}=h_{\bar v}\begin{pmatrix} e^s& 0\\
r& e^{-s}
\end{pmatrix}.
\ee
\begin{remark}\label{r:odl} \em Notice that for every $\epsilon>0$ there exists $\delta>0$ such that if $d_M(x,y)<\delta$ then
$$
\max\left(\left|\frac{d_U(x,y)}{|\ov{v}|}-1\right|,\; \left|\frac{d_X(x,y)}{|s|}-1\right|,
\left|\frac{d_V(x,y)}{|r|}-1\right|\right)<\epsilon^3.
$$
\end{remark}

Let $\chi_{x,y}:\R\to \R$ be given by
\begin{equation}\label{defchi}
\chi_{x,y}(t)=e^{-2s}t-e^{-3s}rt^2.
\end{equation}

 We have the following crucial Proposition \ref{thm:hor}, which  provides an exact formula for the amount of splitting in a time-changed {flow}. This proposition provides {a} stronger form of Ratner property  (indeed, it implies the Ratner property, see {Remark~\ref{itisR}) in} which points are allowed to diverge by an unbounded amount.
 %In fact, I think it is a version of Ratner's property on a longer scale

 In what follows, we set $0^{-1}:=+\infty$. Let  $A_x(T)$ be such that
\begin{equation}\label{funce}
\chi_{x,y}(u(T,x))=u(\chi_{x,y}(T)+A_{x}(T),y).
\end{equation}
{(In fact, $A_x(T)$ depends also on $y$, so $A_x(T)=A_{x,y}(T)$ and this number, or rather $\chi_{x,y}(T)+A_x(T)$, is uniquely determined by $\chi_{x,y}(u(T,x))$ and $y$, cf.\ \eqref{zamcz}}.)

\begin{proposition}\label{thm:hor} Fix $K\geq 1$. For every $\epsilon\in (0,K^{-3})$, there exist $N_\epsilon>0$ and $\bar{\delta}=\bar{\delta}(\epsilon)$ such that for every $x,y$ satisfying $\max(|r|,|s|,|\bar{v}|)<\bar\delta$ (cf.~\eqref{marep}) and every $T\in \R$ with $|T|\in [N_\epsilon, K|r|^{-1/2}]$, for $A_x(T)$ defined as above, we have
\begin{equation}\label{axt}
\left|A_x(T)+e^{-2s}\int_0^{u(T,x)}\left(\tau-\tau\circ g_{-s}\right)(h_tx)\,dt\right|\leq \epsilon.
\end{equation}
Moreover, we have
\begin{equation}\label{distgt}
d_M(\tilde{h}_T{x},\tilde{h}_{\chi_{x,y}(T)+A_x(T)}y)\leq \epsilon.
\end{equation}
\end{proposition}
\begin{remark}\label{itisR}\em
Proposition~\ref{thm:hor}, in particular,  implies the Ratner property for the time-changed flow $(\tilde{h}_t)$ (see {S}~\ref{Rproperties} for the definition). In order to see this, let $\tilde{T}:=\min(|r|^{-1/2},|s|^{-1})$. Notice that by Lemma~\ref{taylor:form}, for every $T={\rm O}(\tilde{T})$, we have
 $$
 \left|\int_{0}^{u(T,x)}(\tau-\tau\circ g_s)(h_tx)\,dt\right|=
 \left|\int_{0}^{u(T,x)}s(X\tau)(h_tx)+{\rm O}(\epsilon^3 |s|)\,dt\right|\leq$$$$
 |s|\left|\int_{0}^{u(T,x)}(X\tau)(h_tx)\,dt\right|+{\rm O}(\epsilon^3)=
 {\rm O}(\epsilon^3).
 $$
 Therefore, by \eqref{axt}, we have $|A_x(T)|={\rm O}(\epsilon^2)$ (we use {Proposition}~\ref{thm:hor} with $\epsilon$ replaced by $\epsilon^2$) and therefore, by \eqref{distgt},
 $$
 d_M(\tilde{h}_{T}{x},\tilde{h}_{\chi_{x,y}(T)}y)={\rm O}(\epsilon^2),
 $$
 for every $T={\rm O}(\tilde{T})$.
Recall that $\chi_{x,y}(t)=e^{-2s}t-e^{-3s}rt^2$ and hence there exists $T_0={\rm O}(\tilde{T})$ such that $\chi(T_0)=T_0\pm 1$ and for every $t\in [T_0,(1+\epsilon^4)T_0]$, $|\chi(t)-\chi(T_0)|={\rm O}(\epsilon^2)$. Hence, for $t\in [T_0,(1+\epsilon^4)T_0]$, we have
 $$
 d_M(\tilde{h}_{t}{x},\tilde{h}_{t\pm 1}y)<\epsilon,
 $$
 what finishes the proof of Ratner property.
\end{remark}

The proof of Proposition~\ref{thm:hor}, which is a little long and {tedious, is postponed} to Appendix~\ref{app:horocycleR}.

\subsection{Deviations of ergodic averages estimates}\label{sec:deviations}
In order to prove Theorem \ref{main:prop}, in addition to Proposition~\ref{thm:hor} % and Lemma \ref{lem:BF}
 stated in the previous section, we will also need the following estimates on ergodic averages for time-changes of {a} horocycle flow which can be deduced from the work of Bufetov and Forni \cite{BF} (see Appendix~\ref{app:BF}).

\smallskip
%We will first define the set of time changes $B_2(M)$ which we consider.
Recall that from the representation theory of $SL(2,\R)$ it follows that {the space $L^2_0(M)$ of zero mean square-integrable functions has a decomposition into irreducible (for the regular representation of $SL(2,\R)$) components, parametrized by the eigenvalues of the
Casimir operator $\square$ (listed with multiplicities) of the following form}:
$$
L^2_0(M)=\bigoplus_{\mu\in Spec(\square)\setminus\{0\}}H_\mu= \mathcal{H}_{p}\oplus \mathcal{H}_{c}\oplus \mathcal{H}_{d},
$$
where
$$
\mathcal{H}_{p}=\bigoplus_{\begin{subarray}{c}\mu\in Spec(\square),\\ \mu\geq 1/4\end{subarray}} H_\mu,\qquad \mathcal{H}_{c}=\bigoplus_{\begin{subarray}{c}\mu\in Spec(\square),\\ \mu\in(0,1/4)\end{subarray}} H_\mu,\qquad \mathcal{H}_{d}=\bigoplus_{\begin{subarray}{c}\mu\in Spec(\square),\\ \mu=-n^2+n,\,n\in\Z\setminus\{0\}\end{subarray}} H_\mu.
$$

It follows from \cite{Fl-Fo} that the cocycles generated by the functions (cf.\ \ref{cocy1} with $R_s$ replaced by $h_s$) supported on $\mathcal{H}_d\setminus H_0$ are coboundaries for $(h_t)$. We will consider
time-changes which belong to  {the class $B^+(M)$ defined} as follows. Let {$W^\alpha (M)=W^{\alpha,\alpha}(M)\subset L^2(M) $ denote the  standard Sobolev space}. Let $B(M)$ be given by
$$
B(M)=W^{6}(M) \backslash \mathcal{H}_{d} \subset L^2_0(M).
$$
Thus, the functions in $B(M)$ are not fully contained in the discrete series, or equivalently, in virtue of the above decomposition of  $L^2_0(M)$, the elements of $B(M)$ are those functions in $W^{6}(M) $ {projecting non-trivially on} $\mathcal{H}_{c}\oplus \mathcal{H}_{d}$. {Finally, we let
$B^+(M)$  consist of those \emph{positive} functions $\tau$ (which can then be taken as roof functions) which are of the form $\tau=c+\tau'$, where $\tau'\in B(M)$ and $c\in\R$}.
We call functions with non-trivial support on  $\mathcal{H}_c$ \emph{of type} \textbf{I} and functions supported on $\mathcal{H}_p$ \emph{of type} \textbf{II}.

\smallskip
The following lemma is a consequence of the work by Bufetov and Forni \cite{BF}\Red{, more precisely,} of  Lemma~\ref{lem:BF}, which is deduced from \cite{BF} \Red{in  Appendix}~\ref{app:BF}
\Red{(for $\beta_\tau$, $\beta_\tau^{(1/4)}$ and $\boldsymbol{v}$, see Lemma~\ref{lem:BF})}.

\begin{lemma}[Consequence of Bufetov-Forni \cite{BF}]\label{lem:buff} For every $\epsilon>0$, there exists $T_\epsilon>0$ such that
\begin{itemize}
\item[(T1)] if $\tau$ is of type \textbf{I} then there exist $\alpha_\tau \in (0,1)$, $c_\tau\in\R\setminus\{0\}$ such that \Red{for}
$|s|<T_\epsilon^{-1}$ and $|T|={\rm O}(|s|^{-1/\alpha_\tau})$,
\Red{we have}
$$
\left|\int_0^T(\tau-\tau\circ g_s)(h_tx)dt-sc_\tau T^{\alpha_\tau}\beta_\tau(g_{\log T}x)\right|<\epsilon^2;
$$

\item[(T2)] if $\tau$ is of type \textbf{II} and  $\beta_\tau^{(1/4)}$ vanishes identically
then \Red{for}
$|s|<T_\epsilon^{-1}$ and $|T|={\rm O}(|s|^{-2})$, \Red{we have}
$$
\left|\int_0^T(\tau-\tau\circ g_s)(h_tx)dt-sT^{1/2}(\frac{d}{ds}\beta_\tau)(0,v\log T,g_{\log T}x)\right|<\epsilon^2;
$$
\item[(T3)] if $\tau$ is of type \textbf{II} and $\beta_\tau^{(1/4)}$ does not vanish identically, then for  $|s|<T_\epsilon^{-1}$ and $|T|={\rm O}(\frac{|s|^{-2}}{\log^2 s})$, \Red{we have}
$$
\left|\int_0^T(\tau-\tau\circ g_s)(h_tx)dt+\frac{s}{2}T^{1/2}\log T\beta^{(1/4)}_\tau(g_{\log T}x)\right|<\epsilon^2.
$$
\end{itemize}
\end{lemma}
\begin{proof}
The proof in case \Red{(T1)} is a straightforward consequence of the first part of  Lemma~\ref{lem:BF}, since $\phi_\tau(s)=\phi_\tau'(0)s+{\rm O}(s^2)$ and hence (setting $c_\tau:=\phi_\tau'(0)$)
$$
\phi_\tau(s)T^{\alpha_\tau}\beta_\tau(g_{\log T}x)=c_\tau sT^{\alpha_\tau}\beta_\tau(g_{\log T}x)+{\rm O}(s^2T^{\alpha_\tau}),$$
where ${\rm O}(s^2T^{\alpha_\tau})={\rm O}(|s|)$. For \Red{(T2),} we have
$$
\beta_\tau(s,\Red{\boldsymbol{v}} \log T, g_{\log T}x)=s(\frac{d}{ds}\beta_\tau)(0,\Red{\boldsymbol{v}} \log T, g_{\log T}x)+ {\rm O}(s^2)
$$
and a \Red{reasoning analogous to the above one} applies. An analogous reasoning (expanding $e^{-s/2}-1$ at $0$) gives \Red{(T3)}. This finishes the proof.
\end{proof}

In the next section, we will also make use of the following observation.
\begin{remark}\label{remarknew}{\em
Recall that the \Red{linear} space $\mathcal{P}([0,1]):=\{p:[0,1]\to \R:\: p\text{ is a quadratic polynomial}\}$ is finite dimensional and hence any two norms on this space are equivalent. In particular, it follows that there exists a constant $C_\mathcal{P}$ such that for every $U>0$ and every quadratic polynomial $w:\R\to \R$, $w=at^2+bt+c$,
$$
C^{-1}_\mathcal{P}\max\left(|a|U^2,|b|U,|c|\right) <\sup_{t\in [0,U]}|w(t)|<C_\mathcal{P}\max\left(|a|U^2,|b|U,|c|\right).
$$
}
\end{remark}

\subsection{Proof of Theorem \ref{main:prop}}\label{sec:proofmainprop}
In this section \Red{we  prove} Theorem~\ref{main:prop}, by showing how the assumptions of the disjointness criterion can be verified using  Proposition~\ref{thm:hor}  and \Red{Lemma~\ref{lem:buff}}.

\begin{proof}[Proof of Theorem \ref{main:prop}] We will first give a proof \Red{when the assumptions of (T2) are satisfied}, \Red{in fact, this is the most complicated case}. We will then state what changes are needed if $\tau$ satisfies \Red{(T1)} or \Red{(T3)}. Fix $p,q$ and assume WLOG that $0<p<q$. \Red{Let $A:=\overline{\{\boldsymbol vt:t\in \R\}}\subset \T^\infty$.} Notice that if $c'>0$ is small enough, then for some $d_{p,q}>0$, we have
\be\label{eq:horl}
\|d/ds(\beta_{\tau})(0,\cdot,\cdot)\|_{C^0\Red{(A\times M)}}\left((p^{-1/2}-q^{-1/2}\right)-\frac{c'}{p^{1/2}}>2d_{p,q}.
\ee
\Red{Set}
\begin{equation}\label{buflarge}
L:=\{(a,x)\in A\times M: |d/ds(\beta_{\tau})(0,a,x)|> \|d/ds(\beta_{\tau})(0,\cdot,\cdot)\|_{C^0(A\times M)}-c'\}.
\end{equation}
By \Red{the} continuity of $d/ds(\beta_{\tau})(0,\cdot,\cdot)$ there exists a set $S\times R\subset L$, with $\lambda_\infty(S)\mu(R)=: c>0$, \Red{where $\lambda_\infty$ stands for Haar measure on $\T^\infty$}. \Red{Moreover, since $L$ is open, we can choose $S$ being open.}

We will show that the assumptions of Theorem \ref{disjoint.flows} are satisfied. Let $P=\{-r_{p,q},r_{p,q}\}$, where $r_{p,q}>0$ is a small constant to be specified later.  Let $(n_k)\subset\R$ be an increasing sequence going to $\infty$ such that $\boldsymbol v\log (pn_k^2)\in S$ (such $(n_k)$ exists since $S$ is open and the orbit of $0$ under the linear flow in direction $\boldsymbol v$ is dense in $A$).
 Define $X_k:=M$ and $A_k(x)=g_{1/n_k}(x)$. Then obviously $A_k\to$ Id uniformly.
 Fix $\epsilon>0$ and $N\in \N$. Let $E_k=g_{-\log (pn_k^2)}\Red{(R)}$; then $\mu(E_k)=\mu(R)>c$.

  Let $\kappa=\kappa(\epsilon)=\epsilon^4$, $Z(\epsilon,N)=M$ and $\delta=\delta(\epsilon,N)=
\min(\epsilon^{10},T_\epsilon^{-10},\bar{\delta})$, where $\Red{\bar{\delta}=}\bar{\delta}(\epsilon)$ comes from Proposition~\ref{thm:hor} and $T_\epsilon$ comes from Lemma~\ref{lem:buff}. Take $k$ such that $d(A_k,Id)<\delta$, $x\in E_k$, $x'=A_kx$ and $y,y'\in M$ so that $\max(d_M(x,x'),d_M(y,y'))<\delta$. By the definition of $A_k$, it follows that $r(x,x')=0$, $\bar{v}(x,x')=0$ and $s(x,x')=1/n_k$ (see~\eqref{marep}).

Set
\begin{equation}\label{bart}
\bar{T}:=\min(n_k^{2},s(y,y')^{-2}, |r(y,y')|^{-1/2}).
\end{equation}
We claim now that
\begin{equation}\label{jul2}
\left|\frac{1}{p}\int_0^{u(pn_k^2,x)}\left(\tau-\tau\circ g_{n_k^{-1}}\right)(h_tx)dt-
\frac{1}{q}\int_0^{u(qn_k^2,y)}\left(\tau-\tau\circ g_{n_k^{-1}}\right)(h_ty)dt\right|\geq d_{p,q}.
\end{equation}
Indeed, notice first that by \Red{(T2)}, for every $W={\rm O}(n_k^2)$ and every $z\in M$, we have (since $|u(T,\cdot)-T|<\epsilon^4T$)
$$
\left|\int_0^{u(W,z)}\left(\tau-\tau\circ g_{n_k^{-1}}\right)(h_tz)dt-
\int_0^{W}\left(\tau-\tau\circ g_{n_k^{-1}}\right)(h_tz)dt\right|\leq
$$
$$
\left|\int_0^{u(W,z)-W}\left(\tau-\tau\circ g_{n_k^{-1}}\right)(h_t(h_W z))dt\right|\Red{=}{\rm O}(n_k^{-1}(\epsilon^4n_k^2)^{1/2})={\rm O}(\epsilon^2).
$$
Hence, \eqref{jul2} is by \Red{(T2)}, up to ${\rm O}(\epsilon^2)$, equal to
$$
\left|\frac{1}{p}n_k^{-1}(pn_k^2)^{1/2}(d/ds\beta_\tau)(0, \Red{\boldsymbol v}\log(pn_k^2),g_{\log(pn_k^2)}x)-
\frac{1}{q}n_k^{-1}(qn_k^2)^{1/2}(d/ds\beta_\tau)(0, \Red{\boldsymbol v}\log(qn_k^2),g_{\log(qn_k^2)}y)\right|.
$$
Since $x\in E_k$, we have $(\Red{\boldsymbol v}\log(pn_k^2),g_{\log(pn_k^2)}x)\in \Red{S\times R} \subset L$ and \Red{therefore, by  \eqref{buflarge} and the triangle inequality, the above expression is larger than}
$$
\|d/ds(\beta_{\tau})(0,\cdot,\cdot)\|_{C^0(A\times M}\left(|p^{-1/2}-c_{p,q}|-q^{-1/2}\right),
$$
which, by \eqref{eq:horl}, gives \eqref{jul2}.

%The following property which is crucial for the proof: there exists $D_{p,q},d_{p,q}>0$ such that for every $T\in [0,k^{1/\alpha}]$ is such that $g_{\log pT}x\in L$, then
%\begin{equation}\label{jul3}
%D_{p,q}>
%\left|\frac{1}{p}\int_0^{pT}\left(\tau-\tau\circ g_{k^{-1}}\right)(h_tx)dt-
%\frac{1}{q}\int_0^{qT}\left(\tau-\tau\circ g_{k^{-1}}\right)(h_tx)dt\right|\geq d_{p,q}.
%\end{equation}
%Indeed, by \eqref{BF} it follows that the LHS follows straightforward since $\beta_{\tau}(\cdot)$ is bounded. For the RHS notice since $g_{\log pT}x\in L$ and $p<q$, we have
%\be\label{eq:lkdf}
%\left|p^{\alpha-1}\beta_{\tau}(g_{\log pT}x)-q^{\alpha-1}\beta_{Xj}(g_{\log qT}y)\right|\geq
%(\sup_{x\in M}\beta_f(x))(p^{\alpha-1}-q^{\alpha-1})- c_{p,q}p^{\alpha-1}>2d_{p,q},
%\ee
%for some constant $d_{p,q}>0$ (if $c_{p,q}>0$ is small enough). Moreover by \eqref{BF} and since $T\leq k^{-1/\alpha}$, for $w\in\{p,q\}$, we have
%$$
%|\frac{1}{w}\int_0^{wT}\left(\tau-\tau\circ g_{k^{-1}}\right)(h_tx)dt-w^{\alpha-1}\beta_{\tau}(g_{\log wT}x)|<\epsilon^{1/2}.
%$$
%This and \eqref{eq:lkdf} finishes the proof of \eqref{jul2}.

By \eqref{distgt}, in Proposition~\ref{thm:hor} (with $K$ to \Red{be} specified at the end of the proof),  for $x,x'$ and then for $y,y'$, using \eqref{bart}, for $N_\epsilon\leq t\leq K\bar{T}$, we get

\be\label{eq:f}
d_M(\tilde{h}_{pt}x,\tilde{h}_{\chi_{x,x'}(pt)+A_x(pt)}x')\leq \epsilon^2
\;\text{ and }\;
d_M(\tilde{h}_{qt}y,\tilde{h}_{\chi_{y,y'}(qt)+A_y(qt)}y')\leq\epsilon^2.
\ee

Define $a(t):=\frac{1}{q}(\chi_{y,y'}(qt)+A_y(qt))$. By the definition of $\chi(\cdot)$ and $A_y(\cdot)$ (from \eqref{funce}, we have that $A_y(\cdot)$ is smooth on every interval $I\subset [0,K\bar{T}]$), it follows that $a$ satisfies $(SAL)$ from Definition \ref{agood} for every interval $I=[\eta_1\bar{T},\eta_2\bar{T}]\subset [0,K\bar{T}]$.~\footnote{Indeed, in view of~\eqref{funce}, $$a(t)=\frac1q\int_0^{\chi_y(u(qt,y))}\tau(h_\theta y)\,d\theta=
\frac1q\int_0^{\chi_y(u(qt,y))-u(qt,y)}\tau(h_\theta y)\,d\theta+t.$$
Now, $\chi_y(u(qt,y))-u(qt,y)={\rm O}(|s|t)={\rm O}(|s|\bar{T})\leq \epsilon\bar{T}$ and since the integrand is bounded, we only need to show that $\frac{d}{dt}\left(\chi_y(u(qt,y))-u(qt,y)\right)={\rm O}(\epsilon^2)$. The latter follows from the fact that $u'(qt,y)$ is bounded and $\chi_y'(u(qt,y))-1=e^{-2s}-1+{\rm O}(r\bar{T})$.}

From \eqref{eq:f}, for all $N_\epsilon\leq t\leq \tilde{T}$, we have
$$
d_M(\tilde{h}_{qt}y,\tilde{h}_{qa(t)}y')\leq\epsilon^2.
$$
Let moreover $b(t)=\frac{1}{p}(\chi_{x,x'}(pt)+A_x(pt))$. Then, analogously, by \eqref{eq:f}, we have
$$
d_M(\tilde{h}_{pt}x,\tilde{h}_{pb(t)}x')\leq\epsilon^2.
$$
So to finish the proof of \eqref{forw}, it is enough to \Red{show that
there} exists $t_0\in [N_\epsilon,K\tilde{T}]$ such that
\begin{equation}\label{show1}
|a(t_0)-b(t_0)-r_{p,q}|< \epsilon^2,
\end{equation}
for some $r_{p,q}\neq 0$
and, for every $t\in[0,\kappa t_0]$, we have
\begin{equation}\label{show2}
|a(t+t_0)-b(t+t_0)-a(t_0)-b(t_0)|< \epsilon^2.
\end{equation}
By the definition of $\chi_{x,x'}(t)$ and $\chi_{y,y'}(t)$, we get
\begin{equation}\label{defa}
a(t)-b(t)= pol(t)+sub(t),
\end{equation}
where
$$pol(t)=(e^{-2s(y,y')}-e^{-2s(x,x')})t-
qe^{-3s(y,y')}r(y,y')t^{2}, \quad \text{and} \quad
sub(t)=\frac{A_x(pt)}{p}- \frac{A_y(qt)}{q}$$ (here we use the fact that $r(x,x')=0$). Notice that by \eqref{axt} for $x$ and by  \Red{(T2)}, for every $t_0\in [0, K\bar{T}]$ and every $t\in [0,\kappa t_0]$, we have (recalling that $s(x,x')=n_k^{-1}$)
$$
\frac{1}{p}|A_x(p(t+t_0))-A_x(pt_0)=e^{-2s}\int_0^{u(p(t+t_0),x)-u(pt_0,x)}
(\tau-\tau\circ g_{s(x,x')})(h_t(h_{u(pt_0,x)}x))dt+{\rm O}(\epsilon^2)\leq $$
$$
s(x,x')\left(u(p(t+t_0),x)-u(pt_0,x)\right)^{1/2}\|d/ds(\beta_{\tau})(0,\cdot,\cdot)\|_{C^0(A\times M}+{\rm O}(\epsilon^2)={\rm O}(\kappa)+{\rm O}(\epsilon^2),
$$
where the \Red{(only)} inequality \Red{above holds} since $|u(p(t+t_0),x)-u(pt_0,x)|={\rm O}(t)=\kappa{\rm O}(\bar{T})=\kappa{\rm O}(s(x,x')^{2})$. Therefore,
$\frac{1}{p}|A_x(p(t+t_0))-A_x(pt_0)|={\rm O}(\kappa)+{\rm O}(\epsilon^2)$.
An analogous reasoning for $y$ shows that $\frac{1}{q}|A_x(q(t+t_0))-A_x(qt_0)|={\rm O}(\kappa)+{\rm O}(\epsilon^2)$. Hence, \Red{for every $t_0\in[0,K\bar{T}]$ and every $t\in[0,\kappa t_0]$, we have}
\begin{equation}\label{subat}
|sub(t+t_0)-sub(t_0)|={\rm O}(\epsilon^2).
\end{equation}
%Moreover, there is  $C_{p,q}>0$ such that for each $t\in [0,\bar T]$, we have
%\be\label{eq:cpq5}
%|sub(t)|<\frac{C_{p,q}}2.
%\ee
%Indeed, ?????

By definition, the functions $a(\cdot)$ and $b(\cdot)$ are continuous on $[0,K\bar{T}]$ and therefore to \Red{prove} \eqref{show1}, it is enough to show that there exists $t_1<K\bar{T}$ such that
\begin{equation}\label{fina}
|a(t_1)-b(t_1)|\geq 2|r_{p,q}|.
\end{equation}

We consider two cases:

\textbf{A.} There exists $t'\in [0,2\bar{T}]$ such that
$|pol(t')|>d_{p,q}/2$ (see \eqref{jul2}). Let then $T'>0$ be the smallest number for which $|pol(T')|=d_{p,q}/2$. By Remark~\ref{remarknew}, for the polynomial $k\mapsto pol(kT')$ \Red{and} for every $\tilde{K}>2$, we have
\be\label{fin:eq}
\sup_{k\in[0,\tilde{K}]}|pol(kT')|\geq C^{-2}_\mathcal{P}\tilde{K}\frac{d_{p,q}}{2}.
\ee
Moreover, by \eqref{axt} for $x$, by  \Red{(T2)} and \eqref{bart} (since $T'<2\bar{T}$), for $k\in[0,K]$, \Red{we have}
\begin{multline}\label{fineq2}
\frac{1}{p}|A_x(pkT')|=\\{\rm O}(\epsilon^2)+p^{-1}|s(x,x')|(pkT'+{\rm O}(\epsilon^4pkT'))^{1/2}\|(d/d\theta\beta_\tau)(0,\cdot,\cdot)\|_{C^0(A\times M)}\leq \\ \frac{2k^{1/2}}{p^{-1/2}}\|(d/d\theta\beta_\tau)(0,\cdot,\cdot)\|_{C^0(A\times M)},
\end{multline}
and similarly
\be\label{fineq3}
\frac{1}{q}|A_y(qkT')|\leq  \frac{2k^{1/2}}{q^{-1/2}}|(d/d\theta\beta_\tau)(0,\cdot,\cdot)\|_{C^0(A\times M)}.
\ee
Notice that if $\tilde{K}>0$ is such that
$$
C^{-2}_\mathcal{P}\tilde{K}\frac{d_{p,q}}{4}>2\tilde{K}^{1/2}(\frac{1}{p^{-1/2}}+\frac{1}{q^{-1/2}})|(d/ds\beta_\tau)(0,\cdot,\cdot)\|_{C^0(A\times M)},
$$
then by \eqref{fin:eq}, \eqref{fineq2} and \eqref{fineq3} it follows that there exists $k\in [0,\tilde{K}]$ such that (see \eqref{defa})

$$
|a(kT')-b(kT')|\geq  |pol(kT')|-|sub(kT')|\geq C^{-2}_\mathcal{P}\tilde{K}\frac{d_{p,q}}{4}
$$
what finishes the proof of \eqref{fina} (and hence also the proof of \eqref{show1}, where $t_0\leq kT'$) if $|r_{p,q}|$ is small enough and $\tilde{K}<K$. Moreover, by Remark~\ref{remarknew} and the definition of $T'$, it follows that for every $t\in [0,\kappa t_0]$, we have
$$
|pol(t+t_0)-pol(t_0)|={\rm O}(\epsilon^2).
$$
This and \eqref{subat} (see \eqref{defa}) finishes the proof of \eqref{show2} and the proof of Theorem~\ref{main:prop} is \Red{complete} in this case.

\textbf{B.} For every $t'\in [0,2\bar{T}]$, we have
$|pol(t')|<\frac{d_{p,q}}{2}$. This, by Remark \ref{remarknew}, in particular means that
\begin{equation}\label{distgeo}
|e^{-2n_k^{-1}}-e^{-2s(y,y')}|\leq \frac{C_\mathcal{P}d_{p,q}}{2\bar{T}}
\end{equation}
and
\begin{equation}\label{distho}
|r(y,y')|\leq \frac{C_\mathcal{P}d_{p,q}}{4q\bar{T}^2}.
\end{equation}
We claim moreover that
\be\label{lmh}
2\bar{T}\geq n_k^2.
\ee
Indeed, if not then by \eqref{bart}, $\bar{T}=\min(|r(y,y')|^{-1/2},s(y,y')^{-2})$. If $\bar{T}=s(y,y')^{-2}$, then by \eqref{lmh}, we get $-2n_k^{-1}<-2^{1/2}|s(y,y')|$. But then (by Taylor's formula)
$$
|e^{-2n_k^{-1}}-e^{-2s(y,y')}|\geq e^{-2s(y,y')}-e^{-2^{1/2}s(y,y')}\geq \frac{2-2^{1/2}}{2}|s(y,y')|
$$
which is a contradiction with \eqref{distgeo}. On the other hand, if
$\bar{T}=|r|^{-1/2}$, then we get a contradiction with \eqref{distho} (by making $d_{p,q}$ smaller if necessary). Hence \eqref{lmh} holds.
 Notice that by \eqref{distgeo} and \eqref{distho} (see also Remark \ref{remarknew}), for every $t_0\in[0,2\bar{T}]$ and every $t\in[0,\kappa t_0]$, we have
$$pol(t+t_0)-pol(t_0)={\rm O}(\epsilon^{2}).$$
This and \eqref{subat} (see \eqref{defa}) \Red{show} that  \eqref{show2} holds for every $t_0\leq 2\bar{T}$. So we only have to show that \eqref{fina} holds for some $t_0\leq 2\bar{T}$.

Notice that by Lemma \ref{taylor:form} for $g_{n_k^{-1}}h_ty$
\Red{(recall that $n_k^{-1}=s(x,x')$)} and $g_{s(y,y')^{-1}}h_ty$, \Red{we have}
$$
\frac{1}{q}\left|\int_0^{u(qn_k^2,y)}\left(\tau-\tau\circ g_{n_k^{-1}}\right)(h_ty)dt-\int_0^{u(qn_k^2,y)}\left(\tau-\tau\circ g_{s(y,y')}\right)(h_ty)dt\right|=
$$
$$
\frac1q\left|\int_0^{u(qn_k^2,y)}(s(x,x')-
s(y,y'))(X\tau)(g_{s(y,y')}h_ty)+{\rm O}(\epsilon^3|s(x,x')-s(y,y')|)\,dt\right|.
$$
Moreover, using \Red{the equality} $g_{s(y,y')}h_ty=h_{e^{-2s(y,y')}t}g_{s(y,y')}$, substituting $t'=e^{-2s(y,y')}t$
and using that $\int_M X\tau d\mu=0$, we get that the above expression is equal to
$$
|s(x,x')-s(y,y')|{\rm O}(\epsilon^3 u(qn_k^2,y))=|s(x,x')-s(y,y')|{\rm O}(\epsilon^3n_k^2)={\rm O}(\epsilon^3),
$$
where the last equality follows by~\eqref{distgeo} and~\eqref{lmh}. Therefore, by \eqref{jul2}
and \eqref{axt}, it follows that if we set $t_1=n_k^2$ (we also use \eqref{lmh} to be able to use \eqref{axt}), then
$$
|sub(t_1)|\geq \frac{2d_{p,q}}{3}.
$$
By \eqref{lmh}, it follows that $t_1=n_k^2<2\bar{T}$ and hence \Red{by the assumption in} \textbf{B.}, we have $|pol(t_1)|<\frac{d_{p,q}}{2}$. This implies \Red{that}
$$|a(t_1)-b(t_1)|\geq |sub(t_1)|-|pol(t_1)|\geq \frac{d_{p,q}}{6}$$
and this gives \eqref{fina} (if $|r_{p,q}|$ is small enough) and hence also \eqref{show1}. This finishes the proof in case $\tau$ satisfies \Red{(T2)}.

In case $\tau$ satisfies the assumptions of \Red{(T1) or (T3)}, we define analogously
$$
L:=\{x\in M: |\Xi(x)|> \|\Xi(\cdot)\|_{C^0(M)}-c'\},
$$
where $\Xi=\beta_\tau$ in $(T1)$ and $\Xi=\beta_\tau^{(1/4)}$ in $(T3)$, and $\mu(L)>c$.

In case $\tau$ satisfies the assumptions of \Red{(T1)}, we set $X_k=M$, $A_kx=g_{k^{-1}}x$ and $E_k=g_{-\log pk^{1/\alpha_\tau}}\Red{(L)}$; and $X_k=M$, $A_kx=g_{k^{-1}}x$ and $E_k=g_{-\log (pk^{2}\log k)}\Red{(L)}$. We define $\kappa, \delta,Z$ as above. Finally, we set
 $$
 \bar{T}:=\min(k^{1/\alpha_\tau},r^{-1/2}(y,y),s(y,y')^{-1/\alpha_\tau})
 $$
 if $\tau$ satisfies \Red{(T1)} and
 $$
\bar{T}:=\min(k^{1/\alpha_\tau},r^{-1/2}(y,y),s(y,y')^{-1/\alpha_\tau})
$$
if $\tau$ satisfies \Red{(T3)}. The rest of the proof follows the lines of the proof in case \Red{(T2)}, i.e. we first show \eqref{jul2} with $u(pk^{1/\alpha_\tau})$ and $u(qk^{1/\alpha_\tau})$ in case \Red{(T1)} and $u(pk^{2}\log k)$ and $u(qk^{2}\log k)$ in case \Red{(T3)}. The functions $a(\cdot)$ and $b(\cdot)$ are defined by the same expressions. The proof (with these new definitions) follows  the same lines as the proof in case \Red{(T2)}. This \Red{completes} the proof of Theorem~\ref{main:prop}.
\end{proof}

\section{Disjointness criterion for special flows}\label{s:Sec5}
In this section we assume that $(T_t)=(T_t^f)$ and $(S_t)=(S_t^g)$ are special flows over  ergodic isometries $T\in Aut(X,\cB,\mu,d_1)$, $S\in Aut(Y,\cC,\nu,d_2)$ and under $f\in L^1_+(X,\cB,\mu)$, $g\in L^1_+(Y,\cC,\nu)$, respectively.  Our aim is to give assumptions on the parameters of special flows so that Theorem~\ref{disjoint.flows} applies.

Then $(T_t^f)$ acts on $X^f$ with metric $d_1^f$ (which is the restriction of the product metric) and $(S_t^g)$ acts on $Y^g$ with metric $d_2^g$. For $(x,s)\in X^f$ and $t\in \R$, we denote by $n(x,s,t)\in \Z$ the unique number for which
$$
f^{(n(x,s,t))}(x)\leq t+s<f^{(n(x,s,t)+1)}(x),
$$
i.e.
\be\label{do1}
T^f_{t}(x,s)=(T^{n(x,s,t)}x,s+t-f^{n(x,s,t)}(x)).\ee
We define $m(y,r,t)$  analogously for $(y,r)\in Y^g$.
We assume that
\be\label{do2}
\mbox{$f$ and $g$ are bounded away from zero.}
\ee
Note that
\be\label{do2A}
n(x,s,t+1)-n(x,s,t)\leq 1/\min(1,\min f).\ee

\begin{proposition}[Disjointness criterion for special flows]\label{cocycle} Let $P=\{-p,p\}$ with $p\neq0$, $0<c<1/2$ and let $$\zeta:=\frac{\int_Xf d\mu}{\int_Yg d\nu}.$$

\noindent Assume that we have a sequence of measurable sets $$(X'_k)\subset X,\qquad  \mu(X'_k)\to\mu(X),$$
together with a sequence of automorphisms
$$A'_k\in Aut(X'_k,\cB|_{X'_k},\mu|_{X'_k}),\ k\geq1,\; \text{such\ that}\
 A'_k\to Id \ \text{uniformly}.$$
Assume moreover that for every $\epsilon'>0$ and $N'\in \N$  there exist
$$(E'_k)=(E'_k(\epsilon'))\subset \cB,\; \mu(E'_k)>c \ \text{for} \ k\geq 1,$$
as well as $0<\kappa'=\kappa'(\epsilon')<\epsilon'$, $\delta'=\delta(\epsilon',N')>0$ and a set
$$Z'=Z(\epsilon',N')\subset Y,\qquad  \nu(Z')\geq (1-\epsilon')\nu(Y),$$ such that for all $y,y'\in Z'$ satisfying $d_2(y,y')<\delta'$, every $k$ such that $d_1(A'_k,Id)<\delta'$ and every $x\in E'_k\cap X'_k$, $x':=A'_{k}x$ there are
$$M'\geq N,\; L'\geq 1,\;  \frac{L'}{M'}\geq \kappa'$$
and $p\in P$ for which one of the following sets of estimates, {\bf (F)} or {\bf (B)} (where {\bf F} stands for Forward and {\bf B} for Backward), holds:
\begin{enumerate}
\item[{\rm\bf (F)}]   \textbf{Forward control}:
\begin{enumerate}
\item[{\rm (F1)}] For any $t$ such that $n(x,s,t)\in[M',M'+L']$, we have\footnote{Here $s,r\in\ \R$ are any numbers for which $(x,s)\in X^f$ and $(y,r)\in Y^g$.}
%\begin{equation}\label{forw.hits}
%$$(F1)\qquad
$$\max\left(\left|n(x,s,t)\int_X fd\mu-t\right|, \left|m(y,r,t)\int_Y gd\nu-t\right|\right)<t^{1/3};$$
%\end{equation}
\item[{\rm (F2)}] for $w\in [M',M'+L']\cap \Z$, we have
%~\footnote{Note that property (F1) of Prop.~\ref{cocycle} does not depend on $s$ (on $r$, respectively), but we now work with $x$ ($y$, respectively) fixed, and the difference $n(x,s,t)-n(x,s',t)$ stays bounded, with a bound depending on $x$ and does not depending on $t$. In fact, the bound is determined by~\eqref{do2}; if $f\geq c$ then the bound, say $B$, must satisfy $f(x)\geq Bc$.}
%\begin{equation}\label{forw.coc}
%$$(F2)\qquad
$$\left|(f^{(w)}(x)-f^{(w)}(x'))-(g^{([\zeta w])}(y)-g^{([\zeta w])}(y'))-p\right|<\epsilon';$$
%\end{equation}
\item[{\rm(F3)}] for $w\in [0,\max(1,\zeta)(M'+L')]\cap \Z$, we have
%\begin{equation}\label{forw.jump}
%$$(F3)\qquad
$$ |f(T^{w}x)-f(T^{w}x')|<\kappa(\epsilon')\text{ and } |g(S^{ w}y)-g(S^{ w}y')|<\kappa(\epsilon');$$
%\end{equati \left|f(T^{w}x)-f(T^{w}x')|<\kappa'(\epsilon')^2\text{ and } |g(S^{ w}y)-g(S^{ w}y')\right|<\kappa'(\epsilon')^2
%\end{equation}
\item[{\rm(F4)}] for every  $w,u\in\left[\frac{\zeta M'}{2},2\zeta (M'+L')\right]$, $|w-u|\leq M'^{1/2}$, we have
$$%(F4)\qquad
%\begin{equation}\label{forw.tight}
\left|(g^{(w)}(y)-g^{(w)}(y'))-(g^{(u)}(y)-g^{(u)}(y'))\right|<\epsilon'.$$
%\end{equation}
\end{enumerate}

\item[\text{\bf (B)}] \textbf{Backward control}:
\begin{enumerate}
\item[{\rm(B1)}] For any $t$  such that $n(x,s,-t)\in[M',M'+L']$, we have
%$$(B1)\qquad
%\begin{equation}\label{backw.hits}
$$\max\left(\left|n(x,s,-t)-t\int_X fd\mu\right|, \left|m(y,r,-t)-t\int_Y gd\mu\right|\right)<t^{1/3};
$$%\end{equation}
\item[{\rm(B2)}] for $w\in [M',M'+L']\cap \Z$, we have
%\begin{equation}\label{backw.coc}
$$\left|(f^{(-w)}(x)-f^{(-w)}(x'))-(g^{([-\zeta w])}(y)-g^{([-\zeta w])}(y'))-p\right|<\epsilon';$$
%\end{equation}

\item[{\rm(B3)}] for $w\in [0,\max(1,\zeta)(M'+L')]\cap \Z$, we have
%\begin{equation}\label{backw.jump}
$$\left|f(T^{-w}x)-f(T^{-w}x')|<\kappa(\epsilon')\text{ and } |g(S^{- w}y)-g(S^{- w}y')\right|<\kappa(\epsilon');
$$%\end{equation}
\item[{\rm(B4)}] for every $w,u\in\left[\frac{\zeta M'}{2},2\zeta (M'+L')\right]$, $|w-u|\leq M'^{1/2}$, we have
%\begin{equation}\label{backw.tight}
$$ \left|(g^{(-w)}(y)-g^{(-w)}(y'))-(g^{(-u)}(y)-g^{(-u)}(y'))\right|
<\epsilon'.
$$
%\end{equation}
\end{enumerate}
\end{enumerate}

Then  $(T_t^f)$ and $(S_t^g)$ are disjoint.
\end{proposition}
The rest of this section will be devoted to the proof of this criterion, which will be deduced  from the disjointness criterion for flows given by Theorem~\ref{disjoint.flows}.
\begin{proof} We will show that the assumptions of  Theorem~\ref{disjoint.flows} (with $c$ replaced by $c/2$) are satisfied. Let
$$C_{f,g}:=\max\left(1,\int_X fd\mu,\left(\int_X fd\mu\right)^{-1},\int_Y gd\nu,\left(\int_Y gd\nu\right)^{-1}\right),$$$$ c_{f,g}:=\min(1,\inf_\T f, \inf_\T g)>0.$$
 Fix $\epsilon>0$ and $N\in \N$. Set $\epsilon':=\frac{\epsilon^2}{2c_{f,g}}$. Let $\kappa'=\kappa'(\epsilon')>0$ be obtained from the assumption of our proposition and set
\be\label{adamkappa}\kappa:=\min\left(\frac{\kappa'}{100C^2_{f,g}},\frac1{100},
 \frac1{18(c_{f,g}^{-1}+1)\left(\int f\,d\mu\right)^{2/3}}\right).\ee
By taking $\epsilon'$ still smaller, we can assume that
\be\label{adamkappa1}
\frac{4\kappa'^2}{\int f\,d\mu}<\kappa \epsilon.\ee
Let
$$
W^f(\epsilon):=\left\{(x,s)\in X^f:\;\frac{\epsilon}{100}<s< f(x)-\frac{\epsilon}{100}\right\}
$$
and
$$
W^g(\epsilon):=\left\{(y,r)\in Y^g:\; \frac{\epsilon}{100}<r< g(y)-\frac{\epsilon}{100}\right\}.
$$
For $J\subset \R$ let
$$U_J^f(x,s):=\{t\in J\;:\;T_t^f(x,s)\in W^f(\epsilon)\}, \;\;U_J^g(y,r):=\{t\in J\;:\;S_t^g(y,r)\in W^g(\epsilon)\}.
$$
By Birkhoff ergodic theorem there exist $T_\epsilon>0$, $V_1\subset X^f$, $\mu^f(V_1)\geq (1-\frac{\epsilon}{2})\mu^f(X^f)$ and $V_2\subset Y^g$, $\nu^g(V_2)\geq (1-\frac{\epsilon}{2})\nu^g(Y^g)$ such that for every $(x,s)\in V_1$, $(y,r)\in V_2$, we have
\begin{equation}\label{tgoodset}
\left|U_J^f(x,s)\right|>\left(1-\frac{\epsilon}{2}\right)|J|\text{ and }
\left|U_J^g(y,r)\right|>\left(1-\frac{\epsilon}{2}\right)|J|,
\end{equation}
holds for all $J=[U,U+V]$, with $\frac{|V|}{|U|}\geq \kappa'$, $|U|\geq T_\epsilon$.

Define $N':=[\max(2T_\epsilon,\kappa^{-3}N,\epsilon^{-1},c_{f,g}^{-4})]+1$. Moreover, by enlarging $N'$ if necessary, we have
\be\label{ajajaj}
N_1'^{1/2}\leq \frac{\kappa'}{10}\int f\,d\mu\cdot N'_1\text{ for each }N_1'\geq N'.\ee
 Let
$$X_k:=\{(x,s)\in X^f:x\in X'_k, (A'_kx,s)\in X^f\}$$
and set
$$A_k(x,s):=(A'_kx,s)$$
for $(x,s)\in X_k$. Finally, set
$$E_k=E_k(\epsilon):=\left\{(x,s)\in X^f\;:\;x\in E'_k(\epsilon'), (x,s)\in V_1, s<\frac1{\epsilon^{j_0}}\right\},$$
where $j_0$ is chosen in such a way that whenever for a set $C\subset X$, we have $\mu(C)>c$ then $\mu^f(\{(x,s)\in X^f:\:x\in C, s<1/\epsilon^{j_0}\})>\frac34c$ (the existence of $j_0$ follows immediately from the fact that $f$ is in $L^1$). By the properties of $A'_k,X'_k$ and $E'_k(\epsilon')$ it follows that $A_k,X_k$ and $E_k$ satisfy the assumptions of Theorem~\ref{disjoint.flows}. Let
$$Z(\epsilon,N)=\{(y,r)\in Y^g : y\in Z'(\epsilon',N'), r<\epsilon^{-j_0}\}\cap V_2~\footnote{Again, by changing $j_0$ if necessary, we can assume that $\nu^g(Z(\epsilon,N))>1-\epsilon$.}$$ and
$$\delta(\epsilon,N):=\min(\delta'(\epsilon',N'),
\frac{\epsilon}{1000}).$$ Take $k$ such that $ d_1^f(A_k,Id)<\delta$, $(x,s)\in E_k$ and set $(x',s'):=A_k(x,s)$, take $(y,r),(y',r')\in Z(\epsilon,N)$ for which $d^g_2((y,r),(y',r'))<\delta$.
Then $d_1(A'_k,Id)<\delta, x\in E'_k$, $x'=A'_kx$,
$y,y'\in Z'(\epsilon',N')$ and $d_2(y,y')<\delta'$. Hence A. or B. holds for $(x,x')$ and $(y,y')$. We will assume that A. holds and show that \eqref{forw} holds for $((x,s),(x',s')), ((y,r),(y',r'))$ (if B. holds we argue analogously to show that \eqref{backw} holds). Since A. holds, we obtain $M',L'$ and $p$.

Let $M,L\in \N$ be any numbers such that $n(x,s,M)\in [M'+1,M'+c^{-1}_{f,g}+1]\cap \Z$,  and $n(x,s,M+L)\in [M'+L'-c_{f,g}^{-1}-1, M'+L'-1]\cap \Z$. Notice that such $M,L$ always exist. Indeed, this follows by~\eqref{do2} and \eqref{do2A}: for all $(x,s)\in X^f$, $n(x,s,u+1)-n(x,s,u)\leq c_{f,g}^{-1}$.

We also have:
\begin{equation}\label{eq:st}
\text{if } t\in [M,M+L]\text{ then } n(x,s,t)\in [M',M'+L'].
\end{equation}
We will show that $M\geq N$, $\frac{L}{M}\geq \kappa$. By~property (F1) (and~\eqref{eq:st}) for $t=M$, it follows that
\begin{equation}\label{eq:boun}2 M\geq M+M^{1/3}\geq n(x,s,M)\int_Xfd\mu \geq M'\int_Xfd\mu.
\end{equation}
Hence $M\geq N$ (since $M'\geq N'\geq \kappa^{-3}N$, and use~\eqref{eq:boun}). Moreover,
\begin{equation}\label{eq:boun2}
M'\geq \frac{n(x,s,M)}{2}\geq\\ \frac{M-M^{1/3}}{2\int_Xf d\mu}\geq \frac{M}{3\int_Xf d\mu}~\footnote{Use $M'\geq N'\geq c_{f,g}^{-4}+1$ which implies $2M'\geq M'+c_{f,g}^{-1}+1\geq n(x,s,M)$, then~property (F1) and finally that $M\geq6$.}
\end{equation}
and by property (F1) of Prop.~\ref{cocycle} for $t=M+L$, we have
\begin{equation}\label{eq:ml}
\frac{M+L}{2}\leq (M+L) -(M+L)^{1/3}\leq n(x,s,M+L)\int_Xf d\mu \leq (M'+L')\int_Xf d\mu.
\end{equation}

%~\footnote{We cannot claim to obtain $[M']+1$ as we sample at $f^{(n(x,s,m))}(x)$, $f^{(n(x,s,m+1))}(x)$, $f^{(n(x,s,m+2))}(x)$, etc., however the jump $n(x,s,m+1)-n(x,s,m)$ we can make is bounded by $1/c$,  hence we can achieve one of the numbers $[M']+1,\ldots, [M']+[1/c]$.}
%\footnote{The same applies here.}
Again by property (F1) first for $t=M+L$ and then $t=M$, the definition of $M$ and $L$ and $\kappa$ (to see that $2(c_{f,g}^{-1}+1)\int f\,d\mu\leq M^{1/3}< (M+L)^{1/3}$), by~\eqref{eq:ml} (and the definition of $\kappa$ to see that $3(2(M'+L')\int f\,d\mu)^{1/3}\leq (M'+L')^{1/2}$),~\eqref{ajajaj},~\eqref{eq:boun2}, we have
\begin{multline*}
L=(M+L)-M\geq\\
n(x,s,M+L)\int_Xf d\mu-(M+L)^{1/3}-n(x,s,M)\int_Xf d\mu-M^{1/3}\geq\\  \left(n(x,s,M+L)-n(x,s,M)\right)\int_Xf d\mu- 2(M+L)^{1/3}\geq\\ \left((M'+L')-c_{f,g}^{-1}-1-M'-c_{f,g}^{-1}-1\right)\int_Xf d\mu- 2(M+L)^{1/3}\geq L'\int_Xf d\mu-3(M+L)^{1/3}\geq\\
L'\int_Xf d\mu-(M'+L')^{1/2}\geq L'\int f\,d\mu-\frac{\kappa'}{10}\int f\,d\mu(L'+M')\geq\\
\frac23L'\int f\,d\mu-\frac{\kappa'}{10}\int f\,d\mu\cdot M'\geq\\
\left(\int f\,d\mu\cdot \kappa'M'\right)\left(\frac23-\frac1{10}\right)\geq \frac12\int f\,d\mu\cdot \kappa'M'\geq\\
\frac{\int f\,d\mu}2\kappa'\frac{M}{3\int f\,d\mu}\geq \kappa M,
\end{multline*}
the last inequality by the definition of $\kappa$.

%~\footnote{This and the following estimates I didnot understand (though I agree with them up to some ,,constants"). Usually when one wants to see a relationship between $M$ and $M'$, one writes
%$$n(x,s,M)c\leq f^{(n(x,s,M))}(x)\leq s+M\leq f(x)+M,
%$$
%whence $n(x,s,M)c\leq f(x)+M$  ($c$ stands for the min of $f$) whence a lower bound on $M$ since $n(x,s,M)$ is ,,about'' $M'$. I think, we have to write more precise relations between $M,L$ and $M',L'$, and then use~property (F1) of Prop.~\ref{cocycle} which tells us that up to some multiplicative constants $M$ and $M'$ are ,,equal".}

Let $U:=U^f_{[M,M+L]}(x,s)\cap U^g_{[M,M+L]}(y,r)$. By \eqref{tgoodset}, we have $|U|\geq (1-\epsilon)L$. By the definition of $U$ (see the definitions of $U^f_J,U^g_J$ and $W^f,W^g$), it follows that $U$ is a disjoint union of intervals, $U=\bigcup_{i=1}^v(c_i,d_i)$ and
\begin{equation}\label{eq:smh}
0\leq n(x,s,c_{i+1})-n(x,s,d_i)\leq 1.
\end{equation}

Moreover, $U^f_{[M,M+L]}(x,s)$ is a union of intervals of length $\geq \frac{\inf_\T f}{2}\geq \frac{c_{f,g}}{2}$ and similarly $U^g_{[M,M+L]}(y,r)$ is a union of intervals of length $\geq \frac{\inf_\T g}{2}\geq \frac{c_{f,g}}{2}$. Therefore the number of intervals in  $U^f_{[M,M+L]}(x,s)$ and $U^g_{[M,M+L]}(y,r)$ is bounded by $\frac{2L}{c_{f,g}}$. So $v\leq \frac{4}{c_{f,g}}L$.
%\footnote{OK:). On the other hand if in the definitions of $W^f(\epsilon)$, $W^g(\epsilon)$ we additionally require that $x$ ($y$, resp.) is a point of continuity of $f$  ($g$) then $U$ is even open, hence a disjoint union of intervals.}
 Define
$$a(t):=t+f^{(n(x,s,t))}(x')-f^{(n(x,s,t))}(x)-p.$$
Notice that $a(t)=t+R_i$ on each $(c_i,d_i)$ (since $n(x,s,\cdot)$ is constant on $(c_i,d_i)$) and (cf.~\eqref{eq:smh})
\begin{multline*}
\left|R_{i+1}-R_i\right|=\\
\left|\left(f^{(n(x,s,c_{i+1}))}(x')-f^{(n(x,s,c_{i+1}))}(x)\right)-
\left(f^{(n(x,s,d_i))}(x')-f^{(n(x,s,d_i))}(x)\right)\right|\leq\\
|f(T^{n(x,s,d_i)}x)-f(T^{n(x,s,d_i)}(x'))|\leq \epsilon',
\end{multline*}
the last inequality by~(F3),~\eqref{eq:st} and the definition of $\kappa'$. Using~(F3) and~\eqref{eq:st} again, in view of~\eqref{eq:boun2} and~\eqref{adamkappa1}, we have
$$
|R_1|=|f^{(n(x,s,c_{1}))}(x')-f^{(n(x,s,c_{1}))}(x)-p|\leq$$$$
n(x,s,c_1)\kappa'^2+|p|\leq 2M'\kappa'^2\leq \frac{4\kappa'^2}{\int f\,d\mu} M<\epsilon L.
$$

%/ Hence
%$$
%\sum_i=1^v|a(c^+_{i+1})-a(d_i^-)|\leq v\epsilon'\leq \frac{2L\epsilon'}{c_{f,g}}<\epsilon L,
%$$
%the last inequality by the definition of $\epsilon'$.
%\footnote{Is this definition of $a$ correct? Looking at the definition of $b$ below, one has an impression, we should use $n(x',s',t)$. But this would cost a loose of control on $U$. So I assume the def. of $a$  is correct (in fact, all calculations below show that we use $n(x,s,t)$ for $x$ and $x'$:)). It looks to me that $a$ is not constant however. It is rather the faction $n(x,s,t)$ which is constant on the components of $U$ (correct?). Which results in saying that the derivative of $a$ is constant...}
% Moreover, by definition, as $t\in [M,M+L]$, then $0\leq %n(x,s,t)\leq M'+L'$. Therefore, by (F3) of Prop.~\ref{cocycle}, we get
%$$|f^{(n(x,s,t))}(x')-f^{(n(x,s,t))}(x)|<\epsilon'(M'+L')\leq %\frac{3\epsilon' M'}{2}.
%$$
%Hence (since $\frac{p}{t}<\frac{\epsilon'}{2}$) it follows that $1-2\epsilon'<\frac{a(t)}{t}<1+2\epsilon'$ for $t\in U$.
%$~\footnote{I do not see it,~(F3) of Prop.~\ref{cocycle} gives us some control for the values $f(T^nx),f(T^nx')$ for $n\in [M',M'+L']$, but we need to control $f^{(n(x,s,t))}(x')-f^{(n(x,s,t))}(x)$, how is it done? I understand that roughly when $t$ runs over $[M,M+L)$ then $n$ runs over $[M',M'+L']$.} and \eqref{vari} is satisfied for $a$ on $[M',M'+L']$.~\footnote{Should be $[M,M+L]?$ But why it it so?}
So $(a,U,\frac{c_{f,g}}{4})$ is $\epsilon$-good. Let us also define $b(t):=t+g^{(m(y,r,t))}(y')-g^{(m(y,r,t))}(y)$. We will show that for $t\in U$, we have
\begin{equation}\label{closx}
d_1^f(T_t^f(x,s),T_{a(t)+p}^f(x',s'))<\epsilon/200 \text{ and }
d_2^g(S_t^g(y,r),S_{b(t)}^g(y',r'))<\epsilon/200.
\end{equation}
Let us show the first inequality, the proof of the second being analogous.\footnote{In the calculation below $s=s'$ but, in general, $r\neq r'$.} By the definition of special flow and since $t\in U$, we have
\begin{equation}\label{spec.flow}
f^{(n(x,s,t))}(x)+\frac{\epsilon}{100}\leq t+s<f^{(n(x,s,t)+1)}(x)-\frac{\epsilon}{100}.
\end{equation}
Notice that
\begin{equation}\label{loc.est}
f^{(n(x,s,t))}(x')<a(t)+p+s'<f^{(n(x,s,t)+1)}(x').
\end{equation}
Indeed, by \eqref{spec.flow} we have
$$
a(t)+p+s'=t+s-f^{(n(x,s,t))}(x)+(s'-s)+f^{(n(x,s,t))}(x')\geq \frac{\epsilon}{100}-\delta+f^{(n(x,s,t))}(x')\geq f^{(n(x,s,t))}(x'),
$$
since $\delta<\frac{\epsilon}{100}$.
Similarly, by the cocycle identity\footnote{Applied to RHS of~\eqref{spec.flow}.} and \eqref{spec.flow}
\begin{multline*}
a(t)+p+s'=t+s-f^{(n(x,s,t))}(x)+(s'-s)+f^{(n(x,s,t))}(x')<\\
f(T^{n(x,s,t)}x)+
f^{(n(x,s,t))}(x')-\frac{\epsilon}{100}+\delta<f^{(n(x,s,t)+1)}(x'),
\end{multline*}
the last inequality since $|f(T^{n(x,s,t)}x)-f(T^{n(x,s,t)}x')|<\frac{\epsilon}{1000}$ (see (F3)) and $\delta<\frac{\epsilon}{1000}$ what completes the proof of~\eqref{loc.est}.
Therefore, by the definition of special flow and \eqref{loc.est}, we have
$$
T^f_t(x,s)=(T^{n(x,s,t)}x,t+s-f^{(n(x,s,t))}(x))$$
and
$$T^f_{a(t)+p}(x',s')=(T^{n(x,s,t)}x',a(t)+p+s'-f^{(n(x,s,t))}(x')).
$$
Hence,  the first inequality in~\eqref{closx} follows because of the definition of $a(t)$, since $d_1^f$ is the product metric, $T$ is an isometry and $\delta<\epsilon/1000$. Analogously, we show the second inequality which completes the proof of~\eqref{closx}.

Notice that by~\eqref{closx}, for $t\in U$, by the triangle inequality\footnote{Here $\partial Y^g:=\{(y,0):\:y\in Y\}\cup\{(y,g(y)):\:y\in y\}$.}, we have
\begin{equation}\label{dist.par}
d_2^g(S^g_{b(t)}(y',r'),\partial Y^g)\geq d_2^g(S^g_{t}(y,r),\partial Y^g)-d_2^g(S^g_{b(t)}(y',r'),S^g_{t}(y,r))\geq \epsilon/200.
\end{equation}
Therefore, by the triangle inequality, for $t\in U$
\begin{multline*}
d_2^g(S^g_t(y,r),S^g_{a(t)}(y',r'))\leq\\ d_2^g(S^g_t(y,r),S^g_{b(t)}(y',r'))+
d_2^g(S^g_{b(t)}(y',r'),S^g_{a(t)-b(t)}(S^g_{b(t)}(y',r')))\leq\\ \epsilon/200+d_2^g(S^g_{b(t)}(y',r'),S^g_{a(t)-b(t)}(S^g_{b(t)}(y',r'))).
\end{multline*}
So to finish the proof of our proposition, by \eqref{dist.par}, it is enough to show that
$$
|b(t)-a(t)|<\epsilon/300 \text{ for } t\in U.
$$
which by the definition of $a(t)$, $b(t)$, follows by showing
$$
|(g^{(m(y,r,t))}(y')-g^{(m(y,r,t))}(y))-
(f^{(n(x,s,t))}(x')-f^{(n(x,s,t))}(x))-p|<\epsilon/300$$
 whenever $t$ is such that $n(x,s,t)\in [M',M'+L']\cap \Z$ (see \eqref{eq:st}).
But by property (F1) and~\eqref{eq:boun2},
$$| m(y,r,t)-\zeta n(x,s,t)|=\frac1{\int g\,d\nu}\left|m(y,r,t)\int g\,d\nu-n(x,s,t)\int f\,d\mu\right|\leq$$
$$2t^{1/3}\leq 4M^{1/3}<M'^{1/2},$$
so by (F4), the above follows by
$$
|(g^{([\zeta w])}(y')-g^{([\zeta w])}(y))-
(f^{(w)}(x')-f^{(w)}(x))-p|<\epsilon/600$$  for  $w=n(x,s,t)\in [M',M'+L']\cap \Z$,
which in turn follows from~property (F2)  since $\epsilon'<\epsilon/1000$. Hence,~\eqref{forw} holds and the proof is complete.
\end{proof}

\section{Arnol'd flows and Birkhoff sums estimates}\label{s:Sec6}
\subsection{Definition of a class of Arnol'd flows}

The class of  Arnol'd special flows $(R_\alpha)_t^{f}$ acting on $\T^f$ which we consider consists of special flows over a rotation $R_\alpha$ such that $\alpha$ and $f$ { satisfy the following assumptions}. % are described in the two paragraphs below.

\smallskip
\noindent {\bf {Assumptions on the} base rotation.}
Let $\alpha\in \T\setminus \Q$ and let $(q_n)$ denote the sequence of denominators of $\alpha$.

\begin{definition}[the Diophantine condition $\cD$]\label{rotationassumptions}\em
We say that $\alpha\in \cD$ if $\alpha\in \mathbb{T}$ and $\alpha$ satisfies:
\begin{enumerate}
    \item[($\cD$1)] Let $K_\alpha:=\{n\in\N:\: q_{n+1}\leq q_n\log^{7/8}q_n\}$. Then
$$
\sum_{i\notin K_\alpha}\log^{-7/8}q_i<+\infty.
$$
 	 \item[($\cD$2)] There exists a subsequence $(n_k)$ such that $$q_{n_k+1}\geq q_{n_k}\log(q_{n_k})\log(\log q_{n_k}), \;  \text{for} \ k\geq1.$$
    \item[($\cD$3)] For every $n\geq 1$ sufficiently large, we have $q_{n+1}\leq q_n\log^{2}q_n$.

   \noindent {\Blue Thus, for some constant $D_\alpha>0$, we have  $q_{n+1}\leq D_\alpha q_n\log^{2}q_n$ for every $n \in \mathbb{N}$.}
\end{enumerate}
\end{definition}

{Let us remark that condition ($\cD$1) first appeared in the work \cite{Fa-Ka} by Fayad and the first author, where it was introduced since it plays a crucial role in proving the SR-property.}
\begin{lemma}
The set $\cD$ has full Lebesgue measure in $\mathbb{T}$.
\end{lemma}
\begin{proof}
 We claim that ($\cD$2), ($\cD$3) are also full (Lebesgue) measure conditions by Khinchine's theorem. { Indeed,} recall that Khintchine theorem states that, given $\psi:\N\to\R^+$, the inequality
\begin{equation}\label{BC}
\left|\alpha-\frac{ p}{q} \right|<\frac{\psi(q)}{q}
\end{equation}
is satisfied infinitely often for Lebesgue a.e.\ $\alpha$ if $\sum_{q}\psi(q)=\infty$, or Lebesgue a.e.\ $\alpha$ does not satisfy the inequality infinitely often if the series converges. For ($\cD$2), consider $\psi(q)=\frac{1}{q\log q\log\log q}$ and use Legendre theorem to conclude that for the infinitely many solutions $p_{n_k}/q_{n_k}$ of \eqref{BC} we have that $q_{n_k}$ are denominators of $\alpha$. For ($\cD$3), one can reason analogously using the function $\psi(q)=\frac{1}{q \log^2q}$.
\end{proof}

\smallskip
\noindent {\bf Assumptions on the roof function. }
We assume that $f:\T\to\R^+$ is of the form:
\be\label{formroof}
f(x)=A_-(-\log x)+A_+(-\log(1-x))+g(x),\ee
where $g\in C^4(\T)$ and
$$A_-,A_+>0\text{ and } A_-\neq A_+.$$ It follows that
$f\in C^4(\T\setminus \{0\})$, $f>0$ and the behavior around $0$ is given by the following:

\begin{enumerate}
\item[(R1)] $\lim_{x\to 0^+}\frac{f(x)}{-\log(x)}=A_- \text{ and }\lim_{x\to 0^+}\frac{f(x)}{-\log(1-x)}=A_+$;
\item[(Rj)] $\lim_{x\to 0^+}\frac{\frac{d^jf}{dx^j}(x)}{(-1)^jx^{-j}}=(j-1)!A_- \text{ and }\lim_{x\to 0^+}\frac{\frac{d^jf}{dx^j}(x)}{(1-x)^{-j}}=(j-1)!A_+\text{ for }j=1,2,3,4$.
%\end{itemize}
%\item[$F3$.] $\lim_{x\to 0^+}\frac{f''(x)}{x^{-2}}=A_- \text{ %and }\lim_{x\to 0^+}\frac{f''(x)}{(1-x)^{-2}}=A_+$;
%\item[$F4$.] $\lim_{x\to 0^+}\frac{f'''(x)}{-x^{-3}}=-2A_- %\text{ and }\lim_{x\to 0^+}\frac{f'''(x)}{(1-x)^{-3}}=2A_+$;
%\item[$F5$.] $\lim_{x\to 0^+}\frac{f''''(x)}{x^{-4}}=6A_- %\text{ and }\lim_{x\to 0^+}\frac{f''''(x)}{(1-x)^{-4}}=6A_+$;
\end{enumerate}

\subsection{Denjoy-Koksma estimates}\label{sec.DK}
 In what follows, for simplicity of notation, we assume that $\int_\T f\,d\lambda=1$.
For $x\in\T$ and $n\in \N$ define $\inx=i(n,x)$  so that $0\leq \inx < q_n $ satisfies
\be\label{ee3}
\|x+\inx \alpha\|=\min_{0\leq j<q_n}\|x+j\alpha\|.
\ee
Set also
\begin{equation}\label{eq.distsing}
B_{n,x}:= { q_n \ \|x+\inx \alpha\| = }\min_{0\leq j <q_n} q_n\ \|R_\alpha^{j}x\| . %\qquad \text{so} \quad {  \|x+\i \alpha\|  =
\end{equation}
{The quantity $B_{n,x}$ will play a crucial role in all the following estimates of the Birkhoff sums $f$ and its derivatives, since these depend on the contribution of the closest visit to the origin of $\mathbb{T}$, given by $\|x+\inx \alpha\| ={B_{n,x}}/{q_n}$.}
%\begin{equation}
%\|x+\inx \alpha\| =\frac{B_{n,x}}{q_n}.
%\end{equation}  }

\begin{remark}{\em \label{B_order_constant}
 Remark that we have
$$0<B_{n,x}<1\;\;\text{ for each }n\geq0.$$
This can be seen by considering the interval  $(-1/q_n,1/q_n)$ which has length $1/2q_n$, and applying Lemma~\ref{spacing_orbit}, which in particular guarantees that there exists a point of the orbit $\{ x+j \alpha, \ 0\leq j < q_n\}$ which enters it and thus gives $B_{n,x}<1$.
}\end{remark}

\subsection{Estimates of Birkhoff sums at return times}\label{specialtime:sec}
The following lemma, which provides estimates for Birkhoff sums of $f$ and its derivatives at special times {\Blue given by the the denominators of $\alpha$ (which corresponds dynamically to closest returns of the orbit to zero)}, is a consequence of the Denjoy-Koksma inequality.
\begin{lemma}[{Special time estimates}]\label{lem:DK}There exists $C(f)>0$ such that for every $n\in \N$ and $x\in \T$,  we have:
 \begin{equation}\label{eq:DKf}
\left|f^{(q_n)}(x)-q_n\right|<C(f)(\log q_n+ |\log B_{n,x}|),
\end{equation}
\begin{equation}\label{eq:DKf'}
\left|f'^{(q_n)}(x)-(A_- - A_+)q_n\log q_n\right|<C(f)q_n\left(1+ \frac1{B_{n,x}}\right),
\end{equation}
\begin{equation}\label{eq:DKf''}
\left|f''^{(q_n)}(x)-f''(x+i\alpha)\right|\leq C(f)q_n^2,
\end{equation}
\begin{equation}\label{eq:DKf'''}
\left|f'''^{(q_n)}(x)-f'''(x+i\alpha)\right|\leq C(f)q_n^3,
\end{equation}

\begin{equation}\label{eq:DKf''''}
\left|f''''^{(q_n)}(x)-f''''(x+i\alpha)\right|\leq C(f)q_n^4
\end{equation}
for all $i=0,\ldots, q_n-1$.
\end{lemma}
{ Since the function $f$ and its derivatives do not have bounded variation, the proof is based on defining suitable \emph{truncations} to which Denjoy-Koksma inequality can be applied.}
\begin{proof}{ For given $n \in \mathbb{N}$}, let us consider the { truncation} $\bar{f_n}(x):=\raz_{[-\frac{1}{4q_n},\frac{1}{4q_n}]}(x)f(x)$.
Then, { since (by Lemma~\ref{spacing_orbit}) there is at most one point (namely $x+\inx \alpha$) of the form $x+j\alpha$, $j=0,\ldots,q_n-1$, in the interval $[-\frac1{4q_n},\frac1{4q_n}]$}, we have
\be\label{ee4}
\left|\bar{f_n}^{(q_n)}(x)-f^{(q_n)}(x)\right|=
\left\{\begin{array}{cc}
0&\mbox{if $x+i\alpha\in[-\frac1{4q_n},\frac1{4q_n}]$,}\\
f(x+i\alpha)&\text{otherwise}.\end{array}\right.
\ee
Moreover, {recalling \eqref{eq.distsing} and form of $f$ (see \eqref{formroof}),}
$$
f(x+i_{n,x} \alpha)\leq C_{A_-,A_+,g}(f(\|x+i_{n,x}\alpha\|) = {C}_{A_-,A_+,g}\left(f\left(\frac{B_{n,x}}{q_n}\right)\right)\leq C_f (\log q_n+|\log B_{n,x}|).
$$
Besides, by the Denjoy-Koksma inequality for $\bar{f_n}$, we get
$$
\left|\bar{f_n}^{(q_n)}(x)-q_n\int_\T\bar{f_n}\,d\lambda\right|<
2{\rm Var}(\bar{f_n}).
$$
Now, $$
\left|\int_\T\bar{f_n}\,d\lambda-1\right|=
\left|\int_\T\bar{f_n}\,d\lambda-\int_\T fd\lambda\right|=
$$
$$
A_-\int_0^{1/(4q_n)}\log t\,dt+A_+\int_{1-1/(4q_n)}^1\log(1-t)\,dt+{\rm O}(1/q_n)=$$$$
A_-\int_0^{1/(4q_n)}(\log t+1)\,dt+A_+\int_{1-1/(4q_n)}^1(\log(1-t)-1)\,dt+{\rm O}(1/q_n)=$$$$
A_-(t\log t)|_0^{1/(4q_n)}+A_+((1-t)\log(1-t))|_{1-1/(4q_n)}^{1}+{\rm O}(1/q_n)={\rm O}\left(\frac{\log q_n}{q_n}\right).$$
Since ${\rm Var}(\bar{f_n})={\rm O}(\log q_n)$, this completes the proof of~\eqref{eq:DKf}.

\smallskip
We have $f'(x)=\frac{-A_-}x+\frac{A_+}{1-x}+g'(x)$. Now, \eqref{ee4} holds with $f$ replaced by $f'$, so
$$
\left|\bar{f'_n}^{(q_n)}(x)-f'^{(q_n)}(x)\right|\leq\left|f'\left(
\frac{B_{n,x}}{q_n}\right)\right|={\rm O}\left(\frac{q_n}{B_{n,x}}\right).$$
Moreover, by the Denjoy-Koksma inequality
for $\bar{f'_n}$, we get
$$
\left|\bar{f'_n}^{(q_n)}(x)-q_n\int_\T\bar{f'_n}\,d\lambda\right|<
2{\rm Var}(\bar{f'_n}).$$ Then ${\rm Var}\bar{f'_n}={\rm O}(q_n)$ and
$$
q_n\int_{\T}\bar{f'_n}\,d\lambda=q_n\int_{1/(4q_n)}^{1-1/(4q_n)}
f'\,d\lambda=
q_n\left(f\left(1-\frac1{4q_n}\right)-
f\left(\frac1{4q_n}\right)\right)=
$$ $$
q_n(-A_+\log q_n+A_-\log q_n+{\rm O}(1))=(A_--A_+)q_n\log q_n+{\rm O}(q_n).
$$
Hence, \eqref{eq:DKf'} follows.

\smallskip
We now study the second derivative: $f''(x)=\frac{A_-}{x^2}+\frac{A_+}{(1-x)^2}+g''(x)$. Again, \eqref{ee4} holds with $f$ replaced by $f''$. This we can write as
$$
f''^{(q_n)}(x)-\raz_{\left[\frac{-1}{4q_n},\frac{1}{4q_n}\right]}
(x+i\alpha)f''(x+i\alpha)=\bar{f''_n}^{(q_n)}(x).$$
Hence, by the Denjoy-Koksma inequality (for $f_n''$), we obtain
$$
\left|f''^{(q_n)}(x)-\raz_{\left[\frac{-1}{4q_n},\frac{1}{4q_n}\right]}
(x+i\alpha)f''(x+i\alpha)\right|\leq q_n\int_{\T}\bar{f''_n}\,d\lambda+2{\rm Var}(\bar{f''_n}).$$
Now, ${\rm Var}(\bar{f''_n})={\rm O}(q_n^2)$ and
$$
q_n\int_{\T}\bar{f''_n}\,d\lambda=q_n\left(
f'\left(1-\frac1{4q_n}\right)-f'\left(\frac1{4q_n}\right) \right)={\rm O}(q_n^2).$$
It follows that
$$
\left|f''^{(q_n)}(x)-\raz_{\left[\frac{-1}{4q_n},\frac{1}{4q_n}\right]}
(x+i\alpha)f''(x+i\alpha)\right|={\rm O}(q_n^2).$$
Therefore, if
$x+i\alpha\in\left[\frac{-1}{4q_n},\frac1{4q_n}\right]$ then~\eqref{eq:DKf''} follows, while otherwise $f''(\|x+\i\alpha\|)={\rm O}(q_n^2)$, so~\eqref{eq:DKf''} holds again.

The proofs of remaining inequalities follow the same lines.
\end{proof}

\subsection{Estimates on Birkhoff sums of $f$}
In this subsection, we prove the following estimates on $f^{(n)}$.
\begin{lemma}[{Birkhoff sums of $f$}]\label{lem:dk2} There exists $T_0>0$ such that for every $T\geq T_0$ and $x\in \T$ satisfying
$$
\{x+j\alpha:\:0\leq j\leq T\}\cap \left[-\frac{1}{2T\log^4T},\frac{1}{2T\log^4T}\right],
$$
we have
$$%\begin{equation}\label{eq:tight}
|f^{(n)}(x)-n|<T^{1/5} \text{ for all }n\in [0,T]\cap \Z.
$$ %\end{equation}
\end{lemma}
\begin{proof}Let $n=\sum_{j=1}^kb_jq_j$, where $0\leq b_j\leq \frac{q_{j+1}}{q_j}$ with $b_k>0$, be the Ostrowski expansion of~$n$. Then
$$
f^{(n)}(x)=\sum_{j=1}^kf^{(b_jq_j)}\left(x+
\left(\sum_{i=0}^{j-1}b_iq_i\right)\alpha\right)=
\sum_{j=1}^k\left( \sum_{s=0}^{b_j-1}f^{(q_j)}(z_{j,s})\right),$$
where $z_{j,s}=x+\left(\sum_{i=0}^{j-1}b_iq_i\right)\alpha+sq_j\alpha$.
Then, by the estimates given by Lemma~\ref{lem:DK} (see in particular \eqref{eq:DKf}), for $j\in \{1,...,k\}$, we have
$$
|f^{(q_j)}(z_{j,s})-q_j|\leq C(f)(\log q_j-\log B_{j,z_{j,s}}).
$$
Hence
$$
\left|f^{(n)}(x)-n\right|=\left|
\sum_{j=1}^k\left(\sum_{s=0}^{b_j-1}\left(f^{(q_j)}(z_{j,s})-q_j
\right)\right)\right|\leq C(f)\sum_{j=1}^k\sum_{s=0}^{b_j-1}\log\frac{q_j}{B_{j,z_{j,s}}}.$$
Using our assumption, we have
$$
\frac{q_j}{B_{j,z_{j,s}}}=\frac1{\|z_{j,s}+i(j, z_{j,s}) \alpha\|}\leq 2T\log^4T,$$
(we will determine $T_0$ later), where $i=i(z_{j,s},j)$. Therefore,
$$
\left|f^{(n)}(x)-n\right|\leq C(f)\left(\sum_{j=1}^kb_j\right)\log(2T\log^4 T)\leq 2C(f)k\left(\max_{1\leq j\leq k}b_j\right)\log T$$
if $T$ is sufficiently large. The latter expression is bounded from above by $ C\log^4T$ (which { in turn is bounded by } $ T^{1/5}$) for a certain constant $C>0$ since, given that the sequence of denominators grows exponentially fast,
$k\leq C_1(\alpha)\log q_k\leq C_1(\alpha)\log T$ and since by~($\cD$3) (of Definition \ref{rotationassumptions}), $b_j\leq \frac{q_{j+1}}{q_j}\leq C(\alpha)\log^2q_j$ and $q_k\leq T$, we have%\marginpar{square on $log q_k$ in eq was missing}
$$\max_{1\leq j\leq k}b_j\leq C(\alpha)\log^{ 2} q_k\leq C(\alpha)\log^2T.$$
Combining these inequalities, this concludes the proof.
\end{proof}

\subsection{Estimates on Birkhoff sums of $f'$}
The Birkhoff sums of $f'$ are the most delicate to control, since $f'$ is \emph{not} integrable and we need very precise control on the growth in order to later exploit cancelations between different shearing rates. Let us first state the two estimates on Birkhoff sums of $f'$  which will be used in the rest of the paper (Lemma~\ref{lem:bsums} and Lemma \ref{lem:canc}).

The first estimate (in Lemma~\ref{lem:bsums}) provides a fine control of ${f'}^{(r)}$ (of order $r\log r$ with optimal control on the constants) as long as one assumes that the points in the orbit of $x$ of length $r$ stay sufficiently far from the singularity. Estimates similar to Lemma~\ref{lem:bsums} but in the more general context of interval exchange transformations were proved by the third author in \cite{Ul} (see Proposition 3.4 in \cite{Ul}, as well as its generalization by Ravotti in \cite{Ra} and Proposition 4.4 in~\cite{Ka-Ku-Ul}) and inspired in turn by the work of Kochergin on rotations, see e.g.~\cite{Ko}.% (see also Prop.~5.3 in \cite{Ul} in the context of IETs (see also Proposition ADD in ).\marginpar{Refs to add (C)}

Given a constant $M>0$, for arbitrary $n\geq0$, define
\begin{equation}\label{eq:sigma}
\Sigma_n(M):=\bigcup_{i=0}^{Mq_{n+1}}R_\alpha^{-i}\left(
\left[-\frac{1}{q_n\log^{7/8 }q_n},\frac{1}{q_n\log^{7/8} q_n}\right]\right).
\end{equation}
The points in $\Sigma_n(M)$ are those whose orbit get close to the singularity and that should be avoided to have the following estimate (which is proved later in this section).

%{ WHY IN LEMMA 6.4 $M>0$ and in 6.5 IT IS $M>1$?}

\begin{lemma}[{Control of Birkhoff sums of $f'$ far from singularities}] \label{lem:bsums}
{\Blue Fix any  $M>1$.} There exists $\epsilon_0 = \epsilon_0(f)$ such that for every $0<\epsilon<\epsilon_0$ there exists $\N\ni n_0=n_0(\epsilon)$ such that for all $n\geq n_0$, if %\marginpar{was $\epsilon^2, \epsilon^4$ but could be  $\epsilon$?)}
$\epsilon^{4}q_n\leq r\leq Mq_{n+1}$ and  $x\notin \Sigma_n(M)$,
we have
\begin{equation}\label{eq:dercont}
((A_--A_+)-\epsilon^2)r\log r\leq f'^{(r)}(x)\leq ((A_--A_+)+\epsilon^2)r\log r.
\end{equation}
Moreover, for every $r\in[0,\epsilon^{4}q_n]\cap \Z$, we have%\marginpar{Need $\epsilon<\epsilon_0$ small w.r.t $C'(f)$...}
\begin{equation}\label{eq:smallder}
|f'^{(r)}(x)|<\epsilon^2q_{n}\log q_{n}.
\end{equation}
\end{lemma}

The other estimate on the Birkhoff sums of $f'$ which we need later (and is also proved
%deduced from Lemma~\ref{lem:bsums_basic} (see the proof
later in this section) is the following.
\begin{lemma}[{Control of Birkhoff sums of $f'$ on good time scales}]\label{lem:canc} Assume  that $M>1$. There exists {\Blue $n_1$} such that for every $n\in \N$, $n\geq {\Blue n_1}$, any $x \notin \Sigma_n(M)$ and any integers $T$ and $r$ such that
$$q_n\log q_n\leq T < q_{n+1}, \qquad 0\leq r < \frac{B_{n,x}T}{2},$$
there exists a constant $C'(f)$ such that we have
\begin{equation}\label{eq:canc}
|f'^{(r)}(x)-(A_--A_+)r \log { q_{n}}|\leq {  C'(f) T}.
\end{equation}
%where $C'(f)$ is the constant given by Lemma~\ref{lem:bsums_basic}.

Moreover,  there exists $\epsilon_0=\epsilon_0(f)$ such that  for any positive $\epsilon< \epsilon_0$, any $x$ and $T$ as above
%and such that $$\{x+j\alpha: \ :0\leq j\leq \epsilon^3 q_{n+1}\}\cap \left[-\frac{\varepsilon}{q_n},\frac{\varepsilon}{q_n}\right] = \emptyset,$$
and every pair of integers $0\leq r,s < T$ such that in addition { $|r-s|\leq\epsilon^3 B_{n,x} T$}, we have
\begin{equation}\label{eq:cancshort}
|f'^{(r)}(x)-f'^{(s)}(x)-(A_--A_+)(r-s){ \log q_n}|<\epsilon^{2} T.% C'(f) T.
\end{equation}
\end{lemma}

Both Lemma~\ref{lem:bsums} and Lemma~\ref{lem:canc} will be deduced {from} the following estimate of Birkhoff sums of $f'$ (which is written in a form which indeed allows {one} to deduce both previously stated {lemmas}).
% The statement o
%From the previous Lemma~\ref{lem:bsums_basic}, we now deduce two further
%The estimates we will use (given in the form of Lemma~\ref{lem:bsums} and Lemma \ref{lem:canc} below) will be derived from the following Lemma. % \marginpar{Ref to check} %, see Lemma (on which also Proposition in \cite{Ra} and Proposition 4.4 in~\cite{Ka-Ku-Ul} are based).\marginpar{Refs to add (C)}
\begin{lemma}[{Resonant control of Birkhoff sums of $f'$}] \label{lem:bsums_basic}
There exists a constant $C_f>0$ such that for every $\delta>0$ there exists $ J=J(\delta) \in \N$ such that, for all integers $r \geq q_{J}$, if $j_r \geq J $ denotes the unique integer such that
$q_{j_r}\leq r\leq q_{{j_r}+1}$,
we have
\begin{equation} \label{resonantcontrol}
 \left| f'^{(r)}(x)- (A_--A_+)r \log q_{j_r} \right|  \leq {C_f} \left( r+ \delta \ q_{j_r} \log q_{j_r}  + \text{Res}(x,r)\right),
 \end{equation}
{\mbox{for all $x\in\T$},}
where
\begin{equation*}\text{Res}(x,r):= \min \left\{ \, \frac{r}{q_{j_r}} \max_{0\leq i < r} \frac{1 }{|| x + i \alpha  ||},\,  \max_{0\leq i < r} \frac{1 }{|| x + i \alpha  ||} + 2 q_{j_r+1}\log \left(\frac{r}{q_{j_r}} \right)\,  \right\}.
\end{equation*}

\smallskip
{\Blue Furthermore, there exists a constant $\delta_\alpha>0$ (depending on $\alpha$ only) such that the estimate \eqref{resonantcontrol} %holds for every $r$, i.e.
%the estimate \eqref{resontantcontrol}
with $\delta=\delta_\alpha$ holds for every $r\in \mathbb{N}$.}
%\begin{equation}\label{allr_resonantcontrol}
% \left| f'^{(r)}(x)- (A_--A_+)r \log q_{j_r} \right|  \leq {C_f} \left( r+ %\delta_\alpha \ q_{j_r} \log q_{j_r}  + \text{Res}(x,r)\right),\qquad \forall  r\in \mathbb{N} .
% \end{equation}}
\end{lemma}
Let us first show how to deduce the proofs of Lemma~\ref{lem:bsums} and Lemma \ref{lem:canc} from Lemma~\ref{lem:bsums_basic}, we will then prove  Lemma~\ref{lem:bsums_basic}. We will use several times the following remark.
\begin{remark}\label{logreminder}\em
For any $r>0$, let $j_r$ be such that $q_{j_r}\leq r <q_{j_{r+1}} $ (so that the Ostrowski expansion of $r$ is a sum over $j_r$ terms). Then, since by the assumption ($\cD$3) (of Definition \ref{rotationassumptions}) on the rotation  $r\leq q_{j_r+1}\leq  q_{j_r}{\log^2 q_{j_r}} $ for any large $r$, we have that
$$
\lim_{r\to \infty} \frac{|\log q_{j_r}-\log r|}{\log r} = \lim_{r\to \infty} \frac{\log \frac{r}{q_{j_r}}}{\log r} \leq \lim_{r\to \infty}  \frac{\log (\log^2 q_{j_r})}{\log q_{j_r}} = 0.
$$
\end{remark}

\begin{proof}[Proof of Lemma~\ref{lem:bsums}]
Consider the Ostrowski expansion $r= \sum_{j=0}^{j_r} b_j q_j$, where $b_j\leq a_j$ and $b_{j_r}\geq 1$.
We claim that the first part of the lemma will be proved once we show that there exists  $n_0$ such that for $n\geq n_0$ and $r$ as in the assumptions, we have
 \begin{multline}\label{secondhalf}
  \left| f'^{(r)}(x)- (A_- - A_+)r \log q_{j_r} \right| \\ \leq  \frac{\epsilon^2}{2} ( r \log r).
 \end{multline}
Remark first that, since by assumption $r\geq \epsilon^4 q_n$, we can guarantee that $r$ is sufficiently large if $n$ is sufficiently large. Thus, by Remark~\ref{logreminder}, there is $n_0$ such that for $n\geq n_0$ and $r $ as in the assumptions
$$
  \left | (A_--A_+)r \log q_{j_r} - (A_--A_+)r \log r   \right|  \leq   \frac{\epsilon^2}{2} ( r \log r ) .
$$
Thus, once we prove \eqref{secondhalf}, combining it with the above equation we get the first part of the lemma {(namely \eqref{eq:dercont})}.
%\end{equation}
%Thus, we can choose $n_0$ sufficiently large so that, for $n\geq n_0$ and $r, j_r$ as in the Lemma, the above inequality holds\footnote{To see this, recall that we are assuming that $n\geq n_0$ and $r\geq \epsilon^4 q_n $. Thus, since  by  the assumption ($\cD$3) (of Definition \ref{rotationassumptions}) on the rotation $r\leq q_{j_r+1}\leq  q_{j_r}\log q_{j_r}^2 $ for any large $r$, if $n_0$  tends to infinity also $q_{j_r}$ has tend to infinity.}.

\smallskip
Let us now prove \eqref{secondhalf}.  By Lemma~\ref{lem:bsums_basic} for $\delta=\epsilon^2/6C_f$, the trivial bounds $q_{j_r}\leq r $ and the assumption $r\geq \epsilon^4 q_n$, we get the estimate
\begin{align}
 \left| f'^{(r)}(x)- (A_- - A_+)r \log q_{j_r} \right| & \leq {C_f} \left( r+ \frac{\epsilon^2}{6 C_f} r \log r   +  \text{Res}(x,r)  \right) \label{toestimate} \\
 & \leq \left( \frac{C_f}{\log \left( q_{n}\epsilon^4\right)}  +  \frac{\epsilon^2}{6} +\frac{C_f \text{Res}(x,r)}{ r\log r} \right) r \log r. \noindent
\end{align}
We now want to show that   each of the three terms in {parentheses} can be bounded by  {$\frac{\epsilon^2}2$}.
 Since for $n $ sufficiently large the first term in the parenthesis  is less than
  $\epsilon^2/6$,  we only have to show that $\text{Res}(x,r)/{r\log r}\leq \epsilon^2/6C_f$. To see this, remark first that, by the assumptions  $r \leq M q_{n+1}$ and $x\notin \Sigma_n(M) $, we have that $$\min_{0\leq i < r} {|| x + i \alpha  ||} \geq {1/(q_n\log^{7/8 }q_n)}.$$ Then, consider two cases:
  \begin{itemize}
  \item if $q_{j_r} \geq  q_n$, so that $q_n/q_{j_r}\leq 1 $,
 %(so $r/q_{j_r} = r/q_n$),
 we use the first estimate for $\text{Res}(x,r)$ given by Lemma~\ref{lem:bsums_basic} (together with the assumption $r\geq \epsilon^4 q_n$), which gives
$$
\frac{\text{Res}(x,r)}{r \log r} \leq \frac{\frac{r}{q_{j_r}} \left( q_n\log^{7/8 }q_n\right)}{r\log r} \leq \frac{\frac{q_n}{q_{j_r}}\log^{7/8 }q_n }{ \log (\epsilon^4 q_n) } \leq \frac{\log^{7/8 }q_n }{\log (\epsilon^4q_n)}  % \frac{1}{(\log q_n)^{1/8}\left( \frac{\log \epsilon^4}{\log q_n}+ 1\right)}
< \frac{\epsilon^2}{6C_f}
$$
if $n\geq n_0$ as long as $n_0$ is large enough.
\item Otherwise, if $q_{j_r} < q_n$ (and hence $q_{j_r+1} \leq q_n$), we use the other estimate for $\text{Res}(x,r)$ given by Lemma~\ref{lem:bsums_basic}. Using also that, by Remark~\ref{logreminder},
 $\log \frac{r}{q_{j_r} }\leq \epsilon^7  \log r$ if $r$ is large  and the assumption $r\geq \epsilon^4 q_n$, it gives
$$
\frac{\text{Res}(x,r)}{r \log r} \leq \frac{q_n \left(\log^{7/8 }q_n\right) + 2q_{j_r+1}\log (r/q_{j_r})}{r\log r} \leq  \frac{\log^{7/8 }q_n}{\epsilon^4 \log (\epsilon^4q_n)} + \frac{2q_{n}\epsilon^7 \log r}{\epsilon^4 q_n\log r},
$$
which can be made less than  $\frac{\epsilon^3}{6C_f}$ if $n\geq n_0$ with $n_0$  large enough and $\epsilon < \epsilon_0$, where $\epsilon_0:= \min \left\{ \frac{1}{24 C_f}, 1 \right\} $.
\end{itemize}
 This concludes the proof of \eqref{secondhalf} and hence, by the initial claim, of the first part  of the lemma.
%combininig it with  \eqref{ergodicterm2}, %that \eqref{toestimate} is estimated by $(\epsilon^2/2) r \log r $. also the proof of the first part

\smallskip
To prove the second part (namely \eqref{eq:smallder}),  {\Blue  let us apply the final part of Lemma~\ref{lem:bsums_basic} (e.g.~the estimate \eqref{resonantcontrol} with $\delta=\delta_\alpha $, which is valid for any $r$), which gives}
\begin{equation}\label{beginningsecondpart}
|f'^{(r)}(x)|\leq (A_- - A_+)r\log q_{j_r} + C_f(r + {\Blue \delta_\alpha} q_{j_r}\log q_{j_r} + Res(x,r)).
\end{equation}
To estimate $Res(x,r)$, we use the second estimate  given by Lemma~\ref{lem:bsums_basic}. Since $x \notin \Sigma(M)$ with {\Blue $M>1$}, using also assumption ($\cD$3) of Definition \ref{rotationassumptions} on $\alpha$ and $q_{j_r+1} \leq q_n$ (since $q_{j_r} \leq \epsilon^4 q_n <q_n $), we get
$$%\begin{equation}\label{resestimate}
\text{Res}(x,r) \leq \max_{0\leq \ell <r}\frac{1}{||x+\ell \alpha||}+ 2q_{j_{r}+1}\log\left({\Blue D_\alpha}\log^2 q_{j_{r}}\right) \leq {q_n}(\log {q_n})^{7/8}+ 4 q_n \log{\Blue( D_\alpha^{1/2}\log{q_n})}, $$
which is less than $ {\Blue \epsilon^4}q_n\log q_n$ if $n\geq n_0$ for $n_0$ is sufficiently large.
Thus, combining this with the initial estimate \eqref{beginningsecondpart} and using that $q_{j_r}\leq r \leq \epsilon^4q_n$, if $\epsilon_0<1$ is chosen so that $\epsilon_0^2 4 (A_- - A_+ +  {\Blue C_f (2+ \delta_\alpha)})<1$ (recalling also the assumption on $T$), gives that, for any $\epsilon<\epsilon_0$,
$$|f'^{(r)}(x)|\leq \epsilon^4q_n \log q_n \left((A_- - A_+) +{\Blue C_f (2+ \delta_\alpha)}\right) < \epsilon^2 q_n \log q_n < \epsilon^2 T. $$
 This concludes also the proof of the second part.
\end{proof}

Let us now prove Lemma \ref{lem:canc} using again  Lemma~\ref{lem:bsums_basic}.
\begin{proof}[Proof of Lemma~\ref{lem:canc}]
Recall first that $B_{n,x}<1$ (see Remark~\ref{B_order_constant}), hence the assumptions $r<B_{n,x}T/2$ and $T<q_{n+1}$ imply in particular that $r<T$ and that  $j_r\leq n$ (since $q_{j_r}\leq r < T< q_{n+1}$).
Applying {\Blue the last part of Lemma~\ref{lem:bsums_basic}, i.e. the estimate with $\delta=\delta_\alpha$ which holds for any $r$, } (and using the assumption $T> q_n\log q_n $) we get  %Remark that, for $r $ as in the assumptions, $j_r$ (defined by $q_{j_r}\leq r < q_{j_r+1}$) is equal to $n$. Thus,
\begin{align*}
|f'^{(r)}(x)- (A_--A_+)r  & {  \log q_{j_r}} |  < C_f  \left( r+ {\Blue \delta_{\alpha} } q_{j_r} \log q_{j_r}  + \text{Res}(x,r) \right) \\
& \leq C_f \left( T +  {\Blue \delta_{\alpha} }  q_{n} \log q_{n}  + \text{Res}(x,r) \right)\leq  C_f ( {\Blue T( 1+ \delta_{\alpha} )}  +\text{Res}(x,r)).
\end{align*}

{\Blue We will now show that $\text{Res}(x,r)\leq T$, so that the RHS is bounded by {\Blue $C_f (2+\delta_\alpha) T$}.  If $r\geq q_n$, since we then have $q_{j_r}= q_n$ (as we already observed above that  $j_r\leq n$), this estimate will conclude the proof. If, on the other hand, $r< q_n$, we also have $q_{j_r}<q_n$, so that the terms $(A_--A_+)r   {  \log q_{j_r}}$ and $(A_--A_+)r   {  \log q_{n}}$ are both less than $(A_--A_+)q_n  {  \log q_{n}}< (A_--A_+) T$ (by the assumptions on $T$). Thus also in this case we conclude the proof, by setting $C'(f):={\Blue C_f (2+\delta_\alpha) } + 2(A_--A_+)$.

\smallskip

We are hence now left to show that $\text{Res}(x,r)\leq T$.} To see this, let us first show that, {\Blue for any $r<B_{n,x}T/2$}, we have
\begin{equation}\label{delicatemin}
\max_{0\leq \ell < r }\frac{1}{||x+\ell \alpha ||} \leq \min \left\{ q_{n}\log^{7/8} q_n , \frac{2q_n}{B_{n,x}}\right\}.
\end{equation}
{\Blue The first bound (by $q_n \log^{7/8}q_n$) follows simply from the assumption $x \notin \Sigma_n (M) $ and $M>1$ (recall the definition \eqref{eq:sigma} of $\Sigma_n (M)$). }  Let us hence prove the second bound (namely by $2q_n/B_{n,x}$).   Let us denote by $m_{x,r}=m(x,r)$ the index $0\leq m_{x,r}<r$ such that
 $$
 ||x+m_{x,r}  \alpha ||  =\min \{ ||x+\ell \alpha ||, 0\leq \ell <r\}.
 $$
 %$m_{x,r}$ be the minimum of $\{ ||x+\ell \alpha ||, 0\leq \ell <r\}$ and let $k$ be such that $||x+k  \alpha || = m_{x,r}$.
  Since $ m_{x,r}<r$, dividing $m_{x,r}$ by $q_n$ we get that
 % $r/q_n=r/q_{j_r}<M$ (by the local assumption $q_{j_r}=q_n$, the Lemma assumptions $r<B_{x,n} T$ and $T<q_{n+1}$ and the observation that $B_{x,n}<1$, see Remark~\ref{B_order_constant})
   there exists $0 \leq j \leq r/q_n$ such that $0\leq m_{x,r}-j q_n \leq q_n $.
 Since the points $(R_{q_n \alpha})^i x$ and $(R_{q_n \alpha})^{i+1} x$ for $0\leq i < r-1$ are at most $1/q_{n+1}$ apart   (recall \eqref{CFbasicestimate}), using also the assumptions $r<B_{n,x}T/2$ and $T<q_{n+1}$, it follows that
$$
\left\| \left( x+ (m_{x,r}-j q_n) \alpha \right)-   \left( x+m_{x,r}  \alpha \right) \right\| \leq   \frac{r}{q_n} \frac{1}{q_{n+1}} \leq   \frac{B_{n,x}T}{2q_n} \frac{1}{q_{n+1}}  \leq   \frac{B_{n,x}}{2q_n}.
$$
Moreover, by the definition of $B_{n,x}$ (recall \eqref{eq.distsing}), we have that $B_{n,x}/q_n \leq   ||x+(m_{x,r}-j q_n) \alpha || $ and  by Lemma \ref{spacing_orbit} (since $r>q_n$) that  $||x+ m_{x,r}  \alpha ||<1/q_n $. It  follows that  $x':=x+(m_{x,r}-jq_n) \alpha$ and $x'':= x+m_{x,r} \alpha$ either both belong to the interval $\left[\frac{B_{n,x}}{2q_n}, \frac{3}{2q_n}\right]$ or both to the interval  $\left[ 1-\frac{3}{2q_n}, 1- \frac{B_{n,x}}{2q_n},\right]$, so that (for $n$ large) we either  have that $||x' ||=x'$ and $||x'' ||=x''$, or we have that $||x' ||=1- x'$ and $||x'' ||=1-x''$. Thus, from the previous equation (and the definition of $B_{n,x}$), we also deduce that
%$||x+ i q_n \alpha || = |x+i q_n|$ for both $i=m_{x,r}$ and  $i=m_{x,r}-j q_n$, or otherwise $||x+ i q_n \alpha || = 1-|x+i q_n|$ for both and hence that
%Remark that since $r<q_{n+1}$  %and at most one of them belongs to the interval $[-1/4q_{n+1}, 1/4q_{n+1}]$. Since
$$
\frac{B_{n,x }}{q_n} \leq ||x+ (m_{x,r}-j q_n) \alpha ||  \leq   ||x+m_{x,r}  \alpha ||  +  \frac{B_{n,x}}{2q_n}.
$$
Recalling the definition of $m_{x,r}$, this gives
 that $\min \{ ||x+\ell \alpha ||, 0\leq \ell <r\} \geq  \frac{B_{n,x}}{2q_n}$, which
concludes the proof of~\eqref{delicatemin}.
\smallskip

Let us go back to show the estimate $\text{Res}(x,r)< T$. We consider separately two cases:
\begin{itemize}
\item
Assume first that $q_{j_r}< q_n$. Remark that this means that $q_{j_r+1}\leq  q_n$ and hence $r\leq q_n $. Thus, by this remark and since by Remark \ref{logreminder} for any $0<\delta<1$  we have $\log (r/q_{j_r})\leq  \delta \log r  /4$ for $r$ large (and recalling that $T> q_n \log q_n$), we get
$$
\text{Res}(x,r)\leq  q_n\log q_n^{7/8} + 2q_{j_r+1}\log \left(\frac{r}{q_{j_r}}\right) \leq  q_n \log q_n \left( \frac{1}{\log^{1/8} q_n} +  \frac{\delta}{2} \right) < \delta\, T
$$
for $n$ sufficiently large.
\item Otherwise, if $q_{j_r}\geq q_n$, we have $q_{j_r}=q_n$ (since by the assumptions $j_r\leq n$).  In this case  using \eqref{delicatemin} and the assumption $r<B_{n,x}T /2$, we can then estimate
$$
\text{Res}(x,r)\leq \frac{r}{q_n} \max_{0\leq \ell < r }\frac{1}{||x+\ell \alpha ||}  \leq \frac{B_{n,x}T}{2q_n} \frac{2q_n}{B_{n,x}} =  T.
$$
%\smallskip
%If $r\geq q_n$, we have $q_{j_r}=q_n$ (since by the assumptions $j_r\leq n$) and this concludes the proof of \ref{elicatemin}
\end{itemize}
This concludes also the first part of the lemma. %{lemma}.

\smallskip
For the second part of the lemma, assuming without loss of generality that $r> s$ and using the cocycle property of Birkoff sums, we have that $|f'^{(r)}(x)-f'^{(s)}(x)|  = |f'^{(r-s)}(z)|$ for $z:= R_{\alpha}^s (x)$.
{\Blue We claim that, by Lemma~\ref{lem:bsums_basic}, {\Blue for any $\epsilon>0$ there exists $n_0>0$ such that for all $n\geq n_0$,} we have that, for any $r,s$ as in the assumptions, }
\begin{multline}\label{secondpartestimateC}
|f'^{(r)}(x)-f'^{(s)}(x)  -(A_--A_+)(r-s){ \log q_{j_{r-s}}}|  < \\ C_f (r-s) + \frac{\epsilon^2}{6} {\Blue q_n \log q_n}
% q_{j_{r-s}} \log q_{j_{r-s}}
  + C_f \text{Res}(z,r-s) .
%\\ & C_f \left( (r-s) +  q_{n} \log q_{n} \left( \frac{\epsilon^3}{3}  + \text{Res}{z,r-s} \right) \right) .%\leq C_f \epsilon^3 T.
\end{multline}
{\Blue To see this, pick $N=N(\epsilon, f)$ such that,
\begin{equation}\label{choiceNthreshold}  q_{n-N}\leq \frac{1}{\sqrt{2}^N} q_n < \left( \min \left\{ \frac{\epsilon^2}{6 \delta_{\alpha}}, \frac{\epsilon^2}{ 4(A_--A_+)}\right\}\right) q_n,
\end{equation}
where $\delta_\alpha$ is the constant  in the last part of Lemma~\ref{lem:bsums_basic} (the second bound for $N$ will be used only later in the proof). }
Applying {\Blue the first part of} Lemma~\ref{lem:bsums_basic} for $\delta=\epsilon^2/ (6 C_f)$, {\Blue (since $n-N$ goes to infinity as $n$ does)} there exists $n_0$  such that for all $n\geq n_0$ {\Blue  \eqref{secondpartestimateC} holds as long as $r-s \geq q_{n-N}$ (remark that $q_{j_{r-s}}\leq q_n$  since $r-s<T<q_{n+1}$). }
%(take $n_0:= J(\epsilon)+N$ where $J(\epsilon)$ is given by the lemma, so that $q_{j_{r-s}}\geq q_{n-N})$ in this case implies ${jpp_{r-s}} \geq q_J$ and the lemma estimate can be applied).}
{\Blue On the other hand, if $r-s\leq q_{n-N}$, we can use the final part of Lemma~\ref{lem:bsums_basic}  to get that
$$
|f'^{(r)}(x)-f'^{(s)}(x)  -(A_--A_+)(r-s){ \log q_{j_{r-s}}}| < C_f (r-s) +  \delta_{\alpha} q_{j_{r-s}} \log q_{j_{r-s}}  + C_f \text{Res}(z,r-s) $$
which also in this case gives \eqref{choiceNthreshold} since by assumption on $r,s$ and choice \eqref{choiceNthreshold} of $N$ we have that
$\delta_{\alpha} q_{j_{r-s}} \log q_{j_{r-s}}\leq \delta_{\alpha} q_{n-N} \log q_{n} \leq \frac{\epsilon^2}{6} q_{n}\log q_n$.
This concludes the proof of \eqref{secondpartestimateC} and the claim.}

{\Blue We now want to show that {\Blue for some $n_1\geq n_0$ and $\epsilon_0>0$ such that} each of the three terms in the RHS  of \eqref{secondpartestimateC}} is less than $(\epsilon^2/6)T $ {\Blue if $\epsilon<\epsilon_0$ and $n\geq n_1$}.  Since by assumption ${r-s}\leq \epsilon^3 T$ (recall that $B_{n,x}<1$, by Remark~\ref{B_order_constant}) to control the first term, it is enough to choose $\epsilon< \epsilon_0$, where $\epsilon_0:=1/(6C_f)$. For the second  is is enough {\Blue to recall that by} assumption $q_{n}\log q_n<T$. Finally, to estimate $\text{Res}(z,r-s)$ we use two regimes:
\begin{itemize}
\item either $q_{j_{r-s}}\geq q_n$ and hence $q_{j_{r-s}}=q_n$ (since as  we already remarked $r-s<q_{n+1}$), in which case we use the first estimate for $\text{Res}(z,r-s)$  (in Lemma~\ref{lem:bsums_basic}) combined with \eqref{delicatemin} and the assumption on $(r-s)$ to get
\begin{align*}
\text{Res}(z,r-s) \leq \frac{r-s}{q_{j_{r-s}}}\max_{0\leq \ell <r-s}\frac{1}{||z+\ell \alpha||} & \leq \frac{r-s}{q_n}  \max_{0\leq \ell <r}\frac{1}{||x+\ell \alpha||} \\
& \leq \frac{\epsilon^3 B_{n,x} T}{q_n}\frac{2 q_n}{B_{n,x}} = 2 \epsilon^3 T,
\end{align*}
which is less than $(\epsilon^2T)/ (6C_f)$ for $\epsilon<\epsilon_0$, given that  $\epsilon_0 <  1/{(12 C_f)}$.
\item  Otherwise, if  $q_{j_{r-s}}< q_n$ (and hence $q_{j_{r-s}+1}\leq q_n$), we use the other estimate (and assumption ($\cD$3) of Definition \ref{rotationassumptions} on $\alpha$) to get
\begin{align*}
{Res}(z,r-s)& \leq \max_{0\leq \ell <r-s}\frac{1}{||z+\ell \alpha||}+ 2q_{j_{r-s}+1}\log\left({\Blue D_\alpha \log^2 q_{j_{r-s}}}\right) \\ & \leq {q_n}(\log {q_n})^{7/8}+ 4 q_n \log {\Blue (D_\alpha^{1/2} \log{q_n})},
\end{align*} %\leq \epsilon^2 q_n\log q_n = \frac{\epsilon^2}{6} T$$
which is less than $ \epsilon^3 q_n\log q_n$ if $n\geq n_0$ for $n_0$ is sufficiently large and hence {less than} $ \frac{\epsilon^2}{6C_f} T$ if $\epsilon<\epsilon_0$.
\end{itemize}
Thus, we showed that %\eqref{secondpartestimate} is bounded by $\epsilon^2 T/2 $.
%\begin{equation}\label{secondpartestimate}
$$
|f'^{(r)}(x)-f'^{(s)}(x)  -(A_--A_+)(r-s){ \log q_{j_{r-s}}}| \leq \epsilon^2 \frac{T}{2}.
$$%\end{equation}

{In order} to conclude, it is now enough to show that
\begin{equation}\label{comparisonlogs}
\left|(A_--A_+)(r-s){ \log q_{j_{r-s}}}-(A_--A_+)(r-s){ \log q_{n}} \right| \leq \frac{\epsilon^2}{2} T.
\end{equation}
If $r-s\geq q_n$, we have that  $q_{j_{r-s}} = q_n$ (since we also know that $r-s< T<q_{n+1}$) and hence there is nothing to prove. Thus we can assume that $r-s< q_n$.
{\Blue Recall the choice of $N$ made earlier (see \eqref{choiceNthreshold})}
%We now pick $N$ such that $q_{n-N}\leq 1/\sqrt{2}^N q_n <\epsilon^2q_n /4(A_--A_+)$
and consider again two cases.
\begin{itemize}
\item
%or, equivalently that  $\log (q_n/q_{j_{r-s}}) \leq \frac{\epsilon^2 q_n \log q_n /2(A_--A_+) $ (which suffices since $T>q_n\log q_n$ by assumption).  We consider two cases.
 If $ q_{n-N}\leq r-s < q_n$, by definition of $j_{r-s}$ we have $q_{j_{r-s}}\geq q_{n-N}$.
Since we can write $q_{n} = (\prod_{j = n-N}^{n-1} \frac{q_{j+1}}{q_{j}})q_{n-N} $, using also the assumption ($\cD$3) on $\alpha$ (cf.~Def.~\ref{rotationassumptions}), we get in this case  that
$$
%\begin{equation}\label{denominators_ratio1}
\log \frac{q_n}{q_{j_{r-s}} } \leq \log \frac{q_n}{q_{n-N}}  \leq \sum_{j = n-N}^{n-1} \log {\Blue (D_\alpha \log^2 q_{j})} \leq 2N \log { \left(D_\alpha^{1/2} \log  q_{n} \right)}.
$$%\end{equation}
Thus, since this bound divided by $\log q_n$ goes to zero as $n$ grows to infinity, {\Blue increasing $n_1$ if necessary, for $n\geq n_1$} (recalling that by assumption $q_n\log q_n<T$), we have
% Then one can show it by arguing as in \eqref{denominators_ratio}, that
 $$(r-s)\log \left(\frac{q_n}{q_{j_{r-s}}}\right)  \leq q_n \left( \frac{\epsilon^2   }{2(A_--A_+)}\log q_n\right)\leq \frac{\epsilon^2 T  }{2(A_--A_+)}. $$
 %  \leq  \frac{\epsilon^2 }{2(A_--A_+)}T$$
%which concludes  the proof since $q_n\log q_n<T$.
\item Otherwise, if $0\leq r-s \leq q_{n-N}$, we simply have
\begin{align*}
\left| (r-s)\left( \log q_{j_{r-s}}-{ \log q_{n}}\right) \right| & \leq 2 q_{n-N}\log q_n  \\ & \leq 2 \frac{\epsilon^2}{4(A_--A_+)}q_n \log q_n  <  \frac{\epsilon^2 T}{2(A_--A_+)}
\end{align*}
by {the} assumption on $T$ and {the} choice of $N$.
\end{itemize}
In all cases, this concludes the proof of \eqref{comparisonlogs} and hence of the second part of the lemma.
\end{proof}

The rest of this section will be devoted to the proof of Lemma~\ref{lem:bsums_basic}.
\begin{proof}[Proof of Lemma~\ref{lem:bsums_basic}]
  As in the proof of Lemma~\ref{lem:dk2}, consider  the Ostrowski expansion of $r$, which, by {the} definition of $j_r$, is given by
 $r= \sum_{j=1}^{j_r} b_j q_j$, where $b_j\leq a_j\leq \frac{q_{j+1}}{q_j}$ and $b_{j_r}\geq 1$.  Correspondingly, we get the following decomposition of Birkhoff sums of $f'$:
 %(where by convention the sum is zero if it runs from $0$ to $-1$, namely when $j=i$):
$$
f'^{(r)}(x)= \sum_{j=1}^{j_r}\left( \sum_{s=0}^{b_j-1}{f'^{(q_j)}}(z_{j,s})\right),\   \text{where} \, \begin{cases} z_{j_r,s}=x +s q_{j_r}\alpha, & \\ z_{j,s}= x+ \sum_{k=j+1}^{j_r}b_k q_k \alpha +s q_j\alpha & \ \text{for} \ 1\leq j < j_r.\end{cases}$$
 Then, by applying the estimates of Lemma~\ref{lem:DK} (in particular \eqref{eq:DKf'} for the point $z_{j,s}$), we have
\begin{equation}\label{decompsum}
\left| f'^{(r)}(x) -   \sum_{j=1}^{j_r}  (A^-- A^+) b_j q_j \log q_j \right|  \leq      C(f) \sum_{j=1}^{j_r}  \sum_{s=0}^{b_j-1} q_j \left( 1+ \frac{1}{B_{j, z_{j,s}}}\right).
\end{equation}
Let us first show that for any $\delta >0$, there exists $J_0= J_0(\delta)$ such that for  $r\geq q_{J_0}$, {we have}
\begin{equation}\label{ergodicterm}
 \left| \sum_{j=1}^{j_r} (A^- - A^+) b_j q_j \log q_j -  (A^- - A^+) r \log q_{j_r} \right| \leq \frac{\delta}{2} (A^+-A^-) q_{j_r} \log q_{j_r}  .
\end{equation}
{\Blue and also that there exists a constant $\delta_\alpha>0$ depending on $\alpha $ such that the estimate above holds for every $r\in\mathbb{N}$ if we take $\delta=2\delta_\alpha$ (which is needed to prove the final part of the lemma).}
{\Blue To prove the latter, it is enough to use the Ostrowski expansion of $r$ to write
$$\left| \sum_{j=1}^{j_r}  b_j q_j \log q_j -   r \log q_{j_r} \right| = \sum_{j=1}^{j_r-1} b_j q_j \log \frac{q_{j_r}}{q_j} \leq \left( \sum_{j=1}^{j_r-1} q_j \right) \log q_{j_r}$$
(remark that $\log(q_{j_r}/q_j)=0$ for $j=j_r$ so the last sum does not include the term $j_r$) and then using that $\sum_{j=1}^{j_r-1} < q_{j_r}$ by definition of Ostrowski expansion. }
%%(A^- - A^+) r \log r (1-\delta ) \leq (A^- - A^+) \left ( m q_{n+1} \log q_{n+1}  + \sum_{j=1}^n b_j q_j \log q_j \right) \leq (A^- - A^+) r %\log r .
% (A^- - A^+) r \log q_{j_r} - \nu q_{j_r} \log q_{j_r}   \leq   \left ( \sum_{j=1}^{j_r} (A^- - A^+) b_j q_j \log q_j \right)  \leq   (A^- - %A^+) r \log q_{j_r}  .
%\end{equation}
%The upper bound $\left ( \sum_{j=1}^{j_r} (A^- - A^+) b_j q_j \log q_j \right)  \leq   (A^- - A^+) r \log q_{j_r} $ follows immediately from the Ostrowski decomposition, since $\log q_j\leq \log q_{j_r}$ for each $0\leq j\leq j_r$. To prove the lower bound,

{\Blue To now prove \eqref{ergodicterm},} choose $N=N(\delta)$ such that  $q_{j_r-N}/q_{j_r} \leq 1/ \sqrt{2}^{N} \leq \delta/8$. Writing $q_{j_r} = (\prod_{j = j_r-N}^{j_r-1} \frac{q_{j+1}}{q_{j}})q_{j_r-N} $ and using the assumption ($\cD$3) on $\alpha$ (cf.~Def.~\ref{rotationassumptions}), for any $j_r-N\leq j < j_r$, we have
%\begin{equation}\label{denominators_ratio}
$$
\log \frac{q_{j_r}}{q_{j}} \leq \log \frac{q_{j_r}}{q_{j_r-N}} \leq \sum_{j = j_r-N}^{j_r-1} \log (\log q_{j})^2 \leq 2N \log \log q_{j_r}.
$$%\end{equation}
%Thus, for any $j_r-N\leq j < j_r$,  we have that %using also the assumption  $r\geq\epsilon^4 q_n $, we have
%$$
%\frac{\log q_j}{\log q_{j_r}}  \geq \left( 1- \frac{\log \frac{ q_{j_r}}{q_{j_r-N}}}{\log  q_{j_r} }\right) \geq   \left( 1- \frac{2N \log %\log q_{j_r}}{\log q_{j_r}} \right) .
%$$
%Thus, choosing $\overline{j}=\overline{j}(\delta)$ sufficiently large, for $r\geq q_\overline{j}$,  $j_r\geq \overline{j}$ is also large and one can then guarantee that $\log q_j \geq (1-\delta/4) \log r$ for $j\geq j_r-N$.  Thus,
% splitting the sum over $j$ in \eqref{decompsum}  into $1\leq j<  j_r-N $ and $j_r-N\leq j \leq j_r$ and
Thus, using that, by definition of Ostrowski expansion, $r= \sum_{j=1}^{j_r}b_j q_j$ and $\sum_{j<  j_r-N}b_j q_j\leq q_{j_r-N} $, %and hence $\sum_{j=j_r-N}^{j_r} b_j q_j\geq r- q_{r_j-N}$,
 we get that
\begin{align*}
%\left { \sum_{j=1}^{j_r} b_j q_j \log q_j} - r \log q_{j_r}\right|}  \geq &  \frac{\sum_{j=j_r-N}^{j_r} b_j q_j \log q_j q_{j_r-N} \log q_{j_r-N}}{q_{j_r}\log q_{j_r}}
 \left| r \log q_{j_r} -   \sum_{j=1}^{j_r} b_j q_j \log q_j   \right| & \leq \left | \sum_{j=j_r-N}^{j_r-1}  b_j q_j \log \frac{ q_{j_r}}{ q_j }  \right| +  2 \left( \sum_{j<  j_r-N}b_j q_j \right) \log q_{j_r}  \\ & \leq  % b_{j_r }q_{j_r} +
  q_{j_r} (2N \log \log q_{j_r} )+ 2 q_{j_r-N} \log q_{j_r} \\ & \leq %r +
   \left( \frac{2N \log \log q_{j_r} }{\log q_{j_r}} + \frac{\delta}{4}\right) q_{j_r} \log q_{j_r} ,
\end{align*}
%\begin{align*}
%%\left { \sum_{j=1}^{j_r} b_j q_j \log q_j} - r \log q_{j_r}\right|}  \geq &  \frac{\sum_{j=j_r-N}^{j_r} b_j q_j \log q_j q_{j_r-N} \log %q_{j_r-N}}{q_{j_r}\log q_{j_r}}
%\sum_{j=1}^{j_r} b_j q_j \log q_j & \geq \sum_{j=j_r-N}^{j_r} \left( b_j q_j \log q_j \pm  b_j q_j \log q_{j_r}\right)   \geq
%(r - q_{j_r-N}) \log q_{j_r} -  \sum_{j=j_r-N}^{j_r} b_j q_j \log\frac{{q_{j_r}}}{q_j}\\ & \geq r \log q_{j_r} - q_{j_r-N} \log q_{j_r} -  2N %r \log \log q_{j_r}.
%\end{align*}
where the last inequality follows from the previous choice of $N$. {\Blue Thus, for some $J\geq J_0$}, we can hence guarantee that{, \Blue for $r\geq q_J$, } the RHS is less than $ \delta/2(q_{j_r} \log q_{j_r})$ and this concludes the proof of \eqref{ergodicterm}.

\smallskip
Let us now estimate the RHS of \eqref{decompsum}.
Recall  that $q_j /B_{j,z} =1/ || z+ i( j,z) \alpha||$ (by the definition of $B_{j,z}$, cf.~\eqref{eq.distsing}), so that, recalling also the definition of Ostrowski decomposition, we have
\begin{equation}\label{resonantcontribution}
 {\sum_{j=1 }^{j_r} \sum_{s=0}^{b_j-1}q_j \left( 1+ \frac{1}{B_{j, z_{j,s}}}\right)}\leq r %{\sum_{j=1 }^{j_r} b_j q_j }
+ {\sum_{j=1 }^{j_r} \sum_{s=0}^{b_j-1}\frac{1} {|| x+ i(j,z_{j,s})\alpha||} }.
 \end{equation}
  For each $1\leq j\leq  j_r$, the points of points $ x+ i(j, z_{j,s})\alpha$ for $s=0,1, \dots, b_j-1$ are distinct (since they belong to disjoint orbits of length $q_j$) and   they are contained in {an} orbit of $R_\alpha$ of length $q_{j+1}$ (namely the orbit of the point $z_{j+1,b_{j+1}}:= z_{j,0}= x+ \sum_{k=j+1}^{j_r} b_k q_k \alpha $ if $1\leq j<  j_r$, or of the point $x$ for $j=j_r$).  Thus,
  by the orbit spacing given by Lemma \ref{spacing_orbit} (and recalling the definition of $|| \cdot ||$, see {Section}~\ref{sec:DK}), for $1\leq j \leq j_r $ we can estimate
\begin{align}\label{arithmeticestimate}
 \sum_{s=0}^{b_j-1}\frac{1}{ || x+ i(j,z_{j,s})\alpha||}& \leq \sum_{s=0}^{b_j-1}\frac{1}{ \{ x+ i(j,z_{j,s})\alpha\} } + \sum_{s=0}^{b_j-1}\frac{1}{ 1-\{ x+ i(j,z_{j,s})\alpha\}} \\ & \leq 2 \sum_{s=0}^{b_j-1} \frac{1}{ \displaystyle\min_{k=0, \dots, b_{j}-1}|| x+ i(j, z_{j,k}) \alpha ||+ s \, \frac{1}{2q_{j+1}}},\nonumber
%  || x+ i(j+1, z_{j+1,s})\alpha ||+  \frac{s}{2q_{j+1}}, \qquad s=0, 1, \dots,b_j-1.
\end{align}
where moreover, for $1\leq j < j_r $,
$$
\min_{k=0,\dots, b_j-1}  || x+ i(j,z_{j,k})\alpha|| \geq  || x+ i(j+1,z_{j+1,b_{j+1}} )\alpha|| .
$$
%  their distances  $|| x+ i( z_{j,s},j)\alpha||$ with the following arithmetic progression:
%$$
% || x+ i(j,z_{j,s})\alpha|| \geq || x+ i(j+1, z_{j+1,s})\alpha ||+  \frac{s}{2q_{j+1}}, \qquad s=0, 1, \dots,b_j-1.
%$$
One has the following estimate on the sum of inverses of arithmetic progression of step $h >0$  and length $K \in \mathbb{N}$, starting at $x_0>0$ (see for example Lemma 5.1 in \cite{Ko} or Lemma 9 in \cite{Ul}):
\begin{equation}\label{arithmeticprogression}
\sum_{s=1}^{K} \frac{1}{x_{0} + s h} \leq \frac{1}{x_0} + \frac{\log K}{h}.
\end{equation}
Thus, applying this estimate for each $1\leq j < j_r$ and including the contribution of the minimum of each arithmetic progression with $1\leq j< j_r$ (namely $1/|| x+ i(j+1, z_{j+1,s}) \alpha ||$)  with  the arithmetic progression level $j+1$ (notice that it is a distinct term since it belongs to a different block of the orbit), we get
\begin{equation}\label{otherj}
\sum_{j=1}^{j_r-1} \sum_{s=0}^{b_j-1}\frac{1}{|| x+ i(j,z_{j,s})||} \leq   \sum_{j=1}^{j_r-1} 2q_{j+1} {\log( b_j + 1)}.
\end{equation}
Using assumption ($\cD$3) (of Definition \ref{rotationassumptions}) on the rotation number, {\Blue $b_j\leq  q_{j+1}/q_j \leq D_\alpha \log^2 q_j $.}
{\Blue Thus (since $\sum_{j\leq j_r} q_j/q_{j_r}$ is bounded {uniformly} in $r$ because $q_j$ grow exponentially) there exists a constant $C_\alpha>0$ such that for any $r$ we have} %
$$
\frac{\sum_{j=1}^{j_r-1}  2q_{j+1} {\log( b_j + 1)}  }{q_{j_r}\log q_{j_r}} \leq \frac{\log ( {\Blue D_\alpha }\log^2 q_{j_r}+1)}{ \log q_{j_r}}   \sum_{j=1}^{{j_r-1}} \frac{2 {q_{j+1}}}{q_{j_r}}  \leq C_\alpha.% \leq %\frac{\delta}{2} .
$$
{\Blue Moreover,   the above expression can be made less than $\delta/2$ if $j_r\geq J_1$, up to enlarging $J_1$ if necessary (since the term before the sum goes to zero as $j_r$ grows while the sum over $j$, as already remarked, is uniformely bounded in $r$).}

Combining this with \eqref{resonantcontribution} and \eqref{otherj}, we have thus shown that for $r\geq q_{J_1} $, we have %there exists $J=\min\{ J_0, J_1\}$ such that for $r\geq q_{J}$,
\begin{equation}\label{allbutlastjr}
{\sum_{j=1 }^{j_r} \sum_{s=0}^{b_j-1}q_j \left( 1+ \frac{1}{B_{j, z_{j,s}}}\right)} \leq r +  \frac{\delta}{2}( q_{j_r}\log q_{j_r}) + \text{Res}(x,r),
\end{equation}
where the  term $\text{Res}(x,r)$ (where $\text{Res}$ stays for \emph{resonant}) is simply the sum relative to  $j=j_r$, namely, recalling the definition \eqref{eq.distsing} of $B$ and of $i(\cdot,\cdot)$,
$$
\text{Res}(x,r):= {\sum_{s=0}^{b_{j_r}-1}  \frac{q_{j_r}}{B_{j_r, z_{j_r,s}}}} = \sum_{s=0}^{b_{j_r}-1}  \frac{1}{ || z_{j_r,s}+i(j_r, z_{j_r,s}) \alpha ||} = \sum_{s=0}^{b_{j_r}-1} \max_{ 0\leq \ell < q_{j_r} } \frac{1}{ || z_{j_r,s}+\ell \alpha ||}.
$$
{\Blue Moreover, we have also shown that the above estimate \eqref{allbutlastjr} for $\delta=2 C_\alpha$ holds for all $r$.}
Thus, the estimate \eqref{allbutlastjr}  combined with \eqref{decompsum}, \eqref{ergodicterm} and concludes the proof {\Blue of both parts of the lemma} when setting {\Blue $\delta_\alpha:= C_\alpha+1$ and} ${\Blue C_f}:= {\Blue \max \{ C(f),(A_--A_+)\}}$ as long as we prove the  estimate on $\text{Res}(x,r)$ stated in the {\Blue lemma}. This follows from two separate estimates (of which one then takes the minimum). On one hand, we can estimate $\text{Res}(x,r)$ by the number of terms times the largest term, i.e.~(since the points $z_{j_r,s} +i(j_r, z_{j_r,s}) \alpha$ for $s=0,\dots, b_{j_r}$ are just points in the orbit of $x$ of length $r$)
$$
\text{Res}(x,r) \leq b_{j_r} \max_{s=0,\dots, b_{j_r}-1} \frac{1}{ || z_{j_r,s} +i(j_r, z_{j_r,s}) \alpha ||}\leq \frac{r}{q_{j_r}} \max_{0\leq \ell < r} \frac{1}{ || x +\ell \alpha ||}.
$$
Alternatively, we can estimate $\text{Res}(x,r) $ by using \eqref{arithmeticestimate} for $j=j_r$. Using the bound \eqref{arithmeticprogression} for an arithmetic progression of step $1/q_{j_r+1}$ and length $b_{j_r}\leq r/q_{j_r}$ and bounding trivially the minimum term of the progression (as above), we get
$$
 \text{Res}(x,r) \leq \max_{0\leq \ell < r} \frac{1}{ || x +\ell \alpha ||} + 2 q_{j_r+1} \log \left( \frac{r}{q_{j_r}}\right).
$$
%B_{n,x}:= { q_n \ \|x+\inx \alpha\| = }\min_{0\leq j <q_n} q_n\ \|R_\alpha^{j}x\| .
 %+ \frac{(r/q_{j_r})}{ \min_{1\leq i < r}|| x+ i\alpha ||}+
Taking the minimum of these two estimates concludes the estimate of $\text{Res}(x,r)$ and hence the proof of the lemma.
\end{proof}

\subsection{Birkhoff sums of $f''$, $f'''$ and $f''''$.}

\begin{lemma}[{Control of Birkhoff sums of $f''$}]\label{cor:bux} There exists a constant $C'(f)>0$ such that for every $\epsilon>0$ we can find $n_0\in \N$ such that for  every $n\geq n_0$ and $x\in\T$ for which $B_{n,x}\geq \epsilon^{1/5}$ and for all $k\in[1,\frac{B^5_{n,x}q_{n+1}}{6q_n}]\cap \Z$, we have
\begin{equation}\label{cor:co2}
|f''^{(kq_n)}(x)-kf''(x+i\alpha)|\leq C'(f)kq_n^2,
\end{equation}
\begin{equation}\label{cor:co3}
|f'''^{(kq_u)}(x)-kf'''(x+i\alpha)|\leq C'(f)kq_n^3,
\end{equation}
\begin{equation}\label{cor:co4}
|f''''^{(kq_u)}(x)-kf''''(x+i\alpha)|\leq C'(f)kq_n^4,
\end{equation}
where $i=i(n,x)$.
\end{lemma}
\begin{proof} For $s=0,1,\ldots,k-1$, set $z_s=x+sq_n\alpha$. In view of~\eqref{eq:DKf''} (applied to $z_s$), we obtain
$$
f''^{(kq_n)}(x)=\sum_{s=0}^{k-1}f''^{(q_n)}(z_s)=\sum_{s=0}^{k-1}
\left(f''(z_s+i(z_s,n)\alpha)+{\rm O}(q_n^2)\right),$$
whence
$$
\left|f''^{(kq_n)}(x)-kf''(x+i\alpha)\right|=
\left|\sum_{s=0}^{k-1}\left(
f''(z_s+i(z_s,n)\alpha)-f(x+i\alpha)\right)\right|+{\rm O}(kq_n^2)\leq$$$$
\sum_{s=0}^{k-1}\left|
f''(z_s+i(z_s,n)\alpha)-f(x+i\alpha)\right|+{\rm O}(kq_n^2).$$
Note that if neither $x+i\alpha$ nor $z_s+i(n,z_s)\alpha$ belongs to $\left[\frac{-1}{10q_n},\frac1{10q_n}\right]$ then $f''$ at each of these points is of order ${\rm O}(q_n^2)$, so the corresponding summand in the latter sum is of the same order. Suppose now that either $x+i\alpha$ or $z_s+i(n,z_s)$ belongs to $\left[\frac{-1}{10q_n},\frac1{10q_n}\right]$. We claim that in this case $i(n,x)=i(n,z_s)$. Indeed,
since $k\leq B^5_{n,x}q_{n+1}/(6q_n)$, we have
$$
\|x-z_s\|=\|sq_n\alpha\|\leq k\|q_n\alpha\|\leq \frac{B_{n,x}^5q_{n+1}}{6q_n}\|q_n\alpha\|<\frac{B_{n,x}^5}{6q_n}
<\frac1{6q_n}.$$
Now, if $x+i\alpha\in \left[\frac{-1}{10q_n},\frac1{10q_n}\right]$ then $z_s+i\alpha\in \left[\frac{-1}{4q_n},\frac1{4q_n}\right]$ which means that $i=i(n,z_s)$ and reversing the role of the two points (if necessary), the claim follows.

So we have $i(n,x)=i=i(n,z_s)$ and our next claim is that the points $x+i\alpha$ and $z_s+i\alpha$ are on the same side of the singularity (at~0). Indeed, as $\|x-z_s\|\leq \frac{B^5_{n,x}}{6q_n}=\|x+i\alpha\|/6$, we only need to note that
$$0\notin \left(x+i\alpha-\frac{\|x+i\alpha\|}{6}, x+i\alpha+\frac{\|x+i\alpha\|}{6}\right)$$
which is obvious.
It follows that we can use the mean value theorem to obtain
$$
|f''(x+sq_n\alpha+i\alpha)-f''(x+i\alpha)|=f'''(\theta)\|sq_n\alpha\|
\leq f'''(\theta)\frac{B^5_{n,x}}{6q_n},$$
where
$$
\theta\in \left(x+i\alpha-\frac{B_{n,x}^5}{6q_n},
x+i\alpha+\frac{B_{n,x}^5}{6q_n}\right).$$
But $B_{n,x}>\epsilon^{1/5}$ and $\|x+i\alpha\|=B_{n,x}/q_n$, so $\|\theta\|\geq B_{n,x}/(4q_n)$. Since $f'''(x)=\frac{-2A_-}{x^3}+\frac{2A_+}{(1-x)^3}+g'''(x)$, we have
$$
f'''(\theta)\frac{B_{n,x}^5}{6q_n}={\rm O}\left(\frac{4q_n}{B_{n,x}}\right)^3\frac{B_{n,x}^5}{6q_n}={\rm O}(B_{n,x}^2q_n^2)={\rm O}(q_n^2)$$
which completes the proof.
\end{proof}

\begin{lemma}[{Control of Birkhoff sums of higher derivatives}]\label{lem:control} There exist $C(f),c(f)>0$ such that for every $n\in\N$, $x\in\T$ satisfying $B_{n,x}<c(f)$ and for $k\in[1,\frac{B_{n,x}q_{n+1}}{6q_n}]\cap \Z$, we have
\begin{equation}\label{eq:control2}
c(f)kf''(x+i(n,x)\alpha)\leq|f''^{(kq_n)}(x)|\leq C(f)kf''(x+i(n,x)\alpha),
\end{equation}
\begin{equation}\label{eq:control3}
c(f)kf'''(x+i(n,x)\alpha)\leq|f'''^{(kq_n)}(x)|\leq C(f)kf'''(x+i(n,x)\alpha),
\end{equation}
\begin{equation}\label{eq:control4}
c(f)kf''''(x+i(n,x)\alpha)\leq|f''''^{(kq_n)}(x)|\leq C(f)kf''''(x+i(n,x)\alpha).
\end{equation}
\end{lemma}
\begin{proof}
We will show the proof of \eqref{eq:control2}, the proof of \eqref{eq:control3} and \eqref{eq:control4} follows the same lines.
As in the proof of previous lemma, we have
$$
f''^{(kq_n)}(x)=\sum_{s=0}^{k-1}f''^{(q_n)}(z_s)=\sum_{s=0}^{k-1}
f''(z_s+i(z_s,n)\alpha)+{\rm O}(kq_n^2).$$
We will show first that there is a constant $C>0$ such  that
\be\label{doda101}
C^{-1}f''(x+i(n,x)\alpha)\leq f''(z_s+i(n,z_s)\alpha)\leq Cf''(x+i(n,x)\alpha).\ee
Indeed:
\begin{itemize}
\item If $x+i(n,x)\alpha$ and $z_s+i(n,z_s)\alpha$ are not in $\left[\frac{-1}{10q_n},\frac1{10q_n}\right]$ then both numbers $f''(x+i(n,x)\alpha)$ and $f''(z_s+i(n,z_s)\alpha)$ are of order $q_n^2$ since $f''(t)=\frac{A_-}{t^2}+\frac{A_+}{(1-t)^2}+g''(t)$, so~\eqref{doda101} holds for a relevant constant.
\item if $x+i(n,x)\alpha\in \left[\frac{-1}{10q_n},\frac1{10q_n}\right]$ or $z_s+i(n,z_s)\alpha\in \left[\frac{-1}{10q_n},\frac1{10q_n}\right]$ then,
    $$
    \|sq_n\alpha\|\leq s\|q_n\alpha\|\leq\frac{B_{n,x}q_{n+1}}{6q_n}
    \cdot\frac1{q_{n+1}}=\frac{\|x+i(n,x)\alpha\|}6.$$
    Hence, as before, $i(n,z_s)=i(n,x)$ and
    $$
    z_s+i(n,x)\alpha\in\left[x+i(n,x)\alpha-\frac{\|x+i(n,x)\alpha\|}{6},x+i(n,x)\alpha+
    \frac{\|x+i(n,x)\alpha\|}{6}\right].$$
 For simplicity, for the rest of the proof, we denote $i=i(n,x)$. Then (for relevant constants)
 $$
 C_1f''(x+i\alpha)\leq f''\left(\frac56(x+i\alpha)\right)=f''(x+i\alpha-\frac{x+i\alpha}{6})
 \leq C_1'f''(z_s+i\alpha)$$
 and a lower bound will be obtained in the same way by using
 $f''(\frac76(x+i\alpha))$ instead of $f''(\frac56(x+i\alpha)$,
 whence~\eqref{doda101} holds.
 \end{itemize}
 Now, in view of~\eqref{doda101}, we have
 $$
 f''^{(kq_n)}(x)=\sum_{s=0}^{k-1}
f''(z_s+i(z_s,n)\alpha)+{\rm O}(kq_n^2)\leq Ckf''(x+i\alpha)+{\rm O}(kq_n^2)<\widetilde{C}kf''(x+i\alpha)
$$
as $f''(x+i\alpha)$ is \emph{at least} of order $q_n^2$.
For the lower bound, by~\eqref{doda101}, we have
$$
 f''^{(kq_n)}(x)\geq C^{-1}kf''(x+i\alpha)+{\rm O}(kq_n^2)
$$
and we want to show that the latter term is $\geq C''kf''(x+i\alpha)$, that is, for some constant $D>0$ (implicit  in ${\rm O}(kq_n^2)$), we want to have
$$
C^{-1}kf''(x+i\alpha)-Dkq_n^2)>C''kf''(x+i\alpha).$$
Hence, we want to have
$
(C^{-1}-C'')f''(x+i\alpha)>Dq_n^2$
and for that we only need $f''(x+i\alpha)>D'q_n^2$.
\end{proof}

\begin{remark}\label{r:jednost}\em
Note that in the course of the above proof, we showed
$$
|f''^{(kq_n)}(x)|={\rm O}\left(k f''(x+i(n,x)\alpha)+q_n^2\right)$$
with no assumption on $x\in\T$. It follows that the RHS inequality in~\eqref{eq:control2} requires no assumption on $x$.

Recall also that  we have  $\|x+i(n,x)\alpha\|<\frac1{q_n}$ for each $x\in\T$, so we can choose $n_0$ so that $f''(x+i(n,x)\alpha)>0$ for each $n\geq n_0$ and all $x\in\T$ (as $f''(y)$ is of order $1/\|y\|^2$).\end{remark}

\
We also have the following lemma:
\begin{lemma}\label{cor:cont} There exist $D>0$, $n_0\in \N$  such that for every $x\in \T$, $n\in \N$, $n\geq n_0$ and every $w\in\left[q_{n_0},\max\left(\frac{B_{n,x}q_{n+1}}{4}-q_n,0\right)
\right]\cap \Z$, we have
\begin{equation}\label{eq:cont}
|f''^{(w)}(x)|\leq D\frac w{q_n}f''(x+i(n,x)\alpha).
\end{equation}
\end{lemma}
\begin{proof}
Fix $n$ and $x\in\T$ and let $w$ be arbitrary. Then
\be\label{e:ro1}
w=kq_n+s,\; 0\leq s<q_n.\ee
Since $f''(x)=\frac{A_-}{x^2}+\frac{A_+}{(1-x)^2}+g''(x)$, we have
$
f''^{(w)}(x)\geq -C_1w$,
where $C_1=\|g''\|_{\infty}$. By the same token
$$
f''^{((k+1)q_n-w)}(x+w\alpha)\geq -C_1((k+1)q_n-w)\geq -C_1q_n.$$
Thus
$$
-C_1w<f''^{(w)}(x)<f''^{((k+1)q_n)}(x)+C_1q_n,$$
whence (assuming that $C_1>1$)
\be\label{eq:nerty}
\left|f''^{(w)}(x)\right|\leq C_1\max\left(w,q_n+\left|f''^{((k+1)q_n)}(x)\right|\right).
\ee

Let us now consider $w$ as in the assumption.
Since $f''(x+i(n,x)\alpha)$ is at least of order $q_n^2$,  $C'\frac w{q_n}f''(x+i(n,x)\alpha)\geq wq_n$ for some constant $C'>0$. Therefore $w={\rm O}\left(\frac w{q_n}f''(x+i(n,x)\alpha)\right)$ and  $q_n={\rm O}\left(\frac{w}{q_n}f''(x+i(n,x)\alpha)\right)$. Therefore \eqref{eq:nerty} follows by showing
$$
\left|f''^{((k+1)q_n)}(x)\right|={\rm O}\left(\frac{w}{q_n}f''(x+i(n,x)\alpha)\right).$$
But $k+1$ satisfies the assumptions of Lemma~\ref{lem:control}, so the above follows from~\eqref{eq:control2} and \eqref{e:ro1}.
\end{proof}

\subsection{A combinatorial lemma}
We present here a combinatorial lemma which will be used in the proof of Theorem~\ref{main:th}.

\begin{lemma}[{Combinatorial lemma}]\label{lem:cpq} Let $U,V\in \R\setminus\{0\}$. Fix $p,q\in\R\setminus\{0\}$, $p\neq q$.  If $\frac{p}{q}\notin\{1,\frac{U}{V},\frac{V}{U}\}$, then for every  $K>0$ there exists  $c=c(p,q,U,V,K)>0$ such that for every $0<A<10c$, for every $0\neq B\in \R$ and for every $j_1,j_2\in\{U,V\}$, we have
\begin{equation}\label{eq:ab}
\max\left(|qj_1A^{-2}-pj_2B^{-2}|,
|q^2j_1A^{-3}-p^2j_2B^{-3}|\right)\geq K.
\end{equation}
\end{lemma}
\begin{proof} The proof goes by contradiction. Assume that for some $K$, there are $0<A_n\to 0, 0\neq B_n\in \R$ and $j_{1,n},j_{2,n}\in\{U,V\}$ such that \eqref{eq:ab} does not hold, that is
$$
|qj_{1,n}A_n^{-2}-pj_{2,n}B_n^{-2}|<K\text{ and }
|q^2j_{1,n}A_n^{-3}-p^2j_{2,n}B_n^{-3}|< K.$$
Denote $C_n=\frac{B_n}{A_n}$. Then
$$
|qj_{1,n}-pj_{2,n}C_n^{2}|<KA^2_n\text{ and }
|q^2j_{1,n}-p^2j_{2,n}C_n^{3}|< KA_n^3.$$
Hence $C_n$ is bounded, and, passing to a subsequence if necessary, we can assume that $C_n\to C$ and $j_{i,n}=j_i$ for $i=1,2$. It follows that $qj_1=pj_2C^2$ and $q^2j_1=p^2j_2C^3$. Raising to the third power the first,  squaring for the second equality and dividing, we obtain
$\frac{p}{q}=\frac{j_2}{j_1}\in\{1,\frac{U}{V},\frac{V}{U}\}$, a contradiction.
\end{proof}

\section{Disjointness in Arnol'd flows (proof of Theorem~\ref{main:th2})}\label{s:Sec7}

We begin by giving now an overview of the steps of  the proof and the general strategy to prove Theorem~\ref{main:th2}.

\smallskip

\textbf{Strategy of proof.} Let $((R_\alpha)_{t}^{f})$ be an Arnol'd flow. To prove the  disjointness result in Theorem \ref{main:th2} it is enough (by~Remark~\ref{spflor}) to prove that for any real numbers $ p,q>0$, $p\neq q$, the special flows $((R_\alpha)_{t}^{pf})$ and $((R_\alpha)_{t}^{qf})$ are disjoint (see the beginning of {Section}~\ref{sec:beginning_proof}). In order to show this, we will exploit the disjointness criterium for special flows given {by  Proposition~\ref{cocycle} proved in Section}~\ref{s:Sec5}.

Let us recall that Proposition~\ref{cocycle} involves the definition of sets $E_k$ (to which the initial condition $x$ belongs), with {the} corresponding automorphisms $A_k$ (so that  $x':=A_k(x)$) and a set $Z$ (to which the points $y,y'$ belong). In {Section}~\ref{sec:beginning_proof}, we first define the sets $E_k$ so that points $x \in E_k$ do  not  come too close to the singularity at all (large) scales (this is made precise by \eqref{eq:es}). More importantly, for $x\in E_k$, we want the $q_{n_k}$ orbit  {(here $q_{n_k}$ are the denominators of $\alpha$ at the special times $(n_k)$ given by ($\cD$2) in the Diophantine condition, see Definition~\ref{rotationassumptions})} to come close to {the} singularity but in a controlled way (see \eqref{Apq}), where {the} closeness  {is given} by a crucial  parameter $1\gg c_{p,q}>0$. We set $x'=A_{k}(x)=x+\frac{1}{q_{n_k+1}}$. Remark that, because of the definition of $(n_k)$, the orbit of $x$ of length $kq_n$, $k\leq \frac{c_{p,q}q_{n_k+1}}{q_{n_k}}$,
is almost a copy of the orbit of length $q_{n_k}$ (we have \emph{resonance}) and this crucially happens iterating both forward and backward.

The set $Z$ is such that the points $y,y'$ satisfy a Ratner-type form of shearing \emph{either} going forward \emph{or} backward in time (see \eqref{eqw}). {This is essentially} the set on which the \emph{switchable} Ratner property (see Section~\ref{Rproperties}) holds for the Arnold flow and the definition  is indeed the same as  the set $Z$ {is  in} \cite{Fa-Ka} or \cite{Ka-Ku-Ul}.
 If $y,y'$ %(see \eqref{eq:ydist})
 display this good form of Ratner-like shearing going \emph{forward}, we then show that the \emph{Forward} assumptions {$(F1)-(F4)$} in {\textbf F.}~of {Proposition}~\ref{cocycle} hold, while  if the Ratner-like form of shearing happens \emph{backward}, we show that the \emph{Backward} assumptions {$(B1)-(B4)$} in {\textbf B.}~of {Proposition}~\ref{cocycle} hold. Notice that here it is crucial that the  definition of $E_k$ is such that $x,x'$ have controlled shearing \emph{both} going forward \emph{and} backward, since the choice of whether to verify the set of assumptions {\textbf F.}~or {\textbf B.}~depends on the pair $y,y'$ only.

Let us outline the proof if $y,y'$ are \emph{forward} 'good'. The properties $(F1)$, $(F3)$ and $(F4)$ of {\textbf F.}~of {Proposition}~\ref{cocycle} all depend on a \emph{slow} and controlled form of shearing between orbits and follow by the fact that $x,x',y,y'$ approach the singularity in a controlled way (they {do not} come {too} close) and we may use Denjoy-Koksma type estimates (see Section \ref{sec.DK}). This is the content of Lemma \ref{lem:cruc}, which is then proved in {Section~\ref{sec:cruc}}.
The main difficulty is to prove property $(F2)$ of {Proposition}~\ref{cocycle}, which gives the \emph{splitting} {\'a} la Ratner of nearby trajectories. This is the content of Proposition \ref{prop:xy} (in the forward case, the backward being analogous). 
 Here the proof splits into two cases, according to what is the relative order of {magnitude} of the distances between the points $x,x'$ (given by \eqref{eq:xdist}) and between the points $y,y'$ (see \eqref{eq:ydist})

If \textbf{Case 1.}~(called \emph{asynchronous splitting} case) holds, these {orders} of magnitude are different. In this case the proof {becomes} simpler. We know in this case that $x,x'$ and $y,y'$ both display a Ratner-like shearing property going forward, but since their distances are \emph{not} comparable, one pair will start splitting and the other one will still stay close together (depending on which distance is larger), this is why we say that the shearing is \emph{asynchronous}. Hence property $(F2)$ of {Proposition}~\ref{cocycle} in this case reduces essentially to the Ratner property (forward) for one of the two flows (i.e. either $|f^{(n)}(x)-f^{(n)}(x')|\leq \epsilon$ or $|f^{(\zeta n)}(y)-f^{(\zeta n)}(y')|\leq \epsilon$ and the realignement with a shift $r_{p,q}$ is caused only by splitting for the other term).

\textbf{Case 2.} (dabbed \emph{second order splitting}) {corresponds to the delicate case when} the distances between $x,x'$ and $y,y'$ are such that the shearing phenomena in the two flows, in the \emph{main order}, are comparable and hence there is {\Blue a cancellation} of the main order of shearing. This is  the most difficult (and technical) part of the proof. It is here that one has to understand the behaviour of the second order terms and use fine estimates on Birkhoff sums growth for further {derivatives}. The arguments in  {this case} consist of two parts.  The first part consists {in} showing that if the Birkhoff sums split by some $p\in P$ then they stay $\epsilon$-close for a $\kappa$ proportion of time (see \eqref{eq:awf}) and {the} second is that they will split at some point (see \eqref{eq:inp}). This part is split into two further subcases \textbf{2(a)} and \textbf{2(b)}. We give a local outline of the arguments used in this part in {Section~\ref{sec:xy}}.

Let us finish the outline by giving the scale { of the} constants that appear{:} we will show that the \emph{shifts} space is $P=\{-r_{p,q},r_{p,q}\}$ and the constant $c_{p,q}$ involved in the definition of the sets $(E_k)$ is such that $1\gg c_{p,q}\gg r_{p,q}>0$, which should be understood that $c_{p,q}$ is smaller than any algebraic expression involving constants, $p$ and $q$ and similarly $r_{p,q}$
is smaller than any algebraic expression involving $c_{p,q}$ (it is here that the combinatorial Lemma~\ref{lem:cpq} is used).

\subsection{Set up and beginning of the proof}\label{sec:beginning_proof}
Let us first remark that,  for any given $\R\ni p,q>0$, $p\neq q$, in order to show that $((R_\alpha)_{p^{-1}t}^{f})\perp ((R_\alpha)_{q^{-1}t}^{f})$, it is enough to prove that $((R_\alpha)_{t}^{pf})\perp ((R_\alpha)_{t}^{qf})$. This follows from Remark~\ref{spflor}, since for a fixed $0\neq r\in \R$, the flow $((R_\alpha)_{r^{-1}t}^{f})$ acting on $(\T^f,\lambda^f)$ is isomorphic to $((R_\alpha)_{t}^{rf})$ acting on $(\textcolor{red}{\T^{rf}},\lambda^{rf})$.

To show this, we will use Proposition~\ref{cocycle} for $f,g$ replaced with $pf$ and $qf$, respectively. As no confusion arises, to simplify notation we will drop all ' in Proposition~\ref{cocycle} (so we will have $X_k$, $A_k$, $\epsilon$, $N$, etc.). From now on $p,q\in\R$ are fixed. We will assume that $0<p<q$ and hence $\zeta<1$.

\smallskip
To apply  Proposition~\ref{cocycle}, we first of all need to define the sets $(E_k)$, the {automorphisms}  $(A_k)$ and the set $Z$ which {appear} in it. This is the goal of this subsection. In the proof, the  sets will not depend on $\epsilon$ (i.e.\ they are defined globally).

\smallskip
\noindent \textbf{Definition of {the sets} sequences $(E_k)$, $(X_k)$ and $(A_k)$.}
Let us first define the sets $E_k$ so that {each} $x\in E_k$ {does not come} too close to the singularity at all (large) scales (this will be given by {\eqref{eq:es}}), but the $q_{n_k}$ orbit does come close to {the} singularity in a controlled way (this will be given by~\eqref{Apq}). Given $s\in \N$, we let us first define
\begin{equation}\label{eq:es}
E'(s):=\left\{x\in\T\;:\; x\notin\bigcup_{i=-q_s}^{q_s}R_\alpha^i
\left[-\frac{1}{q_s\log^2q_s},\frac{1}{q_s\log^2q_s}\right]\right\}
\end{equation}
and let $E^{s_0}=\bigcap_{s\geq s_0}E'(s)$. Notice that since $(q_s)$ grows exponentially, $\lambda(E^{s_0})\to 1$ as $s_0\to +\infty$.

Let $(q_{n_k})$ be the sequence from ($\cD$2). We now specify the important constant $c_{p,q}>0$. We want $c_{p,q}>0$ to be a small number, smaller than any expression involving $1,p,q$ which will appear in the course of the proof (below). In particular, we require (see Lemma~\ref{cor:bux} for $C'(f)$ and Lemma~\ref{lem:cpq} for $c(\cdot,\cdot,\cdot,\cdot,\cdot)$)
$$
c_{p,q}<\frac{1}{100}c(p,q,A_-,A_+,
\min((C'(f)+1)|p|+|q|+1,(C'(f)+1)(p^2+q^2)+1)).
$$
Define
\begin{equation}\label{Apq}
E_k:=\left(\bigcup_{i=0}^{q_{n_k}-1} R^i_\alpha\left[\frac{2c_{p,q}}{q_{n_k}}, \frac{3c_{p,q}}{q_{n_k}}\right]\right)\cap E^{s_0},
\end{equation}
where $s_0$ is such that  for every $k\in \N$, $\lambda(E_k)\geq c_{p,q}/2>0$.

\smallskip
We now define the sequences $(A_k)$ and $(X_k)$. Let $\delta_k:=\frac{1}{q_{n_k+1}}$, $X_k:=\T$ and $A_k(x)=x+\delta_k$. Notice that $A_k\to Id$ uniformly and obviously $\lambda(X_k)\to \lambda(\T)$. When $k$ is fixed, in view of
the definitions of $E_k$ and $A_k$, by taking $x\in E_k$ and letting $x'=A_kx$, we obtain:
\begin{equation}\label{eq:forback}
\left(\bigcup^{[c_{p,q}q_{n_k+1}]}_{i=-[c_{p,q}q_{n_k+1}]}
R_\alpha^{i}[x,x']\right)\cap
\left[-\frac{c_{p,q}}{2q_{n_k}},\frac{c_{p,q}}{2q_{n_k}}\right]
=\emptyset.
\end{equation}

\smallskip
\noindent \textbf{Definition of $P$ and $Z$.}
We {first} define the set  $P:=\{-r_{p,q},r_{p,q}\}$, where $r_{p,q}>0$ is chosen as follows. Let $\bar{p}=\min(p,p^{-1},q,q^{-1},\zeta,\zeta^{-1})$ and
$$
r_{p,q}:=\left(\frac{c_{p,q}\bar{p}}{100}\right)^{200}.
$$
Fix $\epsilon>0$ much smaller than $r_{p,q}$, $N\in \N$ and let $\kappa:=\epsilon^{10}$. We now define the set $Z=Z(\epsilon,N)$ (in fact, $Z$ will depend only on $\epsilon$). {Set}
\begin{equation}\label{eqw}
W_s:=\left\{y\in \T\;:\; y\notin \bigcup_{i=-q_s}^{q_s}R_\alpha^i\left[-\frac{1}{q_s\log^{7/8}q_s},
\frac{1}{q_s\log^{7/8}q_s}\right]\right\}
\end{equation}
and let $Z:=\left(\bigcap_{s\geq s_0,s\notin K_\alpha}W_s\right)\cap E^{s_0}$, where $s_0=s_0(\epsilon)$ is such that $\lambda(Z)\geq 1-\epsilon$ ($s_0$ does exist since $\lambda(W_s)\geq 1-\frac{2}{\log^{7/8}q_s}$ and ($\cD1$) holds for $\alpha$, cf.~\eqref{eq:es}). Let $\delta=\delta(\epsilon,N):= \frac{1}{q_{s'}\log q_{s'}}$, where $s':=[\kappa^{-20}\max\{N,s_0^2\}]$. We will moreover assume that
\be\label{eq:stycz}
\frac1{\log^{1/2}\delta}<\kappa^2.\ee

Assume $d(A_k,Id)<\delta$. We now take any  $E_k\ni x$, set $x':=A_kx$,  and  let $y,y'\in Z$, $\|y-y'\|<\delta$ and we have to show that either of A.\ or B.\ in Proposition~\ref{cocycle} holds for $x,x',y,y'$. Let $v\in \N$ be unique such that
\begin{equation}\label{eq:ydist}
\frac{1}{q_{v+1}\log q_{v+1}}<\|y-y'\|\leq \frac{1}{q_v\log q_v}. \qquad { \text{($y$-variation scale)}}
\end{equation}
 By the definition of $A_k$ and ($\cD$2), we have
\begin{equation}\label{eq:xdist}
\|x-x'\|=\frac{1}{q_{n_k+1}}\in\left[\frac{1}{q_{n_k+1}\log q_{n_k+1}},\frac{1}{q_{n_k}\log q_{n_k}}\right]. \qquad { \text{($x$-variation scale)}}
\end{equation}
Let, moreover
\begin{equation}\label{eq:time}
T=T(x,x',y,y'):=\zeta c_{p,q}\min(\|x-x'\|^{-1},\|y-y'\|^{-1}, q_{v+1}).\qquad { \text{(choice of $T$)}}
\end{equation}

\subsection{Forward or backward shearing and concluding arguments}
We {\Blue now} explain the mechanism which allows {one} to choose whether to verify the \emph{Forward} assumptions \textbf{F.}~or the \emph{Backward} assumptions \textbf{F.} of Proposition~\ref{cocycle}. We state some {lemmas} which correspond to the verification of these assumptions and then give the concluding arguments of the {proof of
Theorem~\ref{main:th2}}.
%For the proof we will need some lemmas stated below.

{In the two following subsections all parameters (i.e.\ $\epsilon, \kappa, Z,\delta$) are as defined above}.

\subsubsection{Forward or backward continuity}
The first ingredient needed for the proof of Theorem~\ref{main:th2} is  the following lemma, which was proved as Lemma~4.7 in~\cite{Fa-Ka} and allows to show  that certain orbits of nearby points,  \emph{either} in the \emph{past}, \emph{or} in the \emph{future}, do not get too close to the singularity at the origin. This {lemma} plays a crucial role in~\cite{Fa-Ka} to prove the {switchable} Ratner (SR) property and determines whether the shearing needed in the SR-property can be seen moving \emph{forward} (or in the future), or \emph{backward} (or in the past).

\begin{lemma}[Forward or backward continuity]\label{lem:forback}Fix $s\in \N$ and let $y,y'\in\T$.
If at least one of the following holds:\\
(a) $s\in K_\alpha$,\\
(b) $s\notin K_\alpha$ and
\begin{equation}\label{eq:forback2}
[y,y']\cap \bigcup_{i=-q_s}^{q_s}R_\alpha^i\left[-\frac{1}{q_s\log^{7/8}q_s},
\frac{1}{q_s\log^{7/8}q_s}\right]=\emptyset,
\end{equation}
then one of the following holds:
\begin{equation}\label{eq:forwsing}[y,y']\cap \bigcup_{i=0}^{[\frac{q_{s+1}}{4}]}
R_\alpha^{-i}\left[-\frac{1}{q_s\log^{7/8}q_s},
\frac{1}{q_s\log^{7/8}q_s}\right]=\emptyset.
\end{equation}
or
\begin{equation}\label{eq:backsing}
[y,y']\cap \bigcup_{i=0}^{[\frac{q_{s+1}}{4}]}
R_\alpha^{i}[-\frac{1}{q_s\log^{7/8}q_s},
\frac{1}{q_s\log^{7/8}q_s}]=\emptyset.
\end{equation}
\end{lemma}

\subsubsection{Forward control}\label{sec:for}
In this subsection we will assume that $x,x'$ are as in \eqref{eq:xdist} and satisfy \eqref{eq:forback}, and $y,y'$ are as in \eqref{eq:ydist} and satisfy \eqref{eq:forwsing} (with $s=v$) and $T$ is as in \eqref{eq:time}.

\begin{remark}\em \label{r:wyja}
Notice that if $v\in K_\alpha$ then either~\eqref{eq:forwsing} or~\eqref{eq:backsing} holds. We assume here that~\eqref{eq:forwsing} holds,~\eqref{eq:backsing} will be treated in the next subsection. If $v\notin K_\alpha$ then the same is true provided that we show the validity of~\eqref{eq:forback2}. It holds indeed as $y,y'\in Z$, cf.~\eqref{eqw}.
\end{remark}

We have the following two lemmas.
\begin{lemma}\label{lem:cruc} The following is true:
\begin{enumerate}
\item[(H1)] If $t$ is such that $n(x,s,t)\in [\frac{T}{\log^2 T},T]\cap \Z$, then~property (F1) of {Proposition}~\ref{cocycle} holds (for each $r\in\R$; we deal here with $S^g$ replaced with $(R_\alpha)^{qf}$, and $n(x,s,t)$ is considered for $(R_\alpha)^{pf}$ with $s\in\R$ {arbitrary, $(x,s)\in \T^{pf}$}).
\item[(H2)] If $w\in[0,T]\cap \Z$, then~(F3) of {Proposition}~\ref{cocycle} holds (here and in (H3), $f$, $g$ are replaced with $pf$, $qf$, respectively).
\item[(H3)] If $w,u\in [0,T]\cap \Z$, $|w-u|<T^{2/3}$, then~(F4) of {Proposition}~\ref{cocycle} holds.
\end{enumerate}
\end{lemma}

Let $a_w:=\left((pf)^{(w)}(x)-(pf)^{(w)}(x')\right)-\left((qf)^{([\zeta w])}(y)-(qf)^{([\zeta w])}(y')\right)$.

\begin{lemma}\label{prop:xy}There exists an interval $[R,S]\subset \left[\frac{T}{\log^2T},T\right]$ such that for all $w\in[R,S]\cap \Z$ such that $[w,w+\kappa w]\subset[R,S]$ and all $s\in [w,w+\kappa w]\cap \Z$, we have
\begin{equation}\label{eq:awf}
|a_{w}-a_s|<\epsilon^{3/2}
\end{equation}
and, moreover, there exists $w_0\in [R,(1-\epsilon)S]\cap \Z$ such that
\begin{equation}\label{eq:inp}
\left|a_R\right|<\frac{r_{p,q}}2,\;\;\left|a_{w_0}\right|\geq r_{p,q}.
\end{equation}
\end{lemma}

We will give proofs of Lemma \ref{lem:cruc} and Lemma~\ref{prop:xy} in Sections \ref{sec:cruc} and \ref{sec:xy}.

\subsubsection{Backward control}\label{sec:back}
In this subsection we keep unchanged our assumptions on $x,x'$ and $y,y'$ and $T$. We assume that $y,y'$ satisfy~\eqref{eq:backsing}, cf.\ Remark~\ref{r:wyja}. We have the following two lemmas.
\begin{lemma}\label{lem:crucback} The following is true:
\begin{enumerate}
\item[(H1)] If $t$ is such that $n(x,s,-t)\in [\frac{T}{\log^2 T},T]\cap \Z$, then~(B1) of {Proposition}~\ref{cocycle} holds.
\item[(H2)] If $w\in[0,T]\cap \Z$, then~(B3) of {Proposition}~\ref{cocycle} holds.
\item[(H3)] If $w,u\in [0,T]\cap \Z$, $|w-u|<T^{2/3}$, then~(B4) of {Proposition}~\ref{cocycle} holds.
\end{enumerate}
\end{lemma}

Let $a_w:=\left((pf)^{(w)}(x)-(pf)^{(w)}(x')\right)-\left((qf)^{([\zeta w])}(y)-(qf)^{([\zeta w])}(y')\right)$.

\begin{lemma}\label{prop:xyback}There exists an interval $[R,S]\subset \left[\frac{T}{\log^2T},T\right]$ such that for all $w\in[R,S]\cap \Z$ such that $[w,w+\kappa w]\subset[R,S]$ and all $s\in [w,w+\kappa w]\cap \Z$, we have
\begin{equation}\label{eq:awfback}
|a_{-w}-a_{-s}|<\epsilon^{3/2}
\end{equation}
and, moreover, there exists $w_0\in [R,(1-\epsilon)S]$ such that
\begin{equation}\label{eq:inpback}
\left|a_{-R}\right|<\frac{r_{p,q}}2,\;\;\left|a_{-w_0}\right|\geq r_{p,q}.
\end{equation}
\end{lemma}

\subsubsection{Concluding arguments}\label{sec:conclusion}
We now show that, using the previous {lemmas}, we can conclude the proof of Theorem~\ref{main:th2}. The next section will be then devoted only to the proof of the {lemmas}.

\begin{proof}[Proof of Theorem \ref{main:th2}]
As we {have} already remarked at the beginning of {Section~\ref{sec:beginning_proof}}, it is sufficient to verify the assumptions of Proposition \ref{cocycle}, for $(E_k), (A_k)$ and $Z$ defined as in {that section}.

If~\eqref{eq:forwsing} holds, we show that A.\ in Proposition \ref{cocycle} is true using Lemma~\ref{lem:cruc} and Lemma~\ref{prop:xy}. If~\eqref{eq:backsing} holds, we show that B.\ is true using Lemma~\ref{lem:crucback} and Lemma~\ref{prop:xyback}. Since the proofs follow the same lines, we will assume that \eqref{eq:forwsing} is satisfied and show that A. holds.

We will show that A.\ in Proposition~\ref{cocycle} holds.
Notice that by~\eqref{eq:awf}, we have  $|a_{w+1}-a_w|\leq \epsilon^{3/2}$ for $w\in[R,S]$. This and~\eqref{eq:inp} imply that there exists $w'\in [R,(1-\epsilon)S]$ such that $|a_{w'}\pm r_{p,q}|<2\epsilon^{3/2}$. Define $M:=w'$ and $L:=\kappa w'$. Then by~\eqref{eq:awf}, for $s\in[M,M+\kappa M]\cap \Z$, we get
$$
|a_s\pm r_{p,q}|\leq |a_s-a_{w'}|+|a_{w'}\pm r_{p,q}|\leq \epsilon^{3/2}+2\epsilon^{3/2}<\epsilon.
$$
This implies that property (F2) of {Proposition}~\ref{cocycle} holds on $[M,M+L]\cap \Z$ (we recall that $P=\{r_{p,q},-r_{p,q}\}$). Moreover, $[M,M+L]\subset [R,S]\subset [\frac{T}{\log^2 T},T]$, whence  (H1), (H2) and (H3) (in Lemma~\ref{lem:cruc}) apply in particular on $[M,M+L]$. Therefore,~property (F1), (F3) and (F4) of {Proposition}~\ref{cocycle} hold. This finishes the proof of A.
\end{proof}

\subsection{Slow shearing properties: proof of Lemma \ref{lem:cruc}}\label{sec:cruc}
In this section we give {the} proof of Lemma \ref{lem:cruc}, in which the properties of the disjointness criterium  which depends on \emph{slow} and controlled shearing are verified for the \emph{forward} case F. The proof of the parallel Lemma~\ref{lem:crucback} for the \emph{backward} case is analogous. In the next subsection {we will  prove} Lemma~\ref{prop:xy} which verifies the crucial property~(F3) (which gives the \emph{splitting} of orbits) for the \emph{forward} case F.

\begin{proof}[Proof of Lemma \ref{lem:cruc}]
We first show that~(H2) holds. Notice that by~\eqref{eq:forback}, { the $x$-variation scale~{\eqref{eq:xdist}}} (indeed, $\|x-x'\|^{-1}=q_{n_k+1}$, whence $T<\zeta c_{p,q}q_{n_k+1}$),  ~\eqref{eq:forwsing}, the { $y$-variation scale~\eqref{eq:ydist}} and the { choice of $T$ in \eqref{eq:time}}, we have
\begin{equation}\label{eq:nojump}
0\notin[x+w\alpha,x'+w\alpha],\; 0\notin [y+w\alpha,y'+w\alpha], \mbox{ for every } w\in [0,T]\cap \Z.
\end{equation}
Therefore, for every $w\in[0,T]\cap \Z$, some $\theta_w\in[x,x']$, by \eqref{eq:forback}, { the $x$-variation scale~{\eqref{eq:xdist}}} and ($\cD$2) together with~\eqref{eq:stycz}, we have
$$
|pf(R_\alpha^wx)-pf(R_\alpha^wx')|=p|f'(R_\alpha^w\theta_w)|\|x-x'\|
\leq \frac{\tilde{A}pq_{n_k}}{c_{p,q}}\frac{1}{q_{n_k+1}}\leq \frac{1}{\log^{1/2}q_{n_k}}<\kappa^2,
$$
where $\tilde{A}$ is a constant depending on $A_-,A_+$ and $g$.
Similarly, for some $\theta'_w\in[y,y']$, by~\eqref{eq:forwsing} and the  { $y$-variation scale~\eqref{eq:ydist}}, we have
$$
|qf(R_\alpha^wy)-qf(R_\alpha^wy')|=q|f'(R_\alpha^w\theta'_w)|\|y-y'\|
\leq \tilde{A}qq_v\log^{7/8}q_v\frac{1}{q_{v}\log q_v}\leq$$$$ \frac{1}{\log^{1/9}q_{v}}< \frac{1}{\log^{1/9}q_{s'}}\leq \frac{1}{s'^{1/10}}<\kappa^2.
$$
Therefore~(F3) of {Proposition}~\ref{cocycle} holds. By the cocycle identity (applied to $qf^{(u)}(y)-qf^{(w)}(y)$ and then $y'$) and~\eqref{eq:nojump}, we have that~(H3) follows by showing that
$$
|qf'^{(w-u)}(\theta_u)|\|y-y'\|\leq \epsilon,
$$
for $|w-u|\in [0,T^{2/3}]\cap \Z$, $u\leq T$ and some $\theta_u\in[y+u\alpha,y'+u\alpha]$. By~\eqref{eq:forwsing} (since $u\leq T\leq c_{p,q}q_{v+1}$, cf.~the { choice of $T$ in~\eqref{eq:time}}), it follows that $\theta_u\notin \Sigma_v(1/4)$ (see~\eqref{eq:sigma}). Therefore, by~\eqref{eq:smallder} for $n=v$, since $T^{2/3}\leq q_{v+1}^{2/3}\leq \epsilon^4 q_v$ (cf.~the assumption ($\mathcal{D}3$) on the rotation number)
 and by the $y$-variation {scale~\eqref{eq:ydist}}, we get
$$
|qf'^{(w-u)}(\theta_u)|\|y-y'\|\leq q\epsilon^2q_v\log q_v\cdot \frac1{q_v\log q_v}= q\epsilon^2<\epsilon.
$$
This finishes the proof of~(H3). It remains to show~(H1). We will show the first part of property~(F1) of {Proposition}~\ref{cocycle} (i.e.\ the inequality involving  $n(x,s,t)$), the proof of the second one following the same lines. By definition
\be\label{eq:szac}
-pf(x)<-s\leq t-pf^{(n(x,s,t))}(x)\leq pf(R_\alpha^{n(x,s,t)}x).
\ee
Let $\ell\in \N$ be unique such that $q_\ell\leq T<q_{\ell+1}$. We have
$$
q_{\ell+1}>T>\zeta c_{p,q}\min(q_v\log q_v,q_{n_k},q_{v+1})>\zeta c_{p,q}\min(q_v,q_{n_k})>\min(q_{v-\tilde b},q_{n_k-\tilde b}),$$
where $\tilde{b}$ is a constant (in fact, $\tilde{b}=2\log(\zeta c_{p,q})$),
so $\ell+1>\min(v-\tilde b,n_k-\tilde b)$. Furthermore,
$$
\frac1{q_{s'}}>\frac1{q_{s'}\log q_{s'}}=\delta>\|x-x'\|=\frac1{q_{n_k+1}},$$
$$\frac1{q_{s'}\log q_{s'}}=\delta>\|y-y'\|=\frac1{q_v\log q_v},$$
whence $s'<v$ and $s'<n_k+1$. It follows that $\ell+1>s'-1-\tilde b>s'/2>s_0^2$ and finally $\ell>s_0$.

Since $x\in E_{k}\subset E^{s_0}\subset E'(\ell+1)$ (see \eqref{eq:es} and~\eqref{Apq}) and $n(x,s,t)\leq T<q_{\ell+1}$, it follows that $\min(\|x\|,\|R_\alpha^{n(x,s,t)}x\|)\geq \frac1{q_{\ell+1}\log^2 q_{\ell+1}}$. Therefore, by the definition of $f$, we have
$$
\max(pf(x),pf(R_\alpha^{n(x,s,t)})<\tilde{A}_1\log q_{\ell+1}<T^{1/10}
$$
cf.\ ($\cD$3) (of Definition \ref{rotationassumptions}).
Coming back to~\eqref{eq:szac}, we have shown that
$$
|t-pf^{(n(x,s,t)}(x)|<T^{1/10}.$$
Moreover, in view of~\eqref{eq:szac},
$t\geq pf^{(n(x,s,t))}(x)-T^{1/10}\geq pn(x,s,t)\inf_\T f  -T^{1/10}\geq \frac{T}{\log^3T}$, the latter inequality follows from the lower bound on $n(x,s,t)$.

Since
$$
|t-n(x,s,t)\int pf\,d\lambda|\leq
|pf^{(n(x,s,t))}(x)-n(x,s,t)\int pf\,d\lambda|+|t-pf^{(n(x,s,t))}(x)|$$
and the second summand is bounded by $T^{1/10}$, it is enough to show that the first summand is bounded by $T^{1/4}$ as $T^{1/4}+T^{1/10}<t^{1/3}$ (by the lower bound on $t$ we have just shown).
Therefore, we need to show that
$$
|pf^{(n(x,s,t))}(x)-n(x,s,t)\int_\T pf\, d\lambda|<T^{1/4}.
$$
This, in turn, will follow if we show that
(since $\int_\T f d\lambda=1$)
$$
|f^{(n)}(x)-n|<T^{1/5},
$$
for $n\in [\frac{T}{\log^2 T},T]\cap \Z$. Notice that since $x\in E^{s_0}\subset E'(\ell+1)$, for $j<T<q_{\ell+1}$, (by ($\cD$3) of Definition~\ref{rotationassumptions}), we have
$$
\|x+j\alpha\|\geq \frac1{q_{\ell+1}\log^2q_{\ell+1}}\geq
\frac1{q_\ell\log^2q_\ell\cdot(2\log ^2q_\ell)}\geq\frac1{2T\log^4T}.$$
It follows that the assumptions of Lemma~\ref{lem:dk2} are satisfied. Hence (H1) holds and the proof of Lemma~\ref{lem:cruc} is finished.
\end{proof}

\subsection{Splitting of orbits: proof of Lemma~\ref{prop:xy}}\label{sec:xy}

In this section we give {the} proof of Lemma~\ref{prop:xy}, which gives the splitting of nearby trajectories and is the most delicate and technical part of the proof {of~Theorem~\ref{main:th2}}. We give at the beginning an outline of the proof, to help the reader to go trough the (sometimes complicated technically) parts of the proof.\\

\smallskip
Let us first recall that
$$
a_w^p:=pf^{(w)}(x)-pf^{(w)}(x'), \qquad
a_w^q:=qf^{([\zeta w])}(y)-qf^{([\zeta w])}(y')
$$
and remark that, by mean value (which can be applied thanks to~\eqref{eq:nojump}), for every $w\in[0,T]\cap \Z$ there exist $\theta_{w,x,x'}\in[x,x']$, $\theta_{w,y,y'}\in[y,y']$ such that
$$
a_w^p
=pf'^{(w)}(\theta_{w,x,x'})(x-x'), \quad \text{and} \qquad
a_w^q: =qf'^{([\zeta w])}(\theta_{w,y,y'})(y-y').
$$
The distances {$\|x-x'\|$ and $\|y-y'\|$}  play hence a fundamental role in comparing $a_w^p$ and $a_w^q$. Recall that the denominators $q_{n_k}$ and $q_v$ which encode the magnitude of these distances are defined by the
 { $y$-variation scale~\eqref{eq:ydist}} and $x$-variation {scale~\eqref{eq:xdist}}. \\

\smallskip
\textbf{Outline of the proof.}
%\subsubsection{Outline of the proof.}
Let $n_k$ and $v$ be the {indices} of the
denominators 
which encode the distances between $x,x'$ and $y,y'$, respectively (defined by~\eqref{eq:ydist} and $x$-variation scale~\eqref{eq:xdist}).
As already anticipated in the general outline given at the beginning of {Section~\ref{s:Sec7}} that the proof of the splitting of nearby orbits has two separate cases, namely
 \textbf{Case 1.}~(\emph{asynchronous splitting} case) when $v=n_k$ and \textbf{Case 2.}~(\emph{second order splitting}) when $v\neq n_k$.

The asynchronous splitting case is treated first and is not so difficult, since {the} property $(F2)$ of {Proposition}~\ref{cocycle} can {(in this case)} be deduced by the Ratner property (forward) for one of the two flows.

The arguments in {\textbf{Case 2.}} consist of two parts. { The first part consists {in}}  showing that if the Birkhoff sums split by some $p\in P$ then they stay {$\epsilon$-close} for a $\kappa$ proportion of time (see \eqref{eq:awf}) and {the} second is that they will split at some point (see \eqref{eq:inp}). This part is split into two further subcases \textbf{2(a)} and \textbf{2(b)}.

For this we split the cocycle inequality in property (F2) of {Proposition}~\ref{cocycle} according to \eqref{bara} and show \eqref{eq:awf} separately for $a_{1,w}$ and $a_{2,w}$ (considering two subcases \textbf{2(a)} and \textbf{2(b)}). For $a_{1,w}$ (see \eqref{aw1} and \eqref{eq:aaaw}) we use \eqref{eq:cancshort} and for $a_{2,w}$ (see \eqref{aw2}, \eqref{eq:aaw}) we use \eqref{eq:canc} (we have very precise estimates for the growth of the first derivative). The last part is to show \eqref{eq:inp} (which is also the most technical one).

In \textbf{2(a}, we use the fact that $a_{2,w}$ dominates $a_{1,w}$ and so $a_w$ is large for some $w$ since by {the} assumptions of \textbf{I}, $a_{2,w}$ is large.

Now, in case \textbf{2(b)}, we study $a_w$ for $w=mq_n$ (along {these} times we have the best control). We argue by contradiction that $a_{mq_n}$ is always small. Then $a_{(m+l)q_n}-a_{mq_n}-a_{lq_n}$ also has to be small. We show (see \eqref{a2sm} and \eqref{a2sm2}) that for $a_{2,\cdot}$  the above expression is small and hence we deal only with $a_{1,\cdot}$ (see \eqref{eq:smalldev2}). By {the} cocycle identity (and estimating the third derivative), we derive \eqref{eq:estsec} and by an analogous reasoning (using \eqref{eq:estsec}), we get \eqref{eq:estthi}. Now, by the choice of $c_{p,q}$, it follows that the main contribution to the Birkhoff sums of {the} second and third {derivatives} is given by the closest visit. Roughly, if $\frac{A}{q_n}$ denotes the closest visit of $x$ and $\frac{B}{q_n}$ the closest visit of $y$, by \eqref{eq:estsec} and \eqref{eq:estthi}, we get that $|pA^{-2}-qB^{-2}|$ and $|p^2A^{3}\pm q^2B^{-3}|$ both have to be small. This however cannot hold simultaneously (see Lemma \ref{lem:cpq}). Therefore for some $m$, $a_{qmq_n}$ is large and \eqref{eq:inp} holds. This concludes the outline {of the proof}.

\smallskip
{\Blue In the rest of the section we now present the proof by discussing separately the two cases outlined above.}
%In the rest of this subsection we consider \textbf{Case 1}. {The rest of the proof (\textbf{Case 2}) is treated in the next and last subsection.} {IS THERE REALLY THE NEXT SUBSECTION??}

\smallskip
\noindent \textbf{Case 1: {Asynchronous} shearing.}
%\subsubsection{Case 1: {Asynchronous} shearing.}
We assume in this case that $n_k\neq v$. Let $\tilde{T}:=\frac{T}{\log T}$. We will find $[R,S]$ {which is  contained} in $[\frac{T}{\log^2T},\tilde{T}]$.
By~\eqref{eq:forwsing} and since $\theta_{w,y,y'}\in[y,y']$, it follows that for every $w\in[0,T]\cap \Z$ (see \eqref{eq:sigma}), {we have}
\begin{equation}\label{eq:thy}
\theta_{w,y,y'}\notin \Sigma_{v}(1/4).
\end{equation}
Similarly, by \eqref{eq:forback} and since $\theta_{w,x,x'}\in[x,x']$, it follows that for every $w\in[0,T]\cap \Z$, {we have}
\begin{equation}\label{eq:thx}
\theta_{w,x,x'}\notin \Sigma_{n_k}(c_{p,q}).
\end{equation}
Assume first that $n_k<v$.
 In this case, by the { choice of $T$ in \eqref{eq:time}} and the  { $y$-variation scale~\eqref{eq:ydist}}, $T=c_{p,q}\zeta q_{n_k+1}\leq c_{p,q}q_v$. Therefore $\tilde{T}\leq \frac{q_v}{\log q_v}<\epsilon q_v$ (the function $u/\log u$ is increasing on the interval $(e,\infty)$)  and
by~\eqref{eq:smallder} together with~\eqref{eq:thy}, for $n=v$, $M=\frac14$ and $x=\theta_{w,y,y'}$, and then by the  { $y$-variation scale~\eqref{eq:ydist}}, we get for $w\in [0,\tilde{T}]\cap \Z$
\begin{equation}\label{eq:smallg}
|a_w^q|=|qf'^{([\zeta w])}(\theta_{w,y,y'})(y-y')|\leq q\epsilon^2q_v\log q_v\|y-y'\|\leq \epsilon^{5/3}.
\end{equation}
By~\eqref{eq:dercont} in Lemma~\ref{lem:bsums} for $n=n_k$, $M=c_{p,q}$ and $x=\theta_{w,x,x'}$, we get for $w\in [q_{n_k},\tilde{T}]\cap \Z$
\begin{equation}\label{eq:nwc}
p((A_--A_+)-\epsilon^2)w(\log w)\|x-x'\|\leq |a_w^p|\leq p((A_--A_+)+\epsilon^2)w(\log w)\|x-x'\|.
\end{equation}
Now, set
$$
R:=q_{n_k}\text{ and }S:=\tilde T.$$
Take $w,s\in[R,S]$ with $s\in[w,w+\kappa w]$. We
have
$$
|a_w-a_s|=|(a^p_w-a^q_w)-(a^p_s-a^q_s)|\leq |a^p_w-a^p_s|+|a_w^q|+|a^q_s|\leq |a^p_w-a^p_s|+2\epsilon^{5/3}$$
by~\eqref{eq:smallg}. Now, by \eqref{eq:nwc},
\be\label{eq:szac0}
|a^p_w-p(A_--A_+)w(\log w)\|x-x'\||<p\epsilon^2w(\log w)\|x-x'\|\ee
(and the same holds for $s$; if $a^p_w$ is negative we have to modify signs). Moreover, $w<\tilde T=T/\log T$, whence
$$
w\log w<\frac T{\log T}\log\left(\frac T{\log T}\right)<2T=2\zeta c_{p,q}q_{n_k+1}.$$
It follows that
\be\label{eq:szac2}
p\epsilon^2w(\log w)\|x-x'\|=2 p\zeta c_{p,q}\epsilon^2<\epsilon^{5/3}
\ee
(remembering that $\|x-x'\|=\frac1{q_{n_k+1}}$). Thus
$$
|a^p_w-a^p_s|=p(A_--A_+)\|x-x'\||w\log w-s\log s|+{\rm O}(\epsilon^{5/3}).$$
But $s\in [w,w+\kappa w]$, so
$$
|w\log w-s\log s|\leq |w\log w-(w+\kappa w)\log(w+\kappa w)|=
{\rm O}(\kappa w\log w)$$
and therefore (see~\eqref{eq:szac2})
$$
|a^p_w-a^p_s|={\rm O}(p(A_--A_+)\|x-x'\|\kappa w\log w)+{\rm O}(\epsilon^{5/3})={\rm O}(\kappa)+{\rm O}(\epsilon^{5/3})<\epsilon^{3/2}$$
which completes the proof of~\eqref{eq:awf}.

Moreover, by~\eqref{eq:szac0}
$$
|a^p_{q_{n_k}}-p(A_--A_+)q_{n_k}(\log q_{n_k})\|x-x'\||<p\epsilon^2q_{n_k}(\log q_{n_k})\|x-x'\|.$$
But by ($\cD$3) (of Definition \ref{rotationassumptions}), $q_{n_k}(\log q_{n_k})/q_{n_k+1}\leq 1/\log\log q_{n_k}$, so $|a_R|<r_{p,q}/2$. Finally, setting $w_0=[(1-\epsilon)\tilde{T}]$, by~\eqref{eq:szac0} and~$\eqref{eq:szac2}$, we obtain
$$
|a^p_{w_0}-p(A_--A_+)w_0(\log w_0)\|x-x'\|<\epsilon^{5/3}.$$
But
$$w_0\log w_0\asymp((1-\epsilon)\tilde T\log ((1-\epsilon)\tilde T))\asymp$$$$(1-\epsilon)\frac{T}{\log T}\log\left((1-\epsilon)\frac{T}{\log T}\right)\asymp(1-\epsilon)T\asymp q_{n_k+1}.
$$
Now, by~\eqref{eq:smallg} and~\eqref{eq:nwc} for a constant $D>0$, it follows that
$$
|a_{w_0}|\geq |a_{w_0}^p|-\epsilon^{5/3}\geq D c_{p,q}-\epsilon^{5/3}>r_{p,q},
$$
so~\eqref{eq:inp} holds and the proof is finished.

Now if $n_k>v$ the proof is symmetric. One restricts the interval to $[\frac{T}{\log^2T},\tilde{T}]:=[\frac{T}{\log^2T},\frac{q_{n_k}}{\log q_{n_k}}]$ then shows that \eqref{eq:smallg} holds for $a_w^f$ and \eqref{eq:nwc} holds for   $a_w^g$. This finishes the proof in case $n_k\neq v$.

\medskip
\noindent \textbf{Case 2: Second order splitting.} $n_k=v$.  Let $C_{x,y}:=\min(B_{v,x},B_{v,y})$, cf.\ \eqref{eq.distsing}.
Notice that by \eqref{eq:forback} for $x$ and \eqref{eq:forwsing} for $y$, we have
\begin{equation}\label{eq:axby}
B_{v,x}\geq c_{p,q}/2\;\text{ and }\; B_{v,y}\geq \log^{-7/8}q_v.
\end{equation}
We will find numbers $R,S$ satisfying additionally $[R,S]\subset [\frac{T}{\log T},\frac{C_{x,y}T}{10}]$ (notice that by \eqref{eq:axby} and by the { choice of $T$ in \eqref{eq:time}}, $\frac{C_{x,y}T}{10}>2\frac{T}{\log T}$).
By \eqref{eq:nojump}, for every $w\in[0,T]\cap \Z$ there exist $\theta_{w,x,x'}\in[x,x']$, $\theta_{w,y,y'}\in[y,y']$ such that
$$
a_w^p=pf'^{(w)}(x)(x-x')+
\frac p2f''^{(w)}(\theta_{w,x,x'})(x-x')^2
$$
and
$$
a_w^q= f'^{([\zeta w])}(y)(y-y')+
\frac q2f''^{([\zeta w])}(\theta_{w,y,y'})(y-y')^2.
$$

Notice that since $\theta_{w,x,x'}\in[x,x']$, by \eqref{eq:forback} and \eqref{eq:xdist}, we have $B_{v,\theta_{w,x,x'}}\geq\frac{c_{p,q}}{2}$ for every $w\in [0,\frac{C_{x,y}T}{10}]\cap \Z$. So, by \eqref{eq:cont} in Lemma  \ref{cor:cont}, for every $w\in [0,\frac{C_{x,y}T}{10}]\cap \Z$ (recall the { choice of $T$ in \eqref{eq:time}})
\begin{equation}\label{eq:jk}
|f''^{(w)}(\theta_{w,x,x'})|(x-x')^2\leq DC_{x,y}Tq_v\frac{4}{c^2_{p,q}}\frac{1}{q_{v+1}^2}\leq \epsilon^{10}
\end{equation}
(here, we use the fact that $v=n_k$ satisfies {$(\mathcal{D}2)$}% {IS THIS CORRECT? OR IT SHOULD be (D2)?}
).
Similarly, since $\theta_{w,y,y'}\in[y,y']$, by \eqref{eq:forwsing} and the  { $y$-variation scale~\eqref{eq:ydist}}, we have $B_{v,\theta_{w,y,y'}}\geq \frac1{\log^{7/8}q_v}$ for every $w\in [0,\frac{C_{x,y}T}{10}]\cap \Z$. So by \eqref{eq:cont} in Lemma  \ref{cor:cont}, for every $w\in [0,\frac{C_{x,y}T}{10}]\cap \Z$, since $C_{x,y}\leq B_{v,y}$ and $B_{v,y}\geq \log^{-7/8}q_v$, we have (recall the choice of $T$ in~\eqref{eq:time})
\begin{equation}\label{eq:uy}
f''^{([\zeta w])}(\theta_{w,y,y'})(y-y')^2
\leq \frac{\zeta 4DC_{x,y}Tq_vB^{-2}_{v,y}}
{Tq_v\log q_v}={\rm O}\left( \frac{ \log^{7/8}q_v}{\log q_v}\right)\leq \epsilon^{10}.
\end{equation}
Since $a_w=a_w^p+a_w^q$, by \eqref{eq:jk}, \eqref{eq:uy}, it is enough to study the expression $\bar{a}_w=pf'^{(w)}(x)(x-x')-qf'^{([\zeta w])}(y)(y-y')$ for $w\in[0,\frac{C_{x,y}T}{10}]$. We rewrite it in the form
\begin{align}\label{bara}
\bar{a}_w=a_{1,w}+a_{2,w},\qquad  \text{where}\quad & a_{1,w}:=\left(pf'^{(w)}(x)-qf'^{([\zeta w])}(y)\right)(y-y'),\\
& a_{2,w}:=pf'^{(w)}(x)\left((x-x')-(y-y')\right). \nonumber
\end{align}
\smallskip
 We consider now two subcases (Case 2(a) and Case 2(b)):\\

\smallskip
\noindent \textbf{Case 2(a).} In this subcase, we assume that there exists $w_0\in[\frac{T}{\log T},\frac{C_{x,y}T}{30}]\cap \Z$ such that
$$
|a_{2,w_0}|\geq 10.
$$
Notice that by the { choice of $T$ in \eqref{eq:time}}, $\frac{T}{\log T}\geq c_{p,q}\zeta q_v$. Therefore and since $x\notin \Sigma_v(c_{p,q}/2)$ (cf.\ \eqref{eq:forback}) by \eqref{eq:dercont}, \eqref{eq:xdist} and the  { $y$-variation scale~\eqref{eq:ydist}}, we have
$$
|a_{2,[\frac{T}{\log T}]}|\leq 2p|A_--A_+|T\frac{2c_{p,q}}{T}\leq 5.
$$
Moreover, for every $w\in[\frac{T}{\log T},\frac{C_{x,y}T}{10}]\cap \Z$ by \eqref{eq:forback}, \eqref{eq:xdist} and the  { $y$-variation scale~\eqref{eq:ydist}}, we have
$$|a_{2,w+1}-a_{2,w}|=p|f'(x+w\alpha)|\left|(x-x')-(y-y')\right|\leq 4|A_--A_+| pc_{p,q}^{-1}q_v\frac{2}{q_v\log q_v}<\epsilon^3$$
and therefore there exists $w'\in[\frac{T}{\log T},\frac{C_{x,y}T}{30}]\cap \Z$ such that
\begin{equation}\label{a2w}
|a_{2,w'}\pm10|<\epsilon^2.
\end{equation}
In particular, using additionally~\eqref{eq:dercont}, we have
\be\label{eq:ad}
w'\log w'\|(x-x')-(y-y')\|={\rm O}(1).\ee
Define $R=w'$, $S:=(1+\epsilon)w'$. Then for every $w\in [R,S]$ such that $[w,w+\kappa w]\subset[R,S]$ and every $s\in[w,w+\kappa w]$, by \eqref{eq:dercont} and \eqref{eq:ad}, we have
\begin{multline}\label{aw2}
|a_{2,w}-a_{2,s}|\leq \\ \left(p|A_--A_+|(s\log s-w\log w)+{\rm O}(\epsilon^2 w\log w+\epsilon^2s\log s)\right)\|(x-x')-(y-y')\|=\\
{\rm O}(\kappa w'\log w')\|(x-x')-(y-y')\|+{\rm O}(\epsilon^2)={\rm O}(\epsilon^2).
\end{multline}
 Moreover, by \eqref{eq:cancshort} (for $x$ and $y$), we have (recall that $\zeta=p/q$) %{COULD YOU CORRECT TEX BELOW?}
\begin{multline}\label{aw1}
|a_{1,w}-a_{1,s}|=\|y-y'\|\left|(pf'^{(w)}(x)-pf'^{(s)}(x))-
(qf'^{([\zeta w])}(y)-qf'^{([\zeta s])}(y))\right|\\ ={\rm O}\left( \frac{\epsilon^{2}T}{T}\right)={\rm O}( \epsilon^{2}).
\end{multline}
By the two above, for every $w\in [R,S]$, $s\in[w,w+\epsilon^3 w]$, we have
$$
|a_w-a_s|\leq \epsilon^{3/2}.
$$
which gives \eqref{eq:awf}. Moreover, by \eqref{eq:canc} (recalling that $\zeta=\frac{p}{q}$) and by the { choice of $T$ in \eqref{eq:time}}, we have for $w'$
$$
|a_{w'}|\geq |a_{2,w'}|-|a_{1,w'}|\geq 10-\epsilon^2-
\|y-y'\|T\geq 8>r_{p,q},
$$
and so \eqref{eq:inp} holds. This finishes the proof in case \textbf{1(a)}.\\

\smallskip
\noindent \textbf{Case 2(b).} In this case, we assume that, contrary to \textbf{Case 2(a)}, for every $w\in[\frac{T}{\log T},\frac{C_{x,y}T}{30}]\cap \Z$, we have
$$
|a_{2,w}|\leq 10.
$$
This, by \eqref{eq:dercont} (since $x\notin \Sigma_l$ by \eqref{eq:forback}), means that for every $w,s\in [\frac{T}{\log T},\frac{C_{x,y}T}{30}]\cap \Z$, $s\in [w,w+\epsilon^3w]$, we have (similarly to \textbf{1(a)})
\begin{equation}\label{eq:aaw}
|a_{2,w}-a_{2,s}|\leq \epsilon^{5/3}.
\end{equation}
Moreover, for $w,s\in [\frac{T}{\log T},\frac{C_{x,y}T}{30}]$, $s\in [w,w+\epsilon^3w]$, by \eqref{eq:cancshort} and the definition of $a_{1,w}$, we have (similarly to \textbf{1(a)}.)
\begin{equation}\label{eq:aaaw}
|a_{1,w}-a_{1,s}|\leq \epsilon^{5/3}.
\end{equation}
So, by \eqref{eq:aaw} and \eqref{eq:aaaw}, it follows that \eqref{eq:awf} holds for every interval $[R,S]\subset [\frac{T}{\log T},\frac{C_{x,y}T}{30}]\cap \Z$. Therefore to finish the proof it is enough to show that there exists $w_0\in [\frac{T}{\log T},\frac{C_{x,y}T}{40}]$ such that \eqref{eq:inp} holds.

By assumption, $10\geq|a_{2,w}|=|pf'^{(w)}(x)|\left|(x-x')-(y-y')\right|$. Moreover, by the { choice of $T$ in \eqref{eq:time}}, $T\geq c_{p,q}\zeta q_v\log q_v$.
Therefore, by \eqref{eq:dercont} (notice that $x\notin \Sigma_l$ by \eqref{eq:forback}) for $w=[\frac{C_{x,y}T}{40}]$, we get (using \eqref{eq:axby} and {\Blue $(\mathcal{D}3)$}.)
%{IS THIS CORRECT OR IT SHOULD be $(D3)$?}.)
\begin{multline}\label{eq:smalldif}
10\geq \frac{p}{100}\left|(x-x')-(y-y')\right|C_{x,y}T\log T\geq\\
\frac{pc_{p,q}\zeta}{200}q_v\log^2q_v\left|(x-x')-(y-y')\right|\geq
q_{v+1}\log(\log q_v)\left|(x-x')-(y-y')\right|.
\end{multline}
This, by the $x$-scale \eqref{eq:xdist} and the $y$-scale \eqref{eq:ydist}, means that
\begin{equation}\label{eq:ry}
\frac{1}{2q_{v+1}}<\|y-y'\|<\frac{2}{q_{v+1}}.
\end{equation}
So, by the { choice of $T$ in \eqref{eq:time}}, we get
\begin{equation}\label{eq:smt}
\frac{\zeta c_{p,q}q_{v+1}}{2}<T<2\zeta c_{p,q}q_{v+1}.
\end{equation}
We will show that \eqref{eq:inp} holds for $w_0=qmq_v$ for some $m\in J_{v,T}:=[\frac{T}{qq_v\log T},\frac{C_{x,y}T}{40qq_v}]\cap \Z$. This follows by showing that for some $m,l\in J_{v,T}$  such that $m+l\in J_{v,T}$, we have
\begin{equation}\label{eq:smalldev}
|a_{q(m+l)q_v}-a_{qmq_v}-a_{qlq_v}|>3r_{p,q}.
\end{equation}
Notice that, by definition of $a_{2,w}$, we have for some $\theta'\in[x,x+qmq_v\alpha]$ (by {the} cocycle identity)
\begin{equation}\label{a2sm}
|a_{2,q(m+l)q_v}-a_{2,qmq_v}-a_{2,qlq_v}|=
|pf''^{(qlq_v)}(\theta')qm\|q_v\alpha\|\left((x-x')-(y-y')\right)|
\end{equation}
Notice that $m\in J_{v,T}$, so, by \eqref{eq:smt}, it follows that $qm\|q_v\alpha\|\leq \frac{4c_{p,q}^2}{2}$ and therefore since $\theta'\in[x,x+qmq_v\alpha]$ and \eqref{eq:forback} holds, we have $B_{u,\theta'}\geq \frac{c_{p,q}}{2}$. Therefore, by Lemma \ref{lem:control} and \eqref{eq:smalldif}, {we obtain}
\begin{equation}\label{a2sm2}
p|f''^{(qlq_v)}(\theta')|qm\|q_v\alpha\||\left((x-x')-(y-y')\right)|\leq \frac{80pq^2mlC(f)c_{p,q}^{-2}q_v^2}{q_{v+1}^2\log (\log q_v)}\leq \epsilon,
\end{equation}
{where the last inequality holds} since $m,l\in J_{v,T}$ so in particular (by \eqref{eq:smt}) $m,l\leq \frac{q_{v+1}}{40qq_v}$. Hence \eqref{eq:smalldev} follows by showing
\begin{equation}\label{eq:smalldev2}
|a_{1,q(m+l)q_v}-a_{1,qmq_v}-a_{1,qlq_v}|>4r_{p,q}.
\end{equation}
We argue by contradiction assuming that for every $m,l\in J_{v,T}$ such that $m+l\in J_{v,T}$, \eqref{eq:smalldev2} does not hold.

We will often use the following condition: If $\theta\in [x,x+qmq_v\alpha]$, $\theta'\in[y,y+\zeta pmq_v\alpha]$ and $m\in J_{v,T}$, {then}
\begin{equation}\label{eq:bvt}
2B(v,x)\geq B_{v,\theta}\geq \frac{B_{v,x}}{2}\text{ and } 2B(v,y)\geq B_{v,\theta'}\geq \frac{B_{v,y}}{2}.
\end{equation}
We will first show that
\begin{equation}\label{eq:cpq}
C(f,p,q)c_{p,q}\leq C(f,p,q)B_{v,x}<B_{v,y}<C(f,p,q)^{-1}B_{v,x}.
\end{equation}
for some constant $C(f,p,q)>0$, with $C(f,p,q)^{9}\geq r_{p,q}$. Notice that by~\eqref{eq:axby}, the right inequality is immediate ($B_{v,y}\leq 1$ and $B_{v,x}\geq c_{p,q}$). If $B_{v,y}\geq c(f)$
(see Lemma \ref{lem:control}) then~\eqref{eq:cpq} follows (since $B(v,x)<1$). Hence, we will assume that $B(v,y)\leq c(f)$
so that we can use Lemma \ref{lem:control} for $y$ and $v$. Notice  that we can use Lemma~\ref{lem:control} for $x$ since $B_{v,x}\leq c_{p,q}\leq c(f)$ (see \eqref{Apq})
Notice {also that by the} cocycle identity, for some $\theta_{m,l}\in [x,x+qmq_v\alpha]$,
$$
pf'^{(q(m+l)q_v)}(x)-pf'^{(qmq_v)}(x)-pf'^{(qlq_v)}(x)=
pf''^{qlq_v}(\theta_{m,l})qm\|q_v\alpha\|.
$$
 and for some $\theta'_{m,l}\in[y,y+\zeta qmq_v\alpha]$, we get (using that $\zeta=p/q$)
$$
qf'^{(q(m+l)q_v)}(x)-qf'^{(qmq_v)}(x)-qf'^{(qlq_v)}(x)=
qf''^{qlq_v}(\theta'_{m,l})pm\|q_v\alpha\|.
$$
Hence, by the definition of $a_{1,w}$ (recall that, {for sake of contradiction}, we assume that~\eqref{eq:smalldev2} does not hold), {it follows that}
$$
4r_{p,q}>\|y-y'\|\left|pf''^{(qlq_v)}(\theta_{m,l})qm\|q_v\alpha\|
-qf''^{(plq_v)}(\theta'_{m,l})pm\|q_v\alpha\|\right|.
$$
This, for $m=l=[\frac{C_{x,y}T}{80qq_v}]$ (then $m,l,m+l\in J_{v,T}$), imply by \eqref{eq:ry}, \eqref{eq:bvt} (for $\theta_{m,l}$ and $\theta'_{m,l}$), \eqref{eq:control2} and \eqref{eq:smt} that
\begin{multline*}
\frac{16r_{p,q}}{pq}q_{v+1}^2\geq m|f''^{(qlq_v)}(\theta_{m,l})-f''^{(plq_v)}(\theta'_{m,l})|\geq mlq_v^2(C(f)pB_{v,y}^{-2}-c(f)qB_{v,x}^{-2})\geq\\ q_{v+1}^2\frac{\zeta^2c^2_{p,q}}{10^6q}C^2_{x,y}
(C(f)pB_{v,y}^{-2}-c(f)qB_{v,x}^{-2}).
\end{multline*}
This, by the definition of $C_{x,y}$, finishes the proof of  \eqref{eq:cpq} since if it does not hold then $C_{x,y}=B_{v,y}$ and then the inequality above shows that
$$
C(f)p\leq 3c(f)qc_{p,q}^{-3}B_{v,y}^2+ \frac{10^{8}qr_{p,q}}{\zeta^2c_{p,q}^2},
$$
which is not true since $r_{p,q}>0$ is taken sufficiently small.
 Hence (see \eqref{eq:axby}), $C_{x,y}\geq c_{p,q}\min(1,C(f,p,q))$ and so the right endpoint of  $J_{v,T}$ is a fixed (depending on $p,q$ only) proportion of $q_{v+1}$ (see \eqref{eq:smt}). To emphasize that denote $C_{x,y}:=C_{p,q}$.

Notice that for some $\theta_{m,l}\in [x,x+qmq_v\alpha]$, we have
$$
pf'^{(q(m+l)q_v)}(x)-pf'^{(qmq_v)}(x)-pf'^{(qlq_v)}(x)=
pf''^{(qlq_v)}(x)qm\|q_v\alpha\|+
pf'''^{(qlq_v)}(\theta_{m,l})(qm\|q_v\alpha\|)^2.
$$
Moreover, by \eqref{eq:bvt} and \eqref{eq:control3}, we get
$$
pf'''^{(qlq_v)}(\theta_{m,l})(qm\|q_v\alpha\|)^2\leq \frac{C'(f)pq^3m^2lq_v^{3}C_{p,q}^{-3}}{q_{v+1}^2}
$$
Similarly, for some $\theta'_{m,l}\in [y,y+\zeta qmq_v\alpha]$, we have
$$
qf'^{(p(m+l)q_v)}(y)-qf'^{(pmq_v)}(x)-qf'^{(plq_v)}(x)=
qf''^{(plq_v)}(x)pm\|q_v\alpha\|+
qf'''^{(plq_v)}(\theta'_{m,l})(pm\|q_v\alpha\|)^2
$$
and by \eqref{eq:bvt} and \eqref{eq:control3}, we get
$$
qf'''^{(plq_v)}(\theta'_{m,l})(pm\|q_v\alpha\|)^2\leq \frac{C'(f)qp^3m^2lq_v^{3}C_{p,q}^{-3}}{q_{v+1}^2}.
$$
Therefore, by the definition of $a_{1,w}$ and by \eqref{eq:smalldev2}, for all $m,l,m+l\in J_{v,T}$, we have (since $p<q$ and so $p^3q<pq^3$)
\begin{equation}\label{eq:estsec}
|y-y'|m\|q_v\alpha\| pq|f''^{qlq_v}(x)-f''^{plq_v}(y)|\leq 10r_{p,q}+2\frac{C'(f)pq^3m^2lq_v^{3}C_{p,q}^{-3}}{q_{v+1}^3}.
\end{equation}
Denote the {RHS} of \eqref{eq:estsec} by $e_1(l,m)$.

Let $h,l\in J_{v,T}$ be such that $m+l+h\in J_{v,T}$. Notice that {for $b\in\{p,q\}$}, for some $\theta^b_{h,l}\in[x,x+qlq_v\alpha]$, we have
$$
f''^{(b(l+h)q_v)}(x)-f''^{(blq_v)}(x)-f''^{(bhq_v)}(x)=
f'''^{(bhq_v)}(x)bl\|q_v\alpha\|+f''''^{(bhq_v)}(\theta^b_{h,l})
(bl\|q_v\alpha\|)^2.
$$
By \eqref{eq:bvt} for $\theta^b_{h,l}$ and \eqref{eq:control4}, we have
$$
f''''^{(bhq_v)}(\theta_{h,l})
(bl\|q_v\alpha\|)^2\leq \frac{4C'(f)b^3l^2hq_v^4C_{p,q}^{-4}}{q_{v+1}^2}.
$$
Using this for $b=p$ and $b=q$, {in view of~\eqref{eq:estsec}, for $l+h,l,h$, we have}
\begin{equation}\label{eq:estthi}
|y-y'|ml\|q_v\alpha\|^2 |qf'''^{(qhq_v)}(x)-pf'''^{(phq_v)}(x)|\leq e_2(m,l,h),
\end{equation}
where
$$|e_2(m,l,h)|\leq 6\left(10r_{p,q}+2\frac{C'(f)pq^3m^2lq_v^{3}C_{p,q}^{-3}}{q_{v+1}^3}\right)
+\frac{20C'(f)pq(p^4+q^4)hml^2q_v^4C_{p,q}^{-4}}{q^4_{v+1}}.
$$
We will show that \eqref{eq:estsec} and \eqref{eq:estthi} can not both be true. Consider $l,m,h$ for which \\
$\max(p,q,1) \max(l,m,h)\leq C_{p,q}^6$. Then, by \eqref{eq:cpq}, the assumptions of Lemma \ref{cor:bux} are satisfied for $x,y,v$. So, by \eqref{cor:co2}, for some $j(x,v),j(u,v)\in \{A_-,A_+\}$, we have
$$
|f''^{qlq_v}(x)-f''^{plq_v}(y)|\geq
|qlq_v^2j(x,v)B_{v,x}^{-2}-plq_v^2j(y,v)B_{v,y}^{-2}|- 2(C'(f)+1)(|p|+|q|)lq_v^2
$$
and by \eqref{cor:co3}
$$
|qf'''^{qhq_v}(x)-pf'''^{phq_v}(x)|\geq$$$$
|2j(x,v)q^2hq_v^{3}B_{v,x}^{-3}\pm 2j(y,v)p^2hq_v^{3}
B_{v,y}^{-3}|-2(p^2+q^2+1)(C'(f)+1)qhq_v^3.
$$
Therefore in \eqref{eq:estsec}, using \eqref{eq:ry}, we have
\begin{equation}\label{eq:lk}
|qj(x,v)B_{v,x}^{-2}-pj(x,v)B_{v,y}^{-2}|-2(C(f)+1)(|p|+|q|)\leq e_1(m,l)\frac{q_{v+1}^2}{pqmlq_v^2}
\end{equation}
and
\begin{equation}\label{eq:kj}
|2q^2j(x,v)B_{v,x}^{-3}\pm p^2j(y,v)B_{v,y}^{-3}|-2(C'(f)+1)(p^2+q^2)\leq e_2(m,l,h)\frac{q_{v+1}^3}{pqmlhq_v^3}.
\end{equation}
Let  $m=[\frac{Jq_{v+1}}{q_v}], l=[\frac{Kq_{v+1}}{q_v}], h=[\frac{Wq_{v+1}}{q_v}]$. Then, by the definition of $e_1(m,l)$, we have
$$
|e_1(m,l)\frac{q_{v+1}^2}{pqmlq_v^2}|\leq \frac{r_{pq}}{pqJK}+ 2JC'(f)q^2C_{p,q}^{-3}
$$
and
$$
|e_2(m,l,h)\frac{q_{v+1}^3}{pqmlhq_v^3}|\leq 6\left(\frac{10r_{p,q}}{pqJKW}+\frac{2C'(f)q^2JC_{p,q}^{-3}}{W}\right)+
20C'(f)(p^4+q^4)KC_{p,q}^{-4}.
$$
If we now choose $J:=[\frac{C_{p,q}^6}{2C'(f)q^2}]$, $K:=\left[\frac{C_{p,q}^{4}}{20C'(f)(p^4+q^4)}\right]$ and $W:=[C_{p,q}^3]$ and {remember that $r_{p,q}<(pq)^{-1} C_{p,q}^{30}$ then}, by the above estimates and~\eqref{eq:lk},~\eqref{eq:kj}, we have
$$
|qj(x,v)B_{v,x}^{-2}-pj(y,v)B_{v,y}^{-2}|\leq 2(C(f)+1)(|p|+|q|)+10
$$
and
$$
|q^2j(x,v)B_{v,x}^{-3}\pm p^2j(y,v)B_{v,y}^{-3}|<2(C(f)+1)(p^2+q^2)+10.
$$
If in the {latter inequality ``$\pm$'' is in fact ``$+$''}, then
$$|q^2j(x,v)B_{v,x}^{-3}+ p^2j(y,v)B_{v,y}^{-3}|\geq q^2j(x,v)B_{v,x}^{-3}\geq $$$$ 10^{-3}q_2\min(A_-,A_+)c_{p,q}^{-3}>2(C(f)+1)(p^2+q^2)+10,
$$
by the choice of $c_{p,q}>0$ and this gives a contradiction. So we have
$$
|qj(x,v)B_{v,x}^{-2}-pj(y,v)B_{v,y}^{-2}|\leq 2(C(f)+1)(|p|+|q|)+10
$$
and
$$
|q^2j(x,v)B_{v,x}^{-3}- p^2j(y,v)B_{v,y}^{-3}|<2(C(f)+1)(p^2+q^2)+10.
$$
But, by Lemma \ref{lem:cpq} for $U=A_-$ and $V=A_+$, this means (since $B_{v,x}\leq c_{p,q}$) that $\frac{p}{q}\in
\{1,\frac{A_-}{A_+},\frac{A_+}{A_-}\}$  which is a contradiction with the assumptions on $p,q$. This contradiction shows that \eqref{eq:smalldev2} and hence also \eqref{eq:smalldev} holds, which in turn implies \eqref{eq:inp}. This finishes the proof.

\section{Disjointness of time changes of horocycle flows and Arnol'd flows \\ (proof of Theorem~\ref{main:th})}\label{s:Sec9}
In this section we will prove Theorem \ref{main:th}. Since $\tau\in C^1(M)$ is fixed, we denote $(\tilde{h}_t^\tau)$ by $(\tilde{h}_t)$. We will divide the proof {into} several steps. We will start with the following general lemma.  We recall that $d^f$ stands for the product metric.
\begin{lemma}\label{lem:add}For every $0<\epsilon<1/100$, if $d^f((y,s),(y',s'))<\epsilon^3$~%{\footnote{{We recall that $d^f$ stands for the product metric.}}}
 and $t\in \R$
is such that $((R_\alpha)^f)_t(y,s)\in \{(x,u)\in\T^f\;:\; \epsilon^2<u<f(x)-\epsilon^2\}$ and for $w\in\{n(y,s,t),n(y,s,t)+1\}$, we have
$$
|f^{(w)}(y)-f^{(w)}(y')-\tilde{b}(t)|<4\epsilon^{8/3}
$$
for some function $\tilde{b}:\R\to \R$, then
$$
d^f\left((R_\a)^f_t(y,s),(R_\a)^f_{t-\tilde{b}(t)}(y',s')\right)<\epsilon^2.
$$
\end{lemma}
\begin{proof}We will show the proof in case $t>0$; the proof of case $t<0$ follows the same lines. By assumptions
$$
f^{(n(y,s,t)}(y')\leq f^{(n(y,s,t)}(y)-\tilde{b}(t)+4\epsilon^{8/3}\leq t-\epsilon^2+s'-\tilde{b}(t)+|s-s'|+4\epsilon^{8/3}\leq t-\tilde{b}(t)+s'
$$
and
$$
f^{(n(y,s,t)+1)}(y')\geq f^{(n(y,s,t)+1}(y)-\tilde{b}(t)-4\epsilon^{8/3}\geq t+\epsilon^2+s'-\tilde{b}(t)-|s-s'|-4\epsilon^{8/3}\geq t+s'-\tilde{b}(t).
$$
Therefore, $(R_\alpha)^f_{t-\tilde{b}(t)}(y',s')=(R_\alpha^{n(y,s,t)}y', t+s'-\tilde{b}(t)- f^{(n(y,s,t)}(y'))$ and, by definition, $(R_\alpha)^f_{t}(y,s)=(R_\alpha^{n(y,s,t)}y, t+s- f^{(n(y,s,t)}(y))$. This, by our assumptions, finishes the {proof.}
\end{proof}

\noindent \textbf{The set $\mathcal{D'}$.} Let us define the set $\mathcal{D'}$ by
$$
\mathcal{D'}:=\{\alpha\in[0,1]\setminus \Q\;:\; \text{($\cD$1) and ($\cD$3) (of Definition \ref{rotationassumptions}) are satisfied}\}.
$$
For the proof of Theorem~\ref{main:th}, we will show that the assumptions of  Theorem~\ref{disjoint.flows} are satisfied for $(T_t)=(\tilde{h}_t)$ and $(S_t)=((R_\alpha)_t^f)$. To simplify notation, we assume that $\int_\T fd\lambda=1$.

\smallskip
\noindent \textbf{The set $P$. } Let $\bar{A}=\max(|A_--A_+|^{\pm1})$ and define
$$P:=\left\{\frac{1}{32\cdot2018\bar{A}^2},\frac{1}{32\cdot2018\bar{A}^2}\right\}.
$$

\smallskip
\noindent \textbf{Construction of $(X_k)$ and $(A_k)$.}  For $k\in \N$, {set} $X_k:=M$ and $A_kx{:=}v_{1/k}x$ (here $(v_t)$ denotes the opposite horocycle flow). Obviously $A_k\to Id$ uniformly on $M$.

\smallskip
\noindent \textbf{Construction of $(E_k)$.} In the statement of Theorem \ref{disjoint.flows} the sets $(E_k(\epsilon))$ depend on the parameter $\epsilon>0$. However in our case we simply set $E_k(\epsilon)=M$ for every $k\in \N$ and $\epsilon>0$.

Fix $\epsilon>0$ and $N\in \N$. Let $\kappa=\kappa(\epsilon)=\epsilon^{20}$.
By ergodic theorem it follows that there exist $V_\epsilon>0$ and a set $\ov{W}_\epsilon\subset \T^f$, $\lambda^f(\ov{W}_\epsilon)>1-\epsilon^2$ such that for every $(y,s)\in \ov{W}_\epsilon$ and every $|V|>V_\epsilon$, $U\geq \kappa |V|$, we have %{COULD YOU CORRECT TEX BELOW?}
\begin{multline} \label{eq:betnew}
\lambda\left(\left\{t\in [V,V+U]\;:\; ((R_\alpha)^f)_t(y,s)\in \{(x,s)\in \T^f\;:\; \epsilon^2<s<f(x)-\epsilon^2\}\right\}\right) \geq \\ (1-\epsilon^{5/3})U
\end{multline}
and for every $(y,s)\in W_\epsilon$ and $|t|\geq V_\epsilon$, we have
\be\label{eq:birert}
|n(y,s,t)-t|<\epsilon^3|t|.
\ee

We will now define a set $Z_\epsilon\subset \T$. First, let
$$
W_s:=\left\{y\in \T\;:\; y\notin \bigcup_{i=-q_s}^{q_s}R_\alpha^i\left[-\frac{1}{q_s\log^{7/8}q_s}, \frac{1}{q_s\log^{7/8}q_s}\right] \right\}.
$$
and {set} $Z(u){:=}\bigcap_{s\geq u, s\notin K_\alpha}W_s$. Since $\lambda(W_s)\geq 1-\frac{2}{\log^{7/8}q_s}$, by {($\mathcal{D}$1)}, %{IT SHOULD NOT BE (D1)?}
 it follows that $\lim_{u\to +\infty}\lambda(Z(u))=1$. Let $u_\epsilon\in \N$ be such that $Z(u_\epsilon)> 1-\epsilon^2$ and define $Z_\epsilon:=Z(u_\epsilon)$. We then set
$$
Z:=\ov{W}_\epsilon\cap Z_\epsilon^f\subset \T^f.
$$
By the definition of $f$, we have $\lambda^f(Z_\epsilon^f)\geq 1-\epsilon^{5/3}$. Therefore, $\lambda^f(Z)\geq 1-\epsilon$.

Finally, let $\delta=\delta(\epsilon,N):=\min\left(\kappa^{20},V_\epsilon^{-4}, \frac{1}{q_{s'}\log q_{s'}}, \bar{\delta}(\epsilon^2),N_{\epsilon^3}^{-3}, n_0(\epsilon^2)^2\right)$, where $s':=\kappa^{-1}\max(N,u_\eps^2)$ and $\bar{\delta}(\epsilon^2),N_{\epsilon^3}^{-3}$ come from Proposition~\ref{thm:hor} (with $K=1$) and $n_0(\epsilon^2)$ comes from Lemma \ref{lem:bsums}.
We now take $x,x':=A_kx$ with $d_M(A_k,Id)\leq \delta$ and $(y,s),(y',s')\in Z$ with $d^f((y,s),(y',s'))<\delta$ and we will show that \eqref{forw} or \eqref{backw} holds for $x,x'$ and $(y,s),(y',s')$.

Let $v\in \N$ be unique such that
\be\label{eq:distv}
\frac{1}{q_{v+1}\log q_{v+1}}\leq \|y-y'\|\leq \frac{1}{q_v\log q_v},
\ee
and let $\ell\in \R$, $\ell\leq q_{v+1}$ be unique such that
\be\label{eq:distexact}
\|y-y'\|=\frac{1}{\ell\log\ell}.
\ee
By the definition of $Z$ (in particular by the definition of $W_v$), it follows that the assumption of Lemma~\ref{lem:forback} are satisfied for $s=v$ and $y,y'\in \T$. If~\eqref{eq:forwsing} holds, we will show~\eqref{forw} and if~\eqref{eq:backsing} holds we will show~\eqref{backw}. Since the proofs of both cases are analogous, we will assume that~\eqref{eq:forwsing} is true for $y,y'$.
 Let
 \be\label{eq:timeell}
 \tilde{T}:=\frac{1}{8}\min(k^{1/2}, \ell).
 \ee
By Proposition~\ref{thm:hor} (for $K=1$ and $\epsilon^2$) and \eqref{eq:timeell}, it follows that for every $t\in [N_\epsilon, \tilde{T}]$, we have that \eqref{distgt} holds for $x,x'$ (with $\epsilon^2$ instead $\epsilon$). Since $x'=v_{k^{-1}}x$, by the definition of $\chi(\cdot)=\chi_{x,x'}(\cdot)$, it follows that $\chi(t)=t-k^{-1}t^2$ (cf. \eqref{defchi}) and moreover, by \eqref{axt}, we get that $|A_x(t)|\leq \epsilon^2$ (since $s=0$ for $x$ and $x'$) for $t\in [N_{\epsilon},\tilde{T}]$. Therefore, by \eqref{distgt}, we have
\be\label{dist:klm}
d_M(\tilde{h}_tx,\tilde{h}_{t-k^{-1}t^2}x')<2\epsilon^2\;\text{ for }\; t\in [N_\epsilon,\tilde{T}].
\ee

Notice that by \eqref{eq:forwsing} and \eqref{eq:timeell} (recall that $\ell\leq q_{v+1}$) it follows that for every $n\in [0,\tilde{T}]$, by the mean value theorem, we have
\be \label{eq:fny'}
f^{(n)}(y)-f^{(n)}(y')=f'^{(n)}(\theta_n)(y-y'),
\ee
where $\theta_n\in[y,y']$. Moreover, since $\theta_n\in[y,y']$, it follows  by \eqref{eq:forwsing} that $\theta_n\notin \Sigma_v(1/4)$ (see \eqref{eq:sigma}) and hence we can use Lemma \ref{lem:bsums}. Therefore, for every $n\in [\epsilon^4q_v,\tilde{T}]$, by \eqref{eq:fny'}, \eqref{eq:distexact} and \eqref{eq:dercont}, we have
\be \label{eq:nvc1}
\left|f^{(n)}(y)-f^{(n)}(y')-(A_--A_+)\frac{n\log n}{\ell \log \ell}\right|={\rm O}(\epsilon^2).
\ee
Similarly, by \eqref{eq:fny'}, \eqref{eq:distexact} and \eqref{eq:smallder} for every $n\in [0,\epsilon^4q_v]$, we have
\be \label{eq:nvc2}
|f^{(n)}(y)-f^{(n)}(y')|={\rm O}(\epsilon^2).
\ee
Notice that by \eqref{eq:distv} and \eqref{eq:distexact}, for $n\in [0,\epsilon^4q_v]$, we have
$|(A_--A_+)\frac{n\log n}{\ell \log \ell}|={\rm O}(\epsilon^4)$. Therefore and by \eqref{eq:nvc1} and \eqref{eq:nvc2} for every $n\in [0,\tilde{T}]$, we have
\be \label{eq:nvc3}
\left|f^{(n)}(y)-f^{(n)}(y')-(A_--A_+)\frac{n\log n}{\ell \log \ell}\right|= {\rm O}(\epsilon^{2}).
\ee
Moreover, since $(y,s)\in Z^f$, by \eqref{eq:birert}, for $t\leq (1-\epsilon^2)\tilde{T}$, we have $$n(y,s,t)\leq \tilde{T}-1$$
and using \eqref{eq:birert} again, we have
$$
|(A_--A_+)\frac{n(y,s,t)\log n(y,s,t)}{\ell \log \ell}-(A_--A_+)\frac{t\log t}{\ell \log \ell}|\leq\\
\frac{|A_--A_+|}{\ell\log \ell}2\epsilon^{3}t\log t ={\rm O}( \epsilon^3),
$$
the last inequality by \eqref{eq:timeell}.
The two above inequalities and
\eqref{eq:nvc3} imply that for every $t\in [0,(1-\epsilon^2)\tilde{T}]$, for $w_t\in\{n(y,s,t),n(y,s,t)+1\}$, {we have}
\begin{equation}\label{eq:newnev}
|f^{(w_t)}(y)-f^{(w_t}(y')-\tilde{b}(t)|\leq 2\epsilon^{8/3},
\end{equation}
where $\tilde{b}(t)=(A_--A_+)\frac{t\log t}{\ell \log \ell}$.\footnote{In what follows we will often use the following property of $\tilde{b}(\cdot)$: for an interval $[R,S]\subset [0,\tilde{T}]$ such that $S\leq R^{1+\epsilon^5}$ and every $t\in [R,S]$, we have
$$|\tilde{b}(t)-(A_--A_+)\frac{t\log R}{\ell \log \ell}|\leq \epsilon^5 |A_--A_+|\frac{t\log R}{\ell \log \ell}|\leq \epsilon^3,
$$
since $t,R\leq \tilde{T}$. This means that for such $[R,S]$, the function $\tilde{b}(\cdot)$ is almost linear with coefficient $(A_--A_+)\frac{\log R}{\ell \log \ell}$.}

{But} Lemma~\ref{lem:add} (cf.\ \eqref{eq:newnev}) implies that for every $t\in [0,\tilde{T}]$
for which $(R_\a)^f_t(y,s)\in \{(y,s)\in\T^f:\; \epsilon^2<s<f(y)-\epsilon^2\}$, we have (since $d^f((y,s),(y',s'))<\delta<\epsilon^3$)
\be \label{eq:conttime}
d^f((R_\a)^f_t(y,s),(R_\a)^f_{t-\tilde{b}(t)}(y',s'))<\epsilon^2.
\ee

Let $\tilde{c}(t):=-k^{-1}t^2+\tilde{b}(t)$.
We {claim} the following:
\be\label{eq:stdrift}
|\tilde{c}(T_0)|> \frac{1}{32\cdot2018\bar{A}^2},\;\text{ for some } \;T_0\in \left[\frac{\tilde{T}}{8\bar{A}},\frac{\tilde{T}}{2\bar{A}}\right].
\ee
Indeed, denote $K_0:=\frac{\tilde{T}}{8\bar{A}}$.
{Then,} notice that by~\eqref{eq:timeell}, we have
$$
|\tilde{c}(4K_0)-16\tilde{c}(K_0)
+12(A_--A_+)\frac{K_0\log K_0}{\ell \log \ell}|=|A_--A_+|\frac{4K_0\log4 K_0-4K_0\log K_0}{\ell \log \ell}<\epsilon^2,
$$
{and}
$$
|\tilde{c}(4K_0)-4\tilde{c}(K_0)-12k^{-1}K_0^{2}|<\epsilon^2.
$$
{Furthermore,} by the definition of $K_0$ and \eqref{eq:timeell}, it follows that
$$
\max(12k^{-1}K_0^{2}, 12|A_--A_+|\frac{K_0\log K_0}{\ell \log \ell})\geq \frac{1}{1000\bar{A}^2}.
$$
The three above equations imply that $\max(|\tilde{c}(4K_0)|,|\tilde{c}(K_0)|)>\frac{1}{32\cdot2018\bar{A^2}}$. This finishes the proof {of the claim}~\eqref{eq:stdrift}.

Let $u:=\frac{1}{32\cdot2018\bar{A}^2}$. Notice that by \eqref{eq:timeell} and the definition of $u$, we have
$$\tilde{c}(u^2\tilde{T})\leq k^{-1}\left(u^2\tilde{T}\right)^2+
|A_--A_+|\frac{u^2\tilde{T}\log u^2 \tilde{T}}{\ell \log \ell}\leq u.
$$
Therefore and since $\tilde{c}(\cdot)$ is continuous,  by \eqref{eq:stdrift}, we have $|\tilde c(T_1)|=u$ for some $T_1\in [u^2\tilde{T}, \frac{\tilde{T}}{2\bar{A}}]$.
Moreover, for every $t\in [T_1,(1+\kappa)T_1]$, we have (since $T_1\leq \tilde{T}$ and by \eqref{eq:timeell})
\begin{multline}\label{eq:tightdrift}
|\tilde{c}(t)-\tilde{c}(T_1)|\leq k^{-1}(t^2-T_1^2)+ \frac{|A_--A_+|}{\ell \log \ell}(t\log t- T_1\log T_1)\leq\\
 k^{-1}(2\kappa \tilde{T}+\kappa^2\tilde{T}^2)+ \frac{2\kappa|A_--A_+|\tilde{T}\log\tilde{T}}{\ell \log \ell} \leq \epsilon^3,
\end{multline}
{where the last inequality follows} by the definition of $\kappa$.

We set $M:=T_1$ and $L:=\kappa T_1$. Let
\be\label{eq:defuu}
U:=
\left\{t\in[M,M+L]\;:\;  (R_\a)^f_t(y,s)\in \{(y,s)\in\T^f:\; \epsilon^2<s<f(y)-\epsilon^2\}\right\}.
\ee

By \eqref{eq:betnew}, it follows that $|U|\geq (1-\epsilon^{3/2})|L|$. Moreover, by \eqref{eq:defuu}, we have $U=\bigcup_{i=1}^w(c_i,d_i)$, where $|d_i-c_i|\geq \inf_\T f-2\epsilon^2$ for every $i=2,\ldots w-1$ and hence $w\leq \frac{|L|}{\bar{d}}$, where $\bar{d}=\frac{\inf_\T f}{2}$. Let
$a(t):=t-\tilde{b}(c_i)$ on $(c_i,d_i)$, $i=1,\ldots w$ and we set $a(t)=0$ for $t\in [M,M+L]\setminus U$. Notice that by~\eqref{eq:timeell} (since $c_1\leq \tilde{T}$), {we have}
\be\label{eq:evgo}
|\tilde{b}(c_1)|=|A_--A_+|\frac{c_1\log c_1}{\ell\log \ell}\leq |A_--A_+|\leq \epsilon^2 L,
\ee
since $\epsilon^2L=\epsilon^2\kappa T_1\geq \kappa^3\tilde{T}\geq \kappa^4 \delta^{-1/2}>\kappa^{-1}>|A_--A_+|$.
Moreover,  by \eqref{eq:forwsing}, for $i=1,\ldots, w$, $|c_{i+1}-c_i|\leq |f^{(n(y,s,c_i))}(y)|+2\epsilon^2={\rm O}(\log q_v)<\epsilon^3 q_v<\epsilon^3 \ell)$ and hence, we have (since $c_{i+1}\leq \tilde{T}\leq \ell$)
\be\label{eq:evgo2}
|\tilde{b}(c_{i+1})-\tilde{b}(c_i)|\leq 2|A_--A_+|\frac{(c_{i+1}-c_i)\log c_{i+1}}{\ell \log \ell}\leq \epsilon^3
\ee
since $|c_{i+1}-c_i|\leq \epsilon^3 \ell$ and $c_{i+1}\leq \tilde{T}\leq \ell$ (see \eqref{eq:timeell})
for $i=1,\ldots ,w$.
Therefore, by \eqref{eq:evgo},~\eqref{eq:evgo2} and the definition of $U$, $(a,U,\bar{d})$ is $\epsilon^{3/2}$-good (see (PAL) in Definition \ref{agood}). Moreover, for $t\in (c_i,d_i)$, $i=1,\ldots w$, by \eqref{eq:evgo2}, we have
\be\label{eq:attil}
|a(t)-(t-\tilde{b}(t))|=|-\tilde{b}(c_i)+\tilde{b}(t)|\leq |\tilde{b}(c_{i+1})-\tilde{b}(c_{i})|\leq \epsilon^3.
\ee

By \eqref{eq:conttime}, \eqref{eq:defuu} and \eqref{eq:attil}, since $[M,M+L]\subset [0,\tilde{T}]$, we have
\be \label{eq:fin1}
d^f((R_\a)^f_t((y,s)),(R_\a)^f_{a(t)}((y',s'))<\epsilon\;\text{ for every } t\in U\cap [M,M+L].
\ee

Furthermore, by \eqref{eq:attil}, we have $t-k^{-1}t^2:=(t-\tilde{b}(t))+\tilde{c}(t)=a(t)+\tilde{c}(t)+{\rm O}(\epsilon^3)$. Since $\tilde{c}(M)=\pm u=\pm\frac{1}{32\cdot2018\bar{A}^2}\in P$, by \eqref{dist:klm} and \eqref{eq:tightdrift}, we have
\be \label{eq:fin2}
d_{G/\Gamma}(\tilde{h}_tx,\tilde{h}_{a(t)+p}x')<\epsilon\;\text{ for every } t\in [M,M+L].
\ee
By \eqref{eq:fin1} and \eqref{eq:fin2}, it follows that \eqref{forw} holds and this, by Theorem \ref{disjoint.flows}, finishes the proof of Theorem \ref{main:th}.

\section{Time changes of horocycle flows and Sarnak's conjecture \\ (answer to M. Ratner's question)}\label{s:Sarnak}
We now turn our attention to the problem of M\"obius disjointness of time automorphisms of flows considered in the present paper.
Recall that the M\"obius function %(whose definition is recalled in the introduction).
$\mob:\N\to\{-1,0,1\}$ is defined as $\mob(1)=1$, $\mob(p_1\ldots p_k)=(-1)^k$ for pairwise different prime numbers $p_1,\ldots, p_k$ and $\mob(n)=0$ for the remaining $n\in\N$.
It is not {hard} to see that $\mob(mn)=\mob(m)\mob(n)$ whenever $m$ and $n$ are coprime, {i.e.\ $\mob$ is an example of arithmetic {\em multiplicative} function}. {Hence}, $\mob$ is a member of
$$
{\cal M}_1:=\{\bfu:\N\to\C:\:\bfu\text{ is multiplicative and }|\bfu|\leq1\}.$$
A basic method to prove disjointness of a bounded numeric sequence $(a_n)\subset\C$ with all members of ${\cal M}_1$  (in fact, with all bounded multiplicative functions) is the following criterion:

\begin{proposition}[\cite{Ka},\cite{Bo-Sa-Zi}] If for all sufficiently large prime numbers $p\neq q$ we have
\be\label{ortog0}
\lim_{N\to\infty}\frac1N\sum_{n\leq N}a_{pn}\overline{a}_{qn}=0,\ee
then
\be\label{orthog}\lim_{N\to\infty}\frac1N\sum_{n\leq N}a_n\bfu(n)=0\ee
for each $\bfu\in{\cal M}_1$.\end{proposition}

{Given a topological dynamical system $(X,T)$, where $T$ is a homeomorphism of a compact metric space $X$, we} want to study {\em disjointness} of $(X,T)$ with $\bfu\in {\cal M}_1$, {meaning that} $\lim_{N\to\infty}\frac1N\sum_{n\leq N}f(T^nx)\bfu(n)=0$ for each $f\in C(X)$ and $x\in X$. As constant functions are continuous, to have such a disjointness, we need to assume that, additionally to ${\bfu}\in{\cal M}_1$,  the mean of $\bfu$ exists and equals zero:
\be\label{mobius21}
M(\bfu):=\lim_{N\to \infty}\frac1N\sum_{n\leq N}\bfu(n)=0\ee
($\mob$ has mean zero, by the Prime Number Theorem).

{In what follows, we consider the problem of disjointness of $(X,T)$ only with the members in ${\cal M}_1$ satisfying~\eqref{mobius21}.}

Now, assume {additionally that} $(X,T)$ is a uniquely ergodic topological dynamical system, with {the} unique $T$-invariant measure $\mu$ which makes $(X,\mu,T)$ totally ergodic. If we fix $x\in X$ then each accumulation point $\rho$ of the empiric measures $\frac1N\sum_{n\leq N}\delta_{(T^p\times T^q)^n(x,x)}$, $N\geq1$,
yields a member of $J(T^p,T^q)$ (indeed, $T^p$ and $T^q$ are also uniquely ergodic, {and $\mu$ is their only invariant measure}). If we take a zero mean $f\in C(X)$ and $T^p\perp T^q$ (for $p\neq q$ sufficiently large), then immediately, {$T^p\times T^q$ is still uniquely ergodic (with $\mu\otimes\mu$ being the unique invariant measure)} and we find that the sequence $(a_n)$, $a_n=f(T^nx)$, satisfies~\eqref{ortog0} (indeed, the limit equals {$\int_{X\times X}f\otimes\overline{f}\,d(\mu\otimes\mu)=0$}). Hence, cf.\ \cite{Bo-Sa-Zi},  if the different, {sufficiently large, prime powers of $T$ are disjoint then we find that the system $(X,T)$ is orthogonal to any $\bfu\in{\cal M}_1$ satisfying~\eqref{mobius21}}.

{Notice that the flows $(T_t)$ we are considering are mixing. Hence, all non-zero time automorphisms $T_{t_0}$ are totally ergodic. When considering any uniquely ergodic model of $T_{t_0}$ (in fact, those considered in Theorem~\ref{main:prop} are themselves uniquely ergodic),  they satisfy the disjointness assumption:  $T_{pt_0}\perp T_{qt_0}$ for $p\neq q$ prime numbers sufficiently large, whence the assumption~\eqref{ortog0} is satisfied for each $x$ and zero mean $f$ (in the model) and we obtain disjointness with $\bfu$. {As a matter of fact, the result holds also for $t_0=0$ as the mean of $\bfu$ equals~0}.}

On the other hand, to distinguish between M\"obius disjointness of zero entropy systems and positive entropy systems\footnote{As proved by Downarowicz and Serafin \cite{Do-Se}, there are  positive entropy systems which are M\"obius disjoint. {On the other hand no positive entropy system satisfies the strong $\mob$-MOMO property \cite{Ab-Ku-Le-Ru}.}}, it is proved in \cite{Ab-Ku-Le-Ru} that the following conditions are equivalent (for the definition of strong $\bfu$-MOMO, {see  Appendix}~\ref{s:Sec10}):

\begin{proposition}\label{p:equiv}The following conditions are equivalent:\\
(i) Sarnak's conjecture holds.\\
(ii)  Sarnak's conjecture holds  uniformly (in $x\in X$).\\
(iii) Strong $\mob$-MOMO property holds for each zero (topological)  entropy system.\end{proposition}

Motivated by some recent break-through results in multiplicative number theory \cite{Ma-Ra}, \cite{Ma-Ra-Ta} and the role of convergence on   (typical) short intervals (see \cite{Fe-Ku-Le} for more details), {we}  will now consider a fourth natural condition of orthogonality
which is a uniform  short interval convergence (USIC {for short}) in $(X,T)$:
\begin{definition}\label{def2} \em $(X,T)$ satisfies the {\em strong $\bfu$-USIC\footnote{Acronym for Uniform Short Intervals Convergence. Note that the strong $\bfu$-USIC property requires for $\bfu$ to satisfy convergence on (typical) short intervals which we obtain by taking $f=1$ in Definition~\ref{def2}. In view of \cite{Ma-Ra}, $\mob$ fulfills this requirement. In what follows, we consider only $\bfu$ satisfying the relevant short interval convergence.} property} if
$$
\frac1M\sum_{M\leq m<2M}\left|\frac1H\sum_{m\leq h<m+H}f(T^hx)\bfu(h)\right|\to0$$
when $H,M\to\infty$, $H={\rm o}(M)$, uniformly in $x\in X$.
\end{definition}

The meaning of convergence in Definition~\ref{def2} is the following: For all sequences $(M_\ell),(H_\ell)$ tending to infinity with $H_\ell/M_\ell\to0$ when $\ell\to\infty$, we have
\be\label{e1}
\lim_{\ell\to0}\frac1{M_\ell}\sum_{M_{\ell}\leq m<2M_{\ell}}\left|\frac1{H_{\ell}}\sum_{m\leq h<m+H_{\ell}}f(T^hx)\bfu(h)\right|=0\ee
uniformly in $x\in X$.

Note also that by considering $M, 2M,\ldots, 2^kM,\ldots$, both in Definition~\ref{def2} and in~\eqref{e1} we can consider ({and} we will) the sum from~1 to~$M$ instead of the sum from~$M$ to~$2M$.

{In  Appendix} \ref{s:Sec10} we show that a system $(X,T)$ satisfies the strong $\bfu$-MOMO property if and only if it satisfies the strong $\bfu$-USIC property.

\begin{corollary}\label{c:equiv21} The following condition\\
(iv) $\mob$-USIC property holds for each zero (topological) entropy system\\
is equivalent to (i) in Proposition~\ref{p:equiv} (and hence to (ii) and (iii)).\end{corollary}
Question
It is proved in \cite{Ab-Ku-Le-Ru} that the systems satisfying so called AOP property, in particular systems whose prime powers are disjoint, satisfy the strong $\bfu$-MOMO property. Collecting all the results, we obtain the following.

\begin{corollary}\label{shorti} For each flow $(T_t)$ being either a uniquely ergodic model of an Arnol'd flow satisfying the assumptions of Theorem~\ref{main:th2}  or a non-trivial smooth  time-change in $B^+(M)$  of a horocycle flow in the cocompact case,  we obtain the following: For each $t_0\in\R$ and each $f\in C(X)$, we have
$$
\frac1M\sum_{1\leq m<M}\left|\frac1H\sum_{m\leq h<m+H}f(T_{ht_0}x)\bfu(h)\right|\to0$$
when $H,M\to\infty$, $H={\rm o}(M)$, uniformly in $x\in X$.

In particular, {for such flows, Sarnak's conjecture holds uniformly}.\end{corollary}

Corollary~\ref{shorti} brings the positive answer to M.\ Ratner's question (see ~7 in \cite{Fe-Ku-Le}). Although, M\"obius disjointness itself is known for horocycle flows \cite{Bo-Sa-Zi}, it remains however
an open question whether the assertions of Corollary~\ref{shorti} hold for horocycle flows themselves even when $\bfu=\mob$.

%\begin{remark} \em Using the same approach as above, we can show %that locally Hamiltonian flows on $\T^2$ having one critical point %which is a saddle connection are M\"obius dsjoint, however, the %method we apply seems to be not sufficient to show the validity of %Sarnak's conjecture in its uniform form.\end{remark}

\begin{remark} \em If we have a dynamical system $(X,T)$ for which
the sums $$(\ast)\;\;\frac1N\sum_{n\leq N}h(T^nx)\mob(n)$$ converge to zero uniformly in $x\in X$ (for each $h\in C(X)$) then one can show (considering respectively the sums $(\ast)$ with $T^{-N}x$ and $T^{-1}$ or with $T^Nx$ and $T$) that we have also convergence for sums of the following type:
\be\label{newc1}
\lim_{N\to\infty}\frac1N\sum_{n\leq N}h(T^nx)\mob(N-n)=0,\ee
\be\label{newc2}
\lim_{N\to\infty}\frac1N\sum_{n\leq N}h(T^nx)\mob(N+n)=0. \ee
We do not know if \eqref{newc1}, \eqref{newc2} hold in case of horocycle flows or locally Hamiltonian flows on $\T^2$.\end{remark}

\appendix

\section{Consequences of shearing for time changes of horocycle flows}\label{app:horocycleR}
This appendix is devoted to the proof of Proposition~\ref{thm:hor}. We use the notation introduced in {Section}~\ref{sec:prelimh}
\begin{proof}[Proof of Proposition~\ref{thm:hor}] We will for simplicity assume that $\int_M \tau d\mu=1$. Fix $0<\epsilon<K^{-3}$. Using the unique ergodicity of $(h_t)$, let
$N_\epsilon>0$ be such that for every $T\in \R$ satisfying $|T|>N_\epsilon$ and every $y\in M$, remembering that $\int_MX\tau\,d\mu=0$, we have
\begin{equation}\label{bet}
\max\left(\left|\int_0^T\tau(h_ty)\,dt-T\right|,\left|\int_0^T(X\tau)
(h_ty)\,dt\right|
\right)\leq \epsilon^3 |T|
\end{equation}
and (cf.\ Remark~\ref{r:ue})
\begin{equation}\label{bet2}
|u(T,y)-T|\leq \epsilon^4 |T|.
\end{equation}
Moreover, let $\delta:=\min(\epsilon^{20},N_\epsilon^{-20})$. Take $x,y\in M$ with $\max(|r|,|s|,|\bar v|)<\delta$ and let $T\in \R$ be such that $|T|\in [N_\epsilon,K|r|^{-1/2}]$. We will additionally assume that $T>0$, the case $T<0$ is analogous.

We have
\begin{equation}\label{clos2}
h_Txy^{-1}h_{-\chi_{x,y}(T)}=h_{\bar v }\begin{pmatrix} e^s+rT& v(x,y,T)\\
r& e^{-s}-r\chi_{x,y}(T)
\end{pmatrix},
\end{equation}
where (using also~\eqref{defchi})
$$
v(x,y,T)=Te^{-s}-\chi_{x,y}(T)e^s-T\chi_{x,y}(T)r=e^{-3s}r^2T^3.$$
By the definition of $\delta$, we hence obtain (remembering that $|r|<\delta$)
\be\label{eq:newform}
|v(x,y,T)|= e^{-3s}|r|^2T^3\leq 2|r|^2K|r|^{-3/2}\leq 2K \delta^{1/2}<\epsilon^2.
\ee
 In what follows (since $x,y$ are now fixed), we will write $\chi$ instead of $\chi_{x,y}$.

By~\eqref{zamcz}) and~\eqref{funce}, we obtain
$$
A_x(T)=\int_{u(\chi(T),y)}^{u(\chi(T)+A_x(T),y)}\tau(h_\theta y)\,d\theta=
\int_{u(\chi(T),y)}^{\chi(u(T,x))}\tau(h_\theta y)\,d\theta.
$$
Representing the last term as the difference of two integrals, where $\int_0^{u(\chi(T),y)}\tau(h_\theta y)d\theta=\chi(T)$, the assertion~\eqref{axt} can be rewritten as
\begin{equation}\label{eq:gh}
\left|\int_0^{\chi(u(T,x))}\tau(h_ty)\,dt-\chi(T)
+e^{-2s}\int_0^{u(T,x)}\left(\tau-\tau\circ g_{-s}\right)(h_tx)\,dt\right|\leq \epsilon.
\end{equation}
Changing variables ($t$ replaced by $\chi(\theta)$), we have
\begin{multline}\label{eq:splint}
\int_0^{\chi(u(T,x))}\tau(h_ty)\,dt-\chi(T)= \int_0^{u(T,x)}\tau(h_{\chi(t)}y)\chi'(t)\,dt-\chi(T)=\\
e^{-2s}\int_0^{u(T,x)}\tau(h_{t}x)\,dt+
e^{-2s}\left(\int_0^{u(T,x)}(-\tau(h_tx)+\tau(h_{\chi(t)}y)\,dt)\right)
+\\
\int_0^{u(T,x)}\tau(h_{\chi(t)}y)(\chi'(t)-e^{-2s})\,dt-\chi(T).
\end{multline}

{We now claim} that
\begin{equation}\label{mult:1}
\left|\int_0^{u(T,x)}(-\tau(h_tx)+\tau(h_{\chi(t)}y))\,dt+
\int_0^{u(T,x)}\left(\tau-\tau\circ g_{-s}\right)(h_tx)\,dt\right|\leq \epsilon^{3/2}.
\end{equation}
Indeed, this claim is equivalent to showing
\be\label{eq:toshow}
\left|\int_{0}^{u(T,x)}(\tau(h_{\chi(t)}y)-\tau(g_{-s}h_tx))\,dt\right|\leq \epsilon^{3/2}.
\ee
Notice that by \eqref{clos2}, we have %{COULD YOU CORRECT TEX BELOW?}
\be\label{matrix:dec}\begin{array}{c}
g_{-s}h_txy^{-1}h_{-\chi(t)}=h_{\bar ve^{-2s}}\begin{pmatrix}1+e^{-s}rt& e^{-s}v(x,y,t)\\
e^sr& 1-e^sr\chi(t)
\end{pmatrix} =\\ h_{w(t)+\bar ve^{-2s}}\begin{pmatrix}(1-e^{s}r\chi(t))^{-1}& 0\\
e^sr& 1-e^{s}r\chi(t)
\end{pmatrix},\end{array}
\ee
where $w(t)=\frac{e^{-s}v(x,y,t)}{1-e^{s}r\chi(t)}$ (recall that $v(x,y,t)=e^{-3s}r^2t^3$). We also have $v'(x,y,t)= 3 e^{-3s}r^2t^2\leq C' |r|$ for some constant $C'>0$ (since $t\leq K|r|^{-1/2}$). Moreover, $|r\chi(t)|={\rm O}(|r|t+r^2t^2)={\rm O}(r^{1/2})$ and
$$
|v(x,y,t)r\chi'(t)|={\rm O}(r^2t^3|r|(1+|r|t))={\rm O}(|r|^{3/2}).
$$
Therefore (enlarging $C'$ if necessary), we have
\be\label{small:w'}
|w'(t)|\leq 2C'|r|.
\ee
Now, \eqref{eq:toshow} (in particular) follows by showing that for every $S\in[0,T]$ and $\Psi\in \{\tau,X\tau\}$, we have
\be\label{shoint1}
\left|\int_{0}^{u(S,x)}(\Psi(h_{-w(t)-\bar{v}e^{-2s}}
g_{-s}h_tx)-\Psi(g_{-s}h_tx))\,dt\right|<
\frac{\epsilon^{3}}{2}
\ee
and
\be\label{shoint2}
\left|\int_{0}^{u(T,x)}(\tau(h_{\chi(t)}y)-\tau(h_{-w(t)-\bar{v}e^{-2s}}
g_{-s}h_tx))\,dt\right|<
\frac{\epsilon^{3/2}}{2}.
\ee
We will first show \eqref{shoint1}. Using $g_{-s}h_t=h_{e^{-2s}t}g_{-s}$, substituting $t'=e^{-2s}t$ and making use of~\eqref{small:w'}, we get
$$
\int_{0}^{u(S,x)}\Psi(h_{-w(t)-\bar{v}e^{-2s}}
g_{-s}h_t)\,dt=e^{2s}\int_{0}^{u(e^{-2s}S,x)}
\Psi(h_{t'-w(e^{2s}t')-\bar{v}e^{-2s}}
g_{-s}x)dt'=$$$$
e^{2s}\int_{0}^{u(e^{-2s}S,x)}
\Psi(h_{t-w(e^{2s}t)-\bar{v}e^{-2s}}
g_{-s}x)(1-e^{2s}w'(e^{2s}t))\,dt+{\rm O}(|r|S).
$$
Substituting $t'=t-w(e^{2s}t)-\bar{v}e^{-2s}$ and using $\max(w(0)+\bar{v}e^{-2s}, w(e^{2s}u(e^{-2s}S,x))-\bar{v}e^{-2s})<\delta^{1/4}$, we get
$$
e^{2s}\int_{0}^{u(e^{-2s}S,x)}
\Psi(h_{t-w(e^{2s}t)-\bar{v}e^{-2s}}
g_{-s}x)(1-e^{2s}w'(e^{2s}t))\,dt=$$$$e^{2s}\int_0^{u(e^{-2s}S,x)}
\Psi(h_{t'}g_{-s}x)\,dt'+{\rm O}(\delta^{1/4})=
\int_{0}^{u(S,x)}\Psi(g_{-s}h_tx)\,dt+{\rm O}(\delta^{1/4}),
$$
where the latter equality follows by substituting $t=e^{2s}t'$ and using the renormalization identity. This finishes the proof of \eqref{shoint1}.

We will now show \eqref{shoint2}. To begin the proof of {that} claim, note that by~\eqref{matrix:dec}, we have
$$
h_{-w(t)-\bar ve^{-2s}}g_{-s}h_tx=\begin{pmatrix}(1-e^{s}r\chi(t))^{-1}& 0\\
e^sr& 1-e^{s}r\chi(t)
\end{pmatrix}h_{\chi(t)}y,$$
whence $d_U(h_{-w(t)-\bar ve^{-2s}}g_{-s}h_tx,h_{\chi(t)}y)=0$.
Now, by Lemma~\ref{taylor:form}, we get that the LHS of \eqref{shoint2} is equal to
$$
\int_0^{u(T,x)}\sum_{W\in\{X,V\}}(W\tau)(h_{-w(t)-\bar{v}e^{-2s}}g_{-s}h_tx)
d_W(h_{-w(t)-\bar{v}e^{-2s}}g_{-s}h_tx,h_{\chi(t)}y)\,dt+$$$$
{\rm O}\left(\epsilon^3u(T,x)\sup_{t\in[0,u(T,x)]}\left|\sum_{W\in\{X,V\}}
d_W(h_{-w(t)-\bar{v}e^{-2s}}g_{-s}h_tx,h_{\chi(t)}y)\right|\right).
$$
Moreover, by \eqref{matrix:dec}, Remark \ref{r:odl} and \eqref{bet2}, the above expression is equal to
$$
\int_0^{u(T,x)}(X\tau)(h_{-w(t)-\bar{v}e^{-2s}}g_{-s}h_tx)e^{s}r\chi(t) dt+
\int_0^{u(T,x)}(V\tau)(h_{-w(t)-\bar{v}e^{-2s}}g_{-s}h_tx)e^sr dt+
{\rm O}(\epsilon^3T^2|r|).
$$
The middle summand we estimate by ${\rm O}(e^s|r|u(x,T))$ which by \eqref{bet2} and the fact that $|r|<\delta^{1/2}<\epsilon^3$ is ${\rm O}(\epsilon^3)$ and since $T\leq K|r|^{-1/2}$, the last expression is equal to
$$
re^{s} \int_0^{u(T,x)}\chi(t)(X\tau)(h_{-w(t)-\bar{v}e^{-2s}}g_{-s}h_tx)\,dt +{\rm O}(\epsilon^3).
$$
Since $re^s\chi(t)=e^{-s}rt+e^{-2s}r^2t^2=e^{-s}rt+{\rm O}(r)$, and $u(T,x)\leq 2T\leq2Kr^{-1/2}$ (see \eqref{bet2}), in order to show \eqref{shoint2}, it is enough to establish
\be\label{shosho}
\left|\int_0^{u(T,x)}t(X\tau)(h_{-w(t)-\bar{v}e^{-2s}}g_{-s}h_tx)\,dt
\right|=
{\rm O}(\epsilon^2|r|^{-1}).
\ee
Integration by parts yields
$$
\int_0^{u(T,x)}t(X\tau)(h_{-w(t)-\bar{v}e^{-2s}}g_{-s}h_tx)\,dt=$$$$
T\int_0^{u(T,x)}(X\tau)(h_{-w(t)-\bar{v}e^{-2s}}g_{-s}h_tx)\,dt-
\int_0^{u(T,x)}\left(
\int_0^t(X\tau)(h_{-w(\theta)-\bar{v}e^{-2s}}g_{-s}h_\theta x)\,d\theta\right)\,dt.
$$
Therefore and by \eqref{shoint1} (for $\Psi=X\tau$ and $S=u(T,x)$ and then $S=t$ {)},~\eqref{shosho} follows by showing {that}
$$
\left|\int_0^{t}(X\tau)(g_{-s}h_\theta x)\,d\theta\right|={\rm O}(\epsilon^3|r|^{-1/2}),~\footnote{Indeed, the first expression is then multiplied by $T$ and the second by evaluation {of} the integrand yields multiplication by $u(T,x)$, in both cases, we obtain ${\rm O}(\epsilon^3|r|^{-1/2}T)={\rm O}(\epsilon^3|r|^{-1})$.}
$$
this however, by the renormalization equation and substituting $t'=e^{-2s}\theta$, follows by showing {that}
$$
\left|\int_0^{e^{-2s}t}(X\tau)(h_{t'}g_{-s}x)\,dt'\right|=
{\rm O}(\epsilon^3|r|^{-1/2}),
$$
which is true by \eqref{bet} if $e^{-2s}t>N_\epsilon$ (if not, we estimate ${\rm O}(N_\epsilon)={\rm O}(\epsilon^3 |r|^{-1/2})$). This finishes the proof of \eqref{shosho} and hence also the proof of \eqref{shoint2}. {Now, the proof of~\eqref{eq:toshow} is complete and the claim~\eqref{mult:1} follows.}

Next, by definition,  $e^{-2s}\int_0^{u(T,x)}\tau(h_{t}x)dt-\chi(T)=e^{-3s}rT^2$, so to finish the proof of \eqref{eq:gh}, by \eqref{eq:splint} and \eqref{mult:1}, it is enough to show that
\begin{equation}\label{sho}
\left|\int_0^{u(T,x)}\tau(h_{\chi(t)}y)(\chi'(t)-
e^{-2s})\,dt+e^{-3s}rT^2\right|<\epsilon^{5/4}.
\end{equation}
Let us notice {that} we consider the function $\chi(t)=e^{-2s}t-e^{-3s}rt^2$ on the interval $[0,{\rm O}(r^{1/2})]$ and {that} its maximal value is obtained at $e^s/2r$. It follows that $\chi$ on the interval under consideration  is invertible and we denote the inverse function by $\chi^{-1}(w)=t$. Set $m(w):=1-\frac{e^{-2s}}{\chi'(\chi^{-1}(w))}$.
Moreover, notice that $m(w)=m(\chi(t))=\frac{-2e^{-3s}rt}{e^{-2s}-2e^{-3s}rt}$.
By a direct computation, we verify the following properties of $m$:
\begin{enumerate}
\item[(\textbf{z1})] $|wm(w)|= {\rm O}(1)$, for every $\chi^{-1}(w)\in[0,{\rm O}(|r|^{-1/2})]$; indeed, $wm(w)=\chi(t)\frac{-2e^{-3s}rt}{e^{-2s}-2e^{-3s}rt}={\rm O}(rt^2)={\rm O}(1)$.
\item[(\textbf{z2})] $|m(N_\epsilon)|<\delta^{1/2}$; indeed, $N_\epsilon\in[
\chi(e^{2s}N_\epsilon-1),\chi(e^{2s}N_\epsilon+1)]$
\item[(\textbf{z3})] $|m'(w)|>0$ for every $\chi^{-1}(w)\in[0,{\rm O}(|r|^{-1/2})]$; indeed, $m'(w)=m'(\chi(t)\chi'(t)$ and both terms in the product are of constant sign.
\end{enumerate}

Changing variable ($w=\chi(t)$) and then integrating by parts, we have
\begin{multline}\label{mult:a}
\int_0^{u(T,x)}\tau(h_{\chi(t)}y)(\chi'(t)-e^{-2s})\,dt=
\int_0^{\chi(u(T,x))}\tau(h_wy)m(w)\,dw=\\
 m(\chi(u(T,x)))\int_0^{\chi(u(T,x))}\tau(h_wy)\,dw-   \int_0^{\chi(u(T,x))}
\left(\int_0^w\tau(h_\theta y)\,d\theta\right)m'(w)\,dw.
 \end{multline}

Therefore, by \eqref{bet} and (\textbf{z1}), we have
 \begin{equation}\label{mult:b}
 \left|m(\chi(u(T,x)))\int_0^{\chi(u(T,x))}
 \tau(h_wy)\,dw-m(\chi(u(T,x)))\chi(u(T,x))\right| ={\rm O}(\epsilon^3).
\end{equation}
Moreover,
\be\label{mult:c}
\int_0^{\chi(u(T,x))}
 \left(\int_0^w\tau(h_\theta y)\,d\theta\right)m'(w)dw=\int_0^{\chi(u(T,x))}(w+p(w))m'(w)\,dw,
\ee
where, by \eqref{bet},  $|p(w)|<\epsilon^2 w$ for $w\geq N_\epsilon$, and
$|p(w)|={\rm O}(N_\epsilon)$ for $w\in[0,N_\epsilon]$.
Notice that by integrating by parts, we have
\begin{equation}\label{eq:asc}
\int_0^{\chi(u(T,x))}wm'(w)\,dw= m(\chi(u(T,x)))\chi(u(T,x))- \int_0^{\chi(u(T,x))}m(w)\,dw.
\end{equation}
But changing variables, $w=\chi(t)$, we obtain %{COULD YOU CORRECT TEX BELOW?}
\begin{multline}\label{eq:asc2}
\int_0^{\chi(u(T,x))}m(w)\,dw =\int_{0}^{u(T,x)}(\chi'(t)-e^{-2s})\,dt = \\ \chi(u(T,x))-e^{-2s}u(T,x)=-e^{-3s}r(u(T,x))^2.
\end{multline}
In particular, from~\eqref{eq:asc}, \eqref{eq:asc2} and~(\textbf{z1}), we obtain %{COULD YOU CORRECT TEX BELOW?}
\be\label{eq:ascz}
\int_0^{\chi(u(T,x))}wm'(w)\,dw= {\rm O}(1).\ee
By (\textbf{z3}), (\textbf{z2}), \eqref{bet2} and~\eqref{eq:ascz}, we get
\begin{multline}\label{mult:d}
\left|\int_0^{\chi(u(T,x))}p(w)m'(w)\,dw\right|\leq \\ \left|
\int_0^{N_\epsilon}p(w)m'(w)\,dw\right|+\left|
\int_{N_\epsilon}^{\chi(u(T,x))}p(w)m'(w)\,dw\right|\leq\\
 {\rm O}(N_\epsilon)m(N_\epsilon)+\epsilon^2 \left|\int_0^{\chi(u(T,x))}wm'(w)\,dw\right|={\rm O}(\epsilon^2).
\end{multline}

Finally, in view of \eqref{mult:a}, \eqref{mult:b}, \eqref{mult:c}, \eqref{eq:asc}, \eqref{eq:asc2} and \eqref{mult:d}, we  get
$$
\left|\int_0^{u(T,x)}\tau(h_{\chi(t)}y)(\chi'(t)-e^{-2s})\,dt+
e^{-3s}r(u(T,x))^2\right|= {\rm O}(\epsilon^{2}).
$$
Hence \eqref{sho} follows by the triangle inequality since by \eqref{bet2} and $T\leq Kr^{-1/2}$, we have
$$
\left|e^{-3s}r(u(T,x))^2-e^{-3s}rT^2\right|\leq e^{-3s}|r|(u(T,x)^2-T^2)\leq \epsilon^2e^{-3s}rT^2\leq \epsilon^{3/2}.$$
This finishes the proof of \eqref{axt}.

For the proof of \eqref{distgt},
notice that $\tilde{h}_T(x)=h_{u(T,x)}(x)$ and, by~\eqref{funce}, $\tilde{h}_{\chi(T)+A_x(T)}y=h_{u(\chi(T)+A_x(T),y)}y=
h_{\chi(u(T,x))}y$. The statement now follows by~\eqref{clos2} and~\eqref{eq:newform}, where $T$ is replaced by  $u(T,x)$.
\end{proof}

%Theorem \ref{thm:hor} has the following important corollary:

%\begin{corollary}\label{cor.main}
%\begin{equation}\label{dmdist}
%d_M\left(\tilde{h}_Tx,\tilde{h}_{g(T)+A_x(T)}y\right)\leq \epsilon.
%\end{equation}
%Moreover the $d_g(\tilde{h}_Tx,\tilde{h}_{g(T)+A_x(T)}y)\leq s+rf(x,T)\leq s+CrT$  and $d_-(\tilde{h}_Tx,\tilde{h}_{g(T)+A_x(T)}y)\leq s+rf(x,T)=r$.
%\end{corollary}

\section{Ergodic averages for horocycle flows}\label{app:BF}
The following {result (Lemma~\ref{lem:BF} below)} on  ergodic averages for horocycle flows follows from the work of Flaminio-Forni \cite{Fl-Fo} and Bufetov-Forni~\cite{BF}. We are indebted to Giovanni Forni for explaining it to us.

{For any smooth real-valued function $f \in W^r(M)$, with $r>11/2$, let
 $$
 f= \sum_{\mu \in \text{spec}(\square)} f_\mu
 $$
 denote the decomposition of $f$ with respect to a splitting of the space $L^2(M)$ into irreducible components, parametrized by the eigenvalues $\mu$ of the
 Casimir operator $\square$ (listed with multiplicities), cf.~Section~\ref{sec:deviations}. That is, $f_{\mu}$ is the projection of $f$ on the irreducible component of Casimir parameter $\mu$.
 Set
 $$
 \mu_f  := \min  \{ \mu \in \text{spec}(\square) \setminus \{0\}{:}\:  f_\mu \not =0\}.
 $$}

 \begin{lemma}[ergodic averages estimates]\label{lem:BF} For any real-valued function $\tau \in W^r(M)$ (for any $r>11/2$) with non-trivial support on the irreducible components of the
 complementary series there {exist} a function $\phi_\tau\in C^\infty(\R)$, with $\phi_\tau(0)=0$ and $\phi'_\tau (0) \not =0$,
 and a function $\beta_\tau \in C(M)$ such that the following holds. There {exist}  constants $C_r>0$,   $\alpha_\tau \in (0,1)$  and
 $\gamma_\tau \in (0, \alpha_\tau)$ such that, for every $(x,T)\in M\times \R$, we have
 $$
  \vert  T^{ - \alpha_\tau}  \int_0^T   (\tau-\tau\circ g_s) \circ h_t (x) dt - \phi_\tau(s) \beta_\tau (g_{\log T}x)   \vert \leq C_r s  \Vert \tau \Vert_r  T^{-\gamma_\tau}  \,.
 $$
 For any function $\tau \in W^r(M)$ (for any $r>11/2$) with trivial support on the irreducible components of the
 complementary series, and not fully supported on the discrete series, there exist functions  $\beta_\tau \in C^\infty(\R\times \T^\infty , C(M))$ and $\beta^{(1/4)}_\tau\in C(M)$
 with  $\beta_\tau(0, {\bf \theta}, x)  \equiv 0$ and $\frac{d}{ds}\beta_\tau(0, {\bf \theta}, x) \not\equiv 0$, and a vector ${\boldsymbol{v} \in \T^\infty}$,
  such that the following holds.  There {exist}  constants $C_r>0$  and  $\gamma  \in (0, 1/2)$ such that, for every $(x,T)\in M\times \R$, we have
$$
  \begin{aligned}
   \vert  T^{ - \frac{1}{2}}  \int_0^T  & (\tau-\tau\circ g_s) \circ h_t (x) dt -  \beta_\tau (s, {\bf v} \log T, g_{\log T}x) \\ & -  (e^{-s/2} -1) \beta^{(1/4)}_\tau (g_{\log T}x) \log T    \vert \leq C_r s  \Vert \tau \Vert_r  T^{-\gamma}  \,.
  \end{aligned}
 $$
 The function $\beta^{(1/4)}_\tau$ vanishes identically  if and only if the projection {$\tau_{1/4}$} of the function $\tau \in W^r(M)$ on the irreducible {component} of Casimir
 parameter $1/4$ is a coboundary of the horocycle flow.
 \end{lemma}
 \begin{proof}
We distinguish {three cases:  $\mu_f <1/4$, $\mu_f >1/4$ and $\mu_f=1/4$}.  In the first case, $\mu_f<1/4$,  let $H_1, \dots, H_k$ denote all the irreducible components
 of Casimir parameters $\mu_1 = \dots =\mu_k = \mu_f$. Let $D^\pm_1, \dots, D^\pm_k$ denote the basis {of  distributional} eigenvectors of the geodesic flow of the space of invariant distributions for the horocycle flow supported on ${\mathcal D}'(H_1) \oplus \dots \oplus  {\mathcal D}'(H_k)$. Let $\beta^\pm_1, \dots , \beta^\pm_k$ be the corresponding cocycles for the horocycle flow. In this case, the components $H_1, \dots, H_k$ belong to the  complementary series and by~\cite{BF}, Lemma~3.1, the following holds.  Let $\nu_f := (1-4\mu_f)^{1/2}\in \R^+$. There exist constants $C_r>0$ and $\gamma_f \in (0,\frac{1+\nu_f}{2})$ such that,
 for all $(x,T) \in  M\times \R$ we have
 \begin{equation}
 \label{eq:asympt_1}
 \vert  T^{ - \frac{1+\nu_f}{2}}  \int_0^T   f \circ h_t (x) dt -  \text{\rm Re} \sum_{i=1} ^k D^-_i (f) \beta^-_i( g_{\log T}x, 1)  \vert \leq  C_r \Vert f \Vert_r T^{-\gamma_f} \,;
\end{equation}
 in the second case, $\mu_f>1/4$, the function $f$ has no components on the complementary series and no component corresponding to the Casimir eigenvalue
 $\mu=1/4$.  Let then {$\left(\mu_n\right)_{n\in \N}$} denote the sequence of Casimir parameters in the interval $(1/4, +\infty)$ (listed with multiplicities), let {$\left(D^\pm_{\mu_n}\right)$} denote the sequence of normalized horocycle invariant distributions and let {$\left(\beta^\pm_{\mu_n}\right)$} denote the corresponding sequence of additive H\"older cocycles.
 By~\cite{BF}, Lemma 3.2,  the following holds:

 For every $n\in \N$, let $v_n = (4\mu_n -1)^{1/2} \in \R^+$. There exist constants $C_r>0$ and  $\gamma \in (0,1/2)$  such that, for all $(x,T) \in  M\times \R$,
 \begin{equation}
 \label{eq:asympt_2}
 \begin{aligned}
 \vert  T^{ - 1/2}  \int_0^T   f \circ h_t (x) dt - &\text{\rm Re} \sum_{n\in \N}  D^+_{\mu_n}  (f) \exp(i \frac{v_n}{2}  \log T)  \beta^+_{\mu_n} ( g_{\log T}x, 1)  \vert \\ &\leq
 C_r \Vert f \Vert_r  T^{-\gamma}.
 \end{aligned}
\end{equation}

Finally, in the case $\mu_f=1/4$, we have to add to the above expansion for the case $\mu>1/4$ the following contribution  of the irreducible
 components with Casimir parameter $\mu=1/4$.  Let $H_{1/4, 1}, \dots, H_{1/4, l}$ denote all the irreducible components
 of Casimir parameters $\mu_1 = \dots =\mu_l = 1/4$. Let $D^\pm_{1/4, 1}, \dots, D^\pm_{1/4, l}$ denote the basis {of  distributional} eigenvectors of the geodesic flow of the space of invariant distributions for the horocycle flow supported on ${\mathcal D}'(H_{1/4, 1}) \oplus \dots \oplus  {\mathcal D}'(H_{1/4, l})$. Let $\beta^\pm_{1/4, 1}, \dots ,
 \beta^\pm_{1/4, l}$ be the corresponding cocycles for the horocycle flow.  By  the exact scaling and asymptotic results of \cite{BF}, Theorem~1.2 and Corollary~1.3, or Corollary~3.2,  it follows that for the projection $f_{1/4}$ of $f \in W^r(M)$ onto the component
 $H_{1/4, 1} \oplus \dots \oplus H_{1/4, l}$ of the Hilbert space $W^r(M)$, the following holds.  There exist constants $C_r>0$ and  $\gamma \in (0,1/2)$  such that,
 for all $(x,T) \in  M\times \R$,
 \begin{equation}
 \begin{aligned}
 \label{eq:asympt_3}
 \vert  &T^{ - 1/2}  \int_0^T   f_{1/4} \circ h_t (x) dt -  \sum_{i=1}^l  D^+_{1/4,i} (f_{1/4}) \beta^+_{1/4,i}( g_{\log T}x, 1)  \\& +   \sum_{i=1}^l ( D^-_{1/4,i} (f_{1/4}) - \frac{\log T}{2} D^+_{1/4,i} (f_{1/4}) ) \beta^-_{1/4,i} ( g_{\log T}x, 1)  \vert \leq  C_r \Vert f \Vert_r T^{-\gamma} \,.
 \end{aligned}
\end{equation}
We then apply the above asymptotic formulas to the functions $f=  \tau - \tau\circ g_s$.  Since the action of the geodesic flow $g_\R$
preserves the splitting of Sobolev spaces into irreducible components, hence in particular, for all $s \in \R$, {we have}
$$
f_\mu = \tau_\mu  - \tau_\mu \circ g_s\,, \quad \text{ for all } \mu \in \text{ spec} (\square)\,.
$$
It remains to compute the values of the invariant distributions $D^\pm_\mu\in {\mathcal D}'(H_\mu)$  on the function $\tau - \tau\circ g_s$ for any irreducible
component $H_\mu$ of Casimir parameter $\mu>1/4$.  {For $\mu \not = 1/4$, it} follows from \cite{Fl-Fo}, Theorem~3.2 and
Lemma~3.5 (see also \cite{BF}, formulas~(5) and~(6)) that we have
$$
D^\pm_\mu (\tau \circ g_s) = [(g_s)_* (D^\pm_\mu )] (\tau) = \exp ( - \frac{1 \pm \sqrt{ 1-4\mu} }{2} s)  D^\pm_\mu (\tau) \,,
$$
and, for $\mu = 1/4$, {we have}
$$
\begin{aligned}
D^+_{1/4} (\tau \circ g_s)  &=  [(g_s)_* (D^+_{1/4} )] (\tau) = e^{ - \frac{s}{2} }  D^+_{1/4} (\tau)  \,, \\
D^-_{1/4} (\tau \circ g_s)  &=  [(g_s)_* (D^+_{1/4} )] (\tau) = e^{ - \frac{s}{2} }   [ D^-_{1/4} (\tau) -\frac{s}{2} D^+_{1/4}(\tau)]   \,.
\end{aligned}
$$
Since for every $r>0$  there exists a constant $C_r>0$ such that
$$
\Vert  \tau \circ g_s - \tau  \Vert  \leq  C_r s \Vert \tau \Vert_r\,,
$$
we can conclude the argument as follows. If $\mu_\tau <1/4$, we set
$$
\begin{aligned}
\phi_\tau (s) &:=  \exp ( - \frac{1 - \sqrt{ 1-4\mu_\tau} }{2} s) - 1 \,, \quad \text{ for all } s \in \R\,, \\
\beta_\tau (x) &:=  (1 \pm \sqrt{ 1-4\mu_\tau} ) \text{\rm Re} \sum_{i=1}^k D^-_i(\tau) \beta^-_i (x)\,,  \quad \text{ for all }  x \in M\,,
\end{aligned}
$$
so that from formula \eqref{eq:asympt_1} we derive that
 $$
  \vert  T^{ - \frac{1+\nu_f}{2}}  \int_0^T   (\tau-\tau\circ g_s) \circ h_t (x) dt - \phi_\tau(s) \beta_\tau (g_{\log T}x)   \vert \leq C_r s  \Vert \tau \Vert_r  T^{-\gamma}  \,.
 $$
 {If} $\mu_\tau > 1/4$,  then we set ${\bf v} = ( \frac{v_n}{2})$ and, for all $(s,\theta, x)\in \R \times \T^\infty \times M$, {we have}
 $$
 \beta_\tau (s, \theta, x) =  \text{\rm Re} \sum_{n\in \N}
 [\exp ( - \frac{1 - \sqrt{ 1-4\mu_n} }{2} s) - 1]  D^+_{\mu_n}  (\tau) \exp(i \frac{v_n}{2}  \theta)  \beta^+_{\mu_n} (x, 1)
 $$
 so that, from formula \eqref{eq:asympt_2}, we derive that
 $$
   \vert  T^{ - \frac{1}{2}}  \int_0^T   (\tau-\tau\circ g_s) \circ h_t (x) dt -  \beta_\tau (s, {\bf v} \log T, g_{\log T}x)   \vert \leq C_r s  \Vert \tau \Vert_r  T^{-\gamma}  \,.
 $$

If $\mu_\tau =1/4$, then we also set
 $$
 \beta^{(1/4)}_\tau (x)  :=   \sum_{i=1}^l  D^+ _{1/4, i} (\tau) \beta^+_{1/4, i} (x, 1)\,,  \quad \text{ for all } x\in M\,,
 $$
 so that, by formulas~\eqref{eq:asympt_2} and \eqref{eq:asympt_3}, we derive
 $$
  \begin{aligned}
   \vert  T^{ - \frac{1}{2}}  \int_0^T   (\tau-\tau\circ g_s) \circ h_t (x) dt &-  \beta_\tau (s, {\bf v} \log T, g_{\log T}x) \\ & -  (e^{-s/2} -1) \beta^{(1/4)}_\tau (g_{\log T}x) \log T    \vert \leq C_r s  \Vert \tau \Vert_r  T^{-\gamma}  \,.
  \end{aligned}
$$
\end{proof}

\section{Strong MOMO and USIC properties}\label{s:Sec10}
We now consider a bounded arithmetic functions $\bfu:\N\to\C$ and a {topological} dynamical {system} $(X,T)$  (i.e.\ $X$ is a compact metric space and $T$ is a homeomorphism of $X$).

The following notion has been introduced in \cite{Ab-Ku-Le-Ru}. MOMO is an acronym for M\"obius Orthogonality of Moving Orbits.
\begin{definition}[strong MOMO]\label{def1}\em $(X,T)$ satisfies the {\em strong $\bfu$-MOMO
 property} if
for all $(b_k)\subset\N$ with $b_{k+1}-b_k\to \infty$, $(x_k)\subset X$ and $f\in C(X)$, we have
$$
\lim_{K\to\infty}\frac1{b_K}\sum_{k<K}\left|\sum_{b_k\leq n<b_{k+1}}f(T^nx_k)\bfu(n)\right|=0.$$
\end{definition}

\begin{proposition} \label{p1} $(X,T)$ satisfies strong $\bfu$-MOMO property if and only if it satisfies strong $\bfu$-USIC property.\end{proposition}
\begin{proof} To prove one implication, suppose that $(X,T)$ does not satisfy $\bfu$-USIC. Then according to~\eqref{e1} there are: $\vep_0>0$, $(M_\ell),(H_\ell)$ tending to~$\infty$ with $H_\ell/M_\ell\to0$, $f\in C(X)$  such that for a certain subsequence $(\ell_k)$ which we still denote by $(\ell)$, we can find $x_\ell\in X$ such  that
\be\label{e2}
\frac1{M_\ell}\sum_{M_{\ell}\leq m<2M_{\ell}}\left|\frac1{H_{\ell}}\sum_{m\leq h<m+H_{\ell}}f(T^hx_\ell)\bfu(h)\right|\geq\vep_0.\ee
We now proceed as in the proof of Theorem~5 {in}~\cite{Ab-Le-Ru}.
We have
$$
\frac1{H_\ell}\sum_{r=0}^{H_\ell-1}
\frac1{M_\ell/H_\ell}\sum_{M_{\ell}\leq m<2M_{\ell}, m=r\;{\rm mod}\;{H_\ell}}\left|\frac1{H_{\ell}}\sum_{m\leq h<m+H_{\ell}}f(T^hx_\ell)\bfu(h)\right|\geq\vep_0.$$
Hence, for each $\ell$, we can choose $0\leq r_\ell<H_\ell$ such that
\be\label{e3}
\frac1{M_\ell/H_\ell}\sum_{M_{\ell}\leq m<2M_{\ell}, m=r_\ell\;{\rm mod}\;H_\ell}\left|\frac1{H_{\ell}}\sum_{m\leq h<m+H_{\ell}}f(T^hx_\ell)\bfu(h)\right|\geq\vep_0.\ee
By passing to a subsequence if necessary, we can assume that $M_{\ell+1}>2M_\ell+H_\ell$. The sequence $(b_k)$ is defined as
$$
\{b_1<b_2<\ldots\}:=\bigcup_\ell\{M_\ell\leq m<2{M_\ell}+H_\ell,\;m=r_\ell\;{\rm mod}\;H_\ell\}$$
(note that $b_{k+1}-b_k\geq H_\ell\to\infty$).
Moreover, for all $k$ such that $b_k\in [M_\ell,2M_\ell+H_\ell]$, we set $x_k=x_\ell$.
We let $K_\ell$ be the largest $k$ so that $b_k<2M_\ell$. Now, $b_{K_\ell}/{(2M_\ell)}\to1$ (as $H_\ell/M_\ell\to0$) and using~\eqref{e3}, we obtain
$$
\liminf_{\ell\to\infty}\frac1{b_{K_\ell}}\sum_{k<K_\ell}\left|
\sum_{b_k\leq n<b_{k+1}}
f(T^nx_k)\bfu(n)\right|\geq$$
$$
\liminf_{\ell\to\infty}\frac1{2M_\ell}
\sum_{M_{\ell}\leq m<2M_{\ell}, m=r_\ell\;{\rm mod}\;H_\ell}\left|\frac1{H_{\ell}}\sum_{m\leq h<m+H_{\ell}}f(T^hx_\ell)\bfu(h)\right|\geq\vep_0/2.$$
We omit what happens between $b_{K_\ell}$ and $b_{K_\ell+1}$ as we are interested in an estimate from below.

\smallskip

To prove the other implication, we will need a lemma which has been shown in \cite{Ab-Ku-Le-Ru} (see Lemma~24 and its proof therein):

\begin{lemma}\label{l1}
Let $f\in C(X)$, $(x_k)\subset X$, $(b_k)\subset\N$, $b_{k+1}-b_k\to\infty$. Assume that
$$
\limsup_{K\to\infty}\frac1{b_{K+1}}\sum_{k\leq K}\left|\sum_{b_k\leq n<b_{k+1}}f(T^nx_k)\bfu(n)\right|>0.
$$
Then there exist $\delta_0>0$ and a subsequence $(k_\ell)$ such that { \Red the upper density}
\begin{equation*}
{\Red  \eta:=\overline{d}\left(\bigcup_{\ell\geq 1}[b_{k_\ell},b_{k_\ell+1})\right)>0}
\end{equation*}
and {\Blue for each $\ell\geq1$ we have}
\be\label{badint}
\frac1{b_{k_\ell+1}-b_{k_\ell}}
\left|\sum_{n=b_{k_\ell}}^{b_{k_\ell+1}-1}
f(T^n{x_{k_\ell}})\bfu(n)\right|\geq\delta_0.%\text{ for each }\ell\geq1.
\ee
\end{lemma}

We now proceed, as before, by contradiction. That is, we {suppose} that the assumption of Lemma~\ref{l1} is satisfied. We need to select $(M_\ell)$ and $(H_\ell)$. For that we use the assertion of that lemma. First of all $(M_\ell)$ will be chosen so that the  upper density $\eta>0$
of the set $\bigcup_{r\geq 1}[b_{k_r},b_{k_r+1})$ is ``realized'', that is, $|[1,M_\ell]\cap \bigcup_{r\geq 1}[b_{k_r},b_{k_r+1})|\geq\frac\eta2M_\ell$ {for each $\ell\geq1$}. Moreover, due to the assumption $b_{k+1}-b_k\to\infty$, we can assume {(by passing to a subsequence of $(M_\ell)$ if necessary)} that the union of intervals $[b_{k_r},b_{k_r+1})$ whose length is short, say $\leq \ell$, has cardinality smaller than or equal to $ M_\ell/\ell$. Finally, set $H_\ell=\sqrt{\ell}$. Clearly $H_\ell/M_\ell\to0$, and according to the strong $\bfu$-USIC property, most of the $m$ in the interval $[1,M_\ell]$ is ``good'' in the sense that the relevant sums along $[m,m+H_\ell)$ are small, {i.e., given $\vep>0$, for $\ell$ large enough, we have}
$$
\frac1{H_\ell}\sup_{x\in X}\left|\sum_{m\leq h<m+H_\ell} f(T^hx)\bfu(h)\right|<\vep$$
{for a $(1-\vep)$-proportion of $m\in[1,M_\ell]$}.
We can make such ``good'' intervals of length $H_\ell$ disjoint by  considering $m=s\;{\rm mod}\;H_\ell$ and then choosing {$0\leq s_\ell<H_\ell$} so that most of the intervals $[m,m+H_\ell)$ with $m=s_\ell\;{\rm mod}\;H_\ell$ are ``good'' (as in the proof of the necessity). Now, such ``good'' intervals will cover (with a small error) many of our ``bad'', i.e.\ satisfying~\eqref{badint},
intervals $[b_{k_r},b_{k_r+1})$ (as the (upper) density of the union of such intervals is fixed, equal to $\eta$). However, if $[m,m+H_\ell)$ is good then $|\sum_{m\leq h<m+H_\ell}f(T^hy)\bfu(h)|={\rm o}(H_\ell)$ (with ``${\rm o}$'' which does not depend on $m$) regardless $y\in X$. It follows that we must have $|\sum_{b_{k_r}\leq n <b_{k_r+1}}f(T^nx_{k_r})\bfu(n)|={\rm o}(b_{k_r+1}-b_{k_r})$, a contradiction.
\end{proof}

\subsection*{Acknowledegements}

The original motivation to start this project came from a question that Marina Ratner asked during the conference  "{\it Homogeneous Dynamics,
Unipotent Flows, and Applications}" held in October 2013 at the IIAS in Jerusalem (and appeared as Question~7 in the survey \cite{Fe-Ku-Le}). We are  thankful to her for asking this question in particular, and for her seminal and inspirational work more generally. We are also grateful to the IIAS and to the conference organizers for making the event possible.
% acknowledge the organizers of the conference and the IIAS for the 
%hospitality. 

We are indebted to Giovanni Forni for explaining  to us  Lemma~\ref{lem:BF} and how to deduce it from \cite{BF}. 
We would  also like to thank him and  Livio Flaminio for useful discussions.

M.L. is supported by an by an NCN grant. C.U. is supported by the ERC Starting Grant ChaParDyn, as well as by the Leverhulme Trust through a Leverhulme Prize and by the Royal Society
through a Wolfson Research Merit Award. The research leading to these
results has received funding from the European Research Council under
the European Union Seventh Framework Programme (FP/2007-2013) / ERC
Grant Agreement n. 335989.


\begin{thebibliography}{99}

\bibitem{Ab-Le-Ru}
{\sc E.~H. El~Abdalaoui, M.~Lema\'nczyk, and T.~de~la Rue}, {\em Automorphisms with {Q}uasi-discrete {S}pectrum, {M}ultiplicative
  {F}unctions and {A}verage {O}rthogonality {A}long {S}hort {I}ntervals}, Int.
  Math. Res. Notices 14 (2017), 4350--4368.

\bibitem{Ab-Ku-Le-Ru}
{{\sc E.~H. El~Abdalaoui, J.~Ku\l{}aga-Przymus},
  M.~Lema{\'n}czyk, and T.~de~la Rue,}
{\em M\"obius disjointness for models of an ergodic system and beyond}, {Israel J.\ Math., published online (2018),}  arXiv:1704.03506.


\bibitem{Ar}
{\sc V.~I. Arnold}, {\em Topological and ergodic properties of closed 1-forms
  with incommensurable periods.}, Funktsional'nyi Analiz i Ego Prilozheniya, 25
  (1991), pp.~1--12.
\newblock (Translated in: \emph{Functional Analyses and its Applications},
  25:2:81--90, 1991).
	

%\bibitem{Ab-Ku-Le-Ru}H. El Abdalauoi, J.\ Ku\l aga-Przymus, M.\ Lema\'nczyk, T.\ de la Rue, {\em and beyond}

%  {\em The {C}howla and the {S}arnak
%  conjectures from ergodic theory point of view}, Discrete Contin. %Dyn. Syst.
%  \textbf{37} (2017), no.~6, 2899--2944.
%\bibitem{Ab-Ku-Le-Ru1}
%E.~H. El~Abdalaoui, J.~Ku\l{}aga-Przymus,
%  M.~Lema{\'n}czyk, and T.~de~la Rue,
% {\em M{\"o}bius disjointness for models of an ergodic system and
%  beyond}, preprint, \url{https://arxiv.org/abs/1704.03506}.

\bibitem{Be-Ka} {\sc P.\ Berk, A.\ Kanigowski}, {\em Self-similarity problem for some surface flows}, preprint.	


\bibitem{Bo-Sa-Zi}
{\sc J.~Bourgain, P.~Sarnak, and T.~Ziegler}, {\em Disjointness of {M}\"obius from
  horocycle flows}, From {F}ourier analysis and number theory to {R}adon
  transforms and geometry, Dev. Math., vol.~28, Springer, New York, 2013,
  pp.~67--83.
	
\bibitem{BF} {\sc A.\ Bufetov, G.\ Forni}, {\em Limit Theorems for Horocycle Flows}, Annales scientifiques de l'ENS 47 (2014), 851-903.

\bibitem{Co-Fo-Si}
{\sc I.~P. Cornfeld, S.~V. Fomin, and Y.~G. Sinai}, {\em Ergodic Theory},
  Springer-Verlag, 1980.

 \bibitem{Do-Ka} {\sc C.\ Dong, A.\ Kanigowski}, {\em Rigidity of von Neumann special flows with one discontinuity}, preprint.

\bibitem{Do-Se}{\sc T.\ Downarowicz, J.\ Serafin}, {\em Almost full entropy subshifts uncorrelated to the Möbius function}, {to appear in Int.\ Math. Res. Notices},  https://arxiv.org/abs/1611.02084.
\bibitem{Fa-Ka}{\sc B.~Fayad and A.~Kanigowski}, {\em Multiple mixing for a class of
  conservative surface flows}, Inventiones Mathematicae,  (2015), pp.~1--60.

\bibitem{Fe-Ku-Le} {\sc S.\ Ferenczi, J.\ Ku\l aga-Przymus, M.\ Lema\'nczyk}, {\em Sarnak's conjecture: what's new}, in:
Ergodic Theory and Dynamical Systems in their Interactions with Arithmetics and Combinatorics,  CIRM Jean-Morlet Chair, Fall 2016,
Editors: S. Ferenczi, J. Ku\l aga-Przymus, M. Lema\'nczyk,
    Lecture Notes in Mathematics 2213,  Springer International Publishing, pp. 418.

\bibitem{Fl-Fo} {\sc L.\ Flaminio, G.\ Forni}, {\em Invariant distributions and time averages for horocycle
flows}, Duke Math. J. 119, 2003, no. 3, 465-526.


\bibitem{Fl-Fo-note} {\sc L.\ Flaminio, G.\ Forni}, {\em Orthogonal powers and M{\"o}bius conjecture for smooth time-changes of horocycle flows}, 
%ORTHOGONAL POWERS AND MÖBIUS CONJECTURE FOR SMOOTH TIME CHANGES OF HOROCYCLE FLOWS}, 
preprint.

\bibitem{ForK} {\sc G.\ Forni, A. Kanigowski}, {\em Multiple mixing and rigidity for reparametrizations of Heisenberg nilflows}, preprint.

\bibitem{Fr-Le}
{\sc K.~Fr{\c a}czek and M.~Lema{\'{n}}czyk}, {\em On mild mixing of special
  flows over irrational rotations under piecewise smooth functions}, Ergodic
  Theory and Dynamical Systems, 26 (2006), pp.~719--738.

\bibitem{FL2}
{\sc K.~Fr{\c a}czek and M.~Lema{\'{n}}czyk}, {\em Ratner's property
  and mild mixing for special flows over two-dimensional rotations}, J. Mod.
  Dyn., 4 (2010), pp.~609--635.
	


\bibitem{Fu} {\sc H. Furstenberg}, {\em Recurrence in ergodic theory and combinatorial number theory},
  Princeton University Press, Princeton, N.J., 1981, M. B. Porter Lectures.


\bibitem{Ga-Le-Sch} {\sc P.\ Gabriel, M.\ Lema{\'n}czyk, K.\ Schmidt}, {\em Extensions of cocycles for hyperfinite actions and applications}, Monatshefte Math. 123 (1997), 209-228.

\bibitem{Gl} {\sc E.\ Glasner}, {\em Ergodic Theory via Joinings}, vol. 101 of Mathematical Surveys and Monographs,
American Mathematical Society, Providence, RI, 2003.

\bibitem{Ka-Ku} {\sc A.\ Kanigowski, J.\ Ku\l aga-Przymus}, {\em Ratner's property and mild mixing for smooth flows on surfaces}, Ergodic Theory Dynam.\ Systems (2015), FirstView.

%\bibitem{Ka-Ku-Ul} A. Kanigowski, J. Ku\l aga-Przymus, C. Ulcigrai
\bibitem{Ka-Ku-Ul} {\sc A.\ Kanigowski, J.\ Ku\l aga-Przymus, C.\ Ulcigrai}, {\em Multiple mixing and parabolic divergence in smooth area-preserving
flows on higher genus surfaces},  to appear in JEMS.

\bibitem{Ka} {\sc I.\ Kat{\'a}i}, {\em A remark on a theorem of H.\ Daboussi}, Acta Math. Hungar., 47 (1986), 223-225.

\bibitem{Kh} {\sc Y.\ Khintchin}, {\em Continued Fractions}, Chicago Univ. Press 1960

\bibitem{Ko} {\sc Andrey V. Kochergin}. Non-degenerate fixed points and mixing in
flows on a 2-torus. Matematicheskii Sbornik, 194(8):83--112, 2003.
(Translated in: Sb. Math., 194(8):1195-1224).

\bibitem{Ku-Ni} {\sc L.\ Kuipers, H.\ Niederreiten}, {\em Uniform distribution of sequences}, Wiley, London 1975.
%\bibitem{Ku} J.\ Ku\l aga, Fund. Math.

%\bibitem{Ma-Ra} K. Matom\"aki, M.\ Radzi\l\l
\bibitem{Ma-Ra}
{\sc K.~Matom\"aki and M.~Radziwi{\l}{\l}}, {\em Multiplicative functions in short
  intervals}, Ann. of Math. (2) \textbf{183} (2016), no.~3, 1015--1056.
	
%\bibitem{Ma-Ra-Ta}K. Matom\"aki, M.\ Radzi\l\l and T. Tao
\bibitem{Ma-Ra-Ta} %\marginpar{check if this}
{\sc K.~Matom{\"a}ki, M.~Radziwi{\l}{\l}, and T.~Tao}, {\em An averaged form of
  {C}howla's conjecture}, Algebra \& Number Theory \textbf{9} (2015),
  2167--2196.
%\bibitem{Ma-Ra-Ta1}
%K.~Matom{\"a}ki, M.~Radziwi{\l}{\l}, and T.~Tao, {\em Sign patterns of the {L}iouville and {M}\"obius functions},
%  Forum Math. Sigma \textbf{4} (2016), e14, 44.
	
\bibitem{Ma}  %\marginpar{(check if this and it's the first...)}
%\bibitem{Ma}
{\sc B.~Marcus}, {\em The horocycle flow is mixing of all degrees}, Invent. Math.
  \textbf{46} (1978), no.~3, 201--209.
	\bibitem{No}
{\sc S.~P. Novikov}, {\em The {H}amiltonian formalism and a multivalued
  analogue of {M}orse theory}, (Russian) Uspekhi Matematicheskikh Nauk, 37
  (1982), pp.~3--49.
\newblock (Traslated in: \emph{Russian Mathematical Surveys}, 37 No 5:1--56,
  1982).

\bibitem{Rat}
{\sc M.~Ratner}, {\em Horocycle flows, joinings and rigidity of products}, Ann.
  of Math. (2), 118 (1983), pp.~277--313.	

\bibitem{Rat2}{\sc M.~Ratner}, {\em Rigid reparametrizations and cohomology for horocycle flows}, Invent. Math,Vol. 88, (2) 1987, pp.~341--374.

	
\bibitem{Ra}
{\sc D.~Ravotti}, {\em Quantitative mixing for locally hamiltonian flows with
  saddle loops on compact surfaces}.
\newblock Preprint arXiv:1610.08743, 2016.


\bibitem{Ryz-Tho} {\sc V.~V. Ryzhikov and J.-P. Thouvenot}, {\em Disjointness, divisibility, and
  quasi-simplicity of measure-preserving actions}, Funktsional. Anal. i
  Prilozhen., 40 (2006), 85-89.

		

\bibitem{Sa}
%\bibitem{Sa-CIRM}
%P.~Sarnak, {\em {M}\"obius randomness and dynamics six years later},
%  \url{http://www.youtube.com/watch?v=LXX0ntxrkb0}.
%\bibitem{Sa}
{\sc P.~Sarnak}, {\em Three lectures on the {M}{\"o}bius function, randomness and
  dynamics}, \url{http://publications.ias.edu/sarnak/}.
%\bibitem{Sa:Af}
%P.~Sarnak, {\em Mobius randomness and dynamics}, Not. S. Afr. Math. Soc.
%  \textbf{43} (2012), no.~2, 89--97.

\bibitem{SK}
{\sc Y.~G. Sinai and K.~M. Khanin}, {\em Mixing for some classes of special
  flows over rotations of the circle.}, Funktsional'nyi Analiz i Ego
  Prilozheniya, 26 (1992), pp.~1--21.
\newblock (Translated in: \emph{Functional Analysis and its Applications},
  26:3:155--169, 1992).


\bibitem{Thou}
{\sc J.-P. Thouvenot}, {\em Some properties and applications of joinings in
  ergodic theory}, in Ergodic theory and its connections with harmonic analysis
  ({A}lexandria, 1993), vol.~205 of London Math. Soc. Lecture Note Ser.,
  Cambridge Univ. Press, Cambridge, 1995, pp.~207--235.



\bibitem{Ul}
{\sc C.~Ulcigrai}, {\em Mixing of asymmetric logarithmic suspension flows over
  interval exchange transformations}, Ergodic Theory Dynam. Systems, 27 (2007),
  pp.~991--1035.


\end{thebibliography}
\end{document}